\documentclass[letterpaper,11pt,reqno]{amsart}
\usepackage{amssymb}
\usepackage{amscd}
\usepackage{amsfonts}
\usepackage{amsmath}
\usepackage{amsthm}
\usepackage[all]{xy}
\usepackage{mathrsfs}
\usepackage{tikz}

\usepackage[paper = letterpaper, left = 30mm, right = 30mm, headsep = 7mm,
footskip = 10mm, top = 30mm, bottom = 30mm, footnotesep=5mm, headheight =
2cm]{geometry}

\usepackage{titletoc}
\titlecontents{part}[0pt]{}{\bfseries\chaptername\ \thecontentslabel\quad}{\bfseries}{\bfseries\hfill}

\usepackage[colorlinks=true]{hyperref}
\usepackage{url}
\usepackage{cite}
\renewcommand\thepart{\Roman{part}}

\def\Amp{\mathrm{Amp}}
\def\adj{\mathrm{adj}}
\def\id{\mathrm{id}}
\def\rank{\mathrm{rank}}

\def\Pic{\mathrm{Pic}}
\def\rk{\mathrm{rank}}
\def\Todd{\mathrm{Todd}}

\def\part#1{%
  \vskip .02\vsize 
  \refstepcounter{part}
  \addcontentsline{toc}{part}{Part \thepart:\ #1}
  {\centering\large \textbf{Part \thepart}. #1\par}%
  \vskip .01\vsize
}

\newcommand{\Conv}{\operatorname{Conv}}
\newcommand{\Ext}{\operatorname{Ext}}
\newcommand{\Hom}{\operatorname{Hom}}
\newcommand{\Pos}{\operatorname{Pos}}

\newcommand{\HN}{H\negthinspace N}

\newcommand*{\sheafHom}{\mathscr{H}\negthinspace om}

\theoremstyle{plain}
\newtheorem{theorem}{Theorem}[section]
\newtheorem*{theorem*}{Theorem}
\newtheorem{lemma}[theorem]{Lemma}
\newtheorem{corollary}[theorem]{Corollary}
\newtheorem{proposition}[theorem]{Proposition}

\numberwithin{equation}{section} 

\makeatletter
\def\ifdraft{\ifdim\overfullrule>\z@
  \expandafter\@firstoftwo\else\expandafter\@secondoftwo\fi}
\makeatother

\newsavebox{\ieeealgbox}

\newcommand{\N}{\mathbb{N}}
\newcommand{\Q}{\mathbb{Q}}
\newcommand{\R}{\mathbb{R}}
\newcommand{\Z}{\mathbb{Z}}

\newcommand{\acts}{\mbox{ \raisebox{0.26ex}{\tiny{$\bullet$}} }}

\def\hq{\hspace{-0.5mm}/\hspace{-0.14cm}/ \hspace{-0.5mm}}

\theoremstyle{definition}
\newtheorem{definition}[theorem]{Definition}
\newtheorem{remark}[theorem]{Remark}
\newtheorem{comment2}[theorem]{Comment}
\newtheorem{example}[theorem]{Example}

\title[Variation of Moduli Spaces via Quiver GIT]{Variation of Gieseker Moduli Spaces via Quiver GIT}
\date{\today}
\author[D.~Greb]{Daniel Greb}
  \address{DG: Essener Seminar f\"ur Algebraische Geometrie und Arithmetik, Fakult\"at f\"ur Mathematik, Universit\"at Duisburg-Essen, 45117 Essen, Germany}
  \email{daniel.greb@ruhr-uni-bochum.de}

\author[J.~Ross]{Julius Ross}
  \address{JR: Department of Pure Mathematics and Mathematical Statistics, University of Cambridge, Wilberforce Road, Cambridge, CB3 0WB, UK}
  \email{j.ross@dpmms.cam.ac.uk}
\author[M.~Toma]{Matei Toma}
\address{MT: Institut de Math\'ematiques \'Elie Cartan, Universit\'e de Lorraine, B.P. 70239, 54506 Vandoeuvre-l\`es-Nancy Cedex,
France}
\email{Matei.Toma@univ-lorraine.fr}
\thanks{}
\begin{document}

\maketitle
\vspace{-0.2cm}

\begin{abstract}
We introduce a notion of stability for sheaves with respect to several polarisations that generalises the usual notion of Gieseker-stability.  We prove, under a boundedness assumption, which we show to hold on threefolds or for rank two sheaves on base manifolds of arbitrary dimension, that semistable sheaves have a projective coarse moduli space that depends on a natural stability parameter.  We then give two 
applications of this machinery. First, we show that given a real ample class $\omega\in N^1(X)_\R$ on a smooth projective threefold $X$ there exists a projective moduli space of sheaves that are Gieseker-semistable with respect to $\omega$. Second, we prove that given any two ample line bundles on $X$ the corresponding Gieseker moduli spaces are related by Thaddeus-flips. 
\end{abstract}

\enlargethispage{0.5cm}

\tableofcontents
\newcommand{\Space}{\,\,\,\,\,\,\,\,\,\,\,\,\,\,\,\,}

\addtocontents{toc}{\protect\setcounter{tocdepth}{0}}
\section*{Introduction}

\addtocontents{toc}{\protect\setcounter{tocdepth}{1}}

Moduli spaces of sheaves play a central role in Algebraic Geometry: they provide intensively studied examples of higher-dimensional varieties, are naturally associated with the underlying space so can be used to define invariants of its differentiable structure, and have found application in numerous problems of mathematical physics. To obtain moduli spaces with nice properties it is necessary to choose a stability condition, which classically depends on a choice of ample class on the underlying space. Thus, along with the general existence problem, it is natural to ask how these moduli spaces vary as this choice changes.

For surfaces there has emerged a rather beautiful answer to this question through the works of Friedman-Qin \cite{FriedmanQin}, Ellingsrud-G\"ottsche \cite{EllingsrudGoettsche}, and Hu-Li \cite{HuLi}, among others.  Suppose $X$ is a smooth projective complex surface, and that we consider torsion-free coherent sheaves on $X$ of a given topological type with large second Chern class. Given a choice of ample class $L$, the moduli space $\mathcal M_{L}^{\mu}$ of slope-semistable sheaves (with respect to $L$) is irreducible and generically smooth. Furthermore, the ample cone $\Amp(X)$ of $X$ is divided up by a locally finite number of \emph{rational linear} walls into chambers within which $\mathcal M_{L}^{\mu}$ does not change, and if $\overline{L}$ lies on one of these walls, and $L_1$ and $L_2$ are points in adjacent chambers, the moduli spaces undergo a birational flip
\begin{equation}\label{eq:flipintro}\tag{$\ast$}
\begin{gathered}
\begin{xymatrix}{
\mathcal M_{L_1}^{\mu} \ar[rd]& & \ar[ld] \mathcal M_{L_2}^{\mu}\\
& \mathcal M_{\overline{L}}^{\mu}.&
}
\end{xymatrix}
\end{gathered}\end{equation}
In particular, any two such moduli spaces are birational, related by a sequence of birational transformations through moduli spaces of sheaves.  There is an analogous picture of the moduli spaces $\mathcal M_L$ of Gieseker-semistable sheaves on a surface, as proved by Matsuki-Wentworth \cite{MatsukiWentworth}, which also relies in a crucial way on the fact that the polarisation $\overline{L}$ lying on the wall is rational.\medskip

Much less is known when $X$ has higher dimension. In fact, it is not hard to see that the same techniques used for surfaces do not generalise, essentially for the following reason: if $\dim X\ge 3$, then the corresponding walls in $\Amp(X)$ that witness the change in (slope) stability are no longer linear (this is easily seen as the slope of a sheaf with respect to $L$ is non-linear if $\dim X\le 3$).   Worse still, there are examples due to Schmitt \cite{Schmitt} in which such a wall may contain no rational points at all! Thus, the natural candidate to replace $\overline{L}$ is a real, but not rational, ample class, and as there is no obvious candidate for the moduli of Gieseker-semistable sheaves taken with respect to $\overline{L}$, and so not much reason to expect a diagram similar to \eqref{eq:flipintro}.\medskip

In this paper, we propose and execute a new strategy that addresses this problem. The main idea is to avoid moving the ample class directly and instead work with a stability notion that depends on a choice of several ample classes at once. We show that on a smooth projective threefold any two Gieseker moduli spaces are related by Thaddeus-flips. As part of the proof, we also prove the existence of a projective moduli space of Gieseker-semistable sheaves taken with respect to any real ample class, thus answering a special case of an old question of Tyurin, cf.~\cite[Sect.~3.2]{Tel08}.
\medskip

\subsection*{Multi-Gieseker Stability}

Our approach rests on the consideration of the following stability condition. 
Let $X$ be a projective manifold and fix a finite collection of ample line bundles $L_j$ on $X$ for $1\le j\le j_0$. Furthermore, suppose that $\sigma=(\sigma_1,\ldots,\sigma_{j_0})$ is a non-zero vector of non-negative real numbers.  We shall say a torsion-free coherent sheaf $E$ on $X$ is \emph{semistable} with respect to this data if for all proper subsheaves $F\subset E$ the inequality 
\begin{equation*}
 \frac{\sum_{j} \sigma_j \chi(F\otimes L_{j}^m)}{\rank(F)} \le  \frac{\sum_{j} \sigma_j \chi(E\otimes L_{j}^m)}{\rank(E)}\label{eq:introstability}
 \end{equation*}
holds for all $m$ sufficiently large.

\begin{theorem*}[Existence of projective moduli spaces, Theorems~\ref{thm:moduliexist} and \ref{thm:modulispaceproper}]
Suppose the set of semistable sheaves (of a given topological type) is bounded. Then, there exists a projective moduli space $\mathcal M_{\sigma}$ of semistable sheaves.
\end{theorem*}
As is clear from the definition, the moduli space $\mathcal M_L$ of Gieseker-semistable sheaves with respect to a single ample line bundle $L$ is a special case of this construction (simply taking $j_0=1$).   Moreover, just as for $\mathcal M_L$, the moduli spaces $\mathcal M_{\sigma}$ contain an open set parameterizing stable sheaves, and the points on the boundary correspond to $S$-equivalence classes of sheaves.  

The boundedness hypothesis is obviously necessary for such a moduli space to exist, and we will prove it holds in a number of cases, to be discussed next.  In fact, our main construction holds more generally and, subject to the same boundedness hypothesis, gives a moduli space of pure sheaves on any projective scheme $X$ over an algebraically closed field of characteristic zero.

\subsection*{Boundedness}
Let $X$ be a smooth $d$-dimensional projective variety over an algebraically closed field $k$ of characteristic zero, let $\underline{L} = (L_1,\ldots,L_{j_0})$ be a vector of ample line bundles, and let $\tau \in B(X)_\mathbb{Q}$, see Definition \ref{def:toptype}.  In order to investigate boundedness and moduli spaces with respect to a whole family of stability conditions, we will say that a set $\Sigma\subset (\mathbb R_{\ge 0})^{j_0}\setminus\{0\}$ of stability parameters is  \emph{bounded} (with respect to the data $\tau,\underline{L}$) if the set of all sheaves of topological type $\tau$ that are semistable with respect to some $\sigma\in \Sigma$ is bounded. Note that for technical reasons, mostly related to the Hodge Index Theorem and Bogomolov's inequality,  we restrict to smooth varieties in this part. Using this terminology, our two main boundedness results can be formulated as follows.

\begin{theorem*}[Corollary~\ref{thm:positiveimpliesstronglybounded}]
Let $X$ be a smooth projective variety of dimension $d$, $\tau \in B(X)_\mathbb{Q}$, and $L_1, \dots, L_{j_0}$ be ample line bundles on $X$.   Furthermore, suppose that $\Sigma\subset (\mathbb R_{\ge 0})^{j_0}$ is a  closed convex polyhedral cone with the origin removed.
If 
$$ \sum\nolimits_j \sigma_j \,c_1(L_j)^{d-1} \in \Pos(X)_{\mathbb R} \quad \text{ for all }(\sigma_1, \dots, \sigma_{j_0}) \in \Sigma, $$
then $\Sigma$ is a bounded set of stability parameters with respect to $\tau$ and $\underline{L}=(L_1, \dots, L_{j_0})$.
\end{theorem*}
Here, the positive cone $\mathrm{Pos}_\R(X) \subset N_1(X)_\R$ is the cone of classes that are $(d-1)$st powers of real ample classes.
\begin{theorem*}[Corollary~\ref{cor:boundednesssurfaceorpicard2}]
  Let $X$ be a smooth projective variety, $\tau \in B(X)_\mathbb{Q}$, and $L_1, \dots, L_{j_0}$ be ample line bundles on $X$. In addition, suppose that either
  \begin{enumerate}
  \item the rank of the torsion-free sheaves under consideration is at most two, or
  \item the dimension of $X$ is at most three, or
  \item the Picard rank of $X$ is at most two.
  \end{enumerate}
Then, the whole set $(\mathbb R_{\ge 0})^{j_0}\setminus\{0\}$ of stability parameters is bounded with respect to $\tau$ and $\underline{L}=(L_1, \dots, L_{j_0})$.
\end{theorem*}

\subsection*{Variation of moduli spaces}
As laid out above, our interest in $\mathcal M_{\sigma}$ really comes from how it changes as $\sigma$ varies.  To discuss this, fix $(L_1, \dots, L_{j_0})$ and suppose $\Sigma\subset (\mathbb R_{>0})^{j_0}\setminus\{0\}$ is such that the set of sheaves of a given topological type that are semistable with respect to some $\sigma\in \Sigma$ is bounded. 

\begin{theorem*}[Chamber structure, Proposition~\ref{prop:weakchamberstructure} and Corollary~\ref{cor:masterspace}]
 The set $\Sigma$ is cut into chambers (such that the moduli space $\mathcal M_{\sigma}$ is unchanged as $\sigma$ varies in the interior of a single chamber) by a finite number of linear rational walls.  As $\sigma$ moves over a wall separating two chambers, the corresponding moduli spaces are related by a finite number of Thaddeus-flips.
\end{theorem*}

Here, by a \emph{Thaddeus-flip} we mean a transformation occurring as a change of GIT stability on a fixed ``master space''.  More precisely, we say two schemes $X^+$ and $X^-$ are related by a Thaddeus-flip if there exists a quasi-projective scheme $R$ with an action of a reductive group $G$ and stability parameters $\sigma_+,\sigma_-,\sigma$ such that there exists a diagram of the form
\[\begin{xymatrix}{
  X^+=R^{ss,\sigma_+}\hq G \ar[rd]_{\psi_+}& & \ar[ld]^{\psi_-} R^{ss,\sigma_-}\hq G=X^-\\
   &  R^{ss, \sigma}\hq G, &  
}
\end{xymatrix}
\]
where  $R^{ss,\sigma}$ denotes the set of points that are GIT-semistable with respect to $\sigma$, and the morphisms $\psi_{\pm}$ are induced by inclusions $R^{ss,\sigma_+}\subset R^{ss,\sigma}\supset R^{ss,\sigma_-}$.   In fact, in our case $R$ will be affine, $G$ will be a product of general linear groups, and the $\sigma,\sigma_{\pm}$ will come from characters of $G$.  
We emphasise that a Thaddeus-flip is not necessarily a flip in the sense of birational geometry, since a priori even if all the spaces involved are non-empty, this transformation could be a divisorial contraction or contract an irreducible component. However, from the theory of Variation of GIT due to Thaddeus \cite{Thaddeus} and Dolgachev-Hu \cite{DolgachevHu} it will consist of a sequence of birational flips under certain circumstances. In fact, our result is slightly stronger in that the same master space $R$ is used for all the Thaddeus-flips that occur between the different $\mathcal M_{\sigma}$ as $\sigma$ varies in $\Sigma$.  We emphasise that for the above variation result we require that if $\sigma\in \Sigma$ then each $\sigma_j$ be \emph{strictly} positive;  this should not really be necessary and we hope to address this in the future.

As an application of this technology, we prove the following result concerning the variation of moduli spaces on smooth threefolds.

\begin{theorem*}[Variation of Gieseker moduli spaces on threefolds, Theorem \ref{thm:MTprincipleGieseker3}]
Let $X$ be a smooth  projective threefold over an algebraically closed field of characteristic zero, let $\tau \in B(X)_\mathbb{Q}$, and let  $L_1, L_2$ be ample line bundles on $X$. Then, the moduli spaces $\mathcal{M}_{L_1}$ and $\mathcal{M}_{L_2}$ of sheaves of topological type $\tau$ that are Gieseker-semistable with respect to $L_1$ and $L_2$, respectively, are related by a finite number of Thaddeus-flips.
\end{theorem*}

\subsection*{Semistability and moduli spaces for K\"ahler polarisations}
As a second application of the above, we consider the notion of Gieseker-stability with respect to a real class $\omega\in N^1(X)_{\mathbb R}$ on a smooth projective manifold $X$. To define stability with respect to $\omega$, for a torsion-free sheaf $E$ consider the quantity
$$p_E(m) = \frac{1}{\rank(E)}\int_X ch(E) e^{m\omega} \Todd(X),$$
where $\Todd(X)$ is the Todd class of $X$.  We say that $E$ is \emph{(semi)stable with respect to $\omega$} if for all proper coherent  $E'\subset E$ we have $p_{E'}(m)(\le) p_E(m)$ for all $m$ sufficiently large.  When $\omega$ represents the first Chern class of an ample line bundle $L$, the Riemann-Roch theorem states that $p_{E}(m)$ equals $\frac{1}{\rank(E)} \chi(E\otimes L^m)$, and so this generalises the notion of Gieseker-stability from integral classes to real classes. Using this notation, our result can be formulated as follows:

\begin{theorem*}[Projective moduli spaces for $\omega$-semistable sheaves, Theorem~\ref{thm:kaehlermodulithreefolds}]
  Let $\omega\in N^1(X)_{\mathbb R}$ be any real ample class on a smooth projective threefold.  Then, there exists a projective moduli space $\mathcal{M}_\omega$ of torsion-free sheaves of fixed topological type that are semistable with respect to $\omega$.  This moduli space contains an open set consisting of stable sheaves, and points on the boundary correspond to $S$-equivalence classes  of properly semistable  sheaves.  
\end{theorem*}

It is likely that the assumption that $X$ has dimension three is not really necessary. We emphasise that the above moduli space is \emph{projective} despite us using a real class to define the stability condition, contrary to the expectation expressed for example in \cite[p.~217, after Main Theorem]{Schmitt}. Note however that algebraicity phenomena similar to the one observed here have been discovered earlier in K\"ahler Reduction Theory and Geometric Invariant Theory, see for example \cite{MomentumProjectivity} or \cite{PaHq}. 

Via the Kobayashi-Hitchin correspondence, the above result thus constructs a modular compactification of the moduli space of vector bundles of topological type $\tau$ on $X$ that carry a Hermite-Einstein connection with respect to a K\"ahler form representing $\omega$. In other words, in our situation it yields a positive answer to the important existence question for compact moduli spaces of semistable sheaves on compact K\"ahler manifolds, which was raised by Tyurin and discussed for example by Teleman in \cite[Sect.~3.2]{Tel08}, and which in its general form is wide open.

\subsection*{Method of construction}

Our method to construct the moduli space is based on the functorial approach introduced by \'Alvarez-C\'onsul--King \cite{ConsulKing}.  This in turn parallels that of Simpson \cite{Simpson} which we recall first.   For a coherent sheaf $E$ we can choose $n$ sufficiently large so that $E\otimes L^n$ is globally generated, i.e., the evaluation map 
$ H^0(E\otimes L^n)\otimes\mathcal O_X \to E\otimes L^n$
is surjective.  Thus, choosing an isomorphism $H^0(E\otimes L^n)\simeq V$, where $V$ is a fixed vector space of the appropriate dimension, we can thus consider $E$ as a point in the Quot scheme of the trivial bundle with fibre $V$.  Expanding slightly, letting $H:=H^0(L^{m-n})$,
we have that for $m$ sufficiently large the natural multiplication
\begin{equation}
 V \otimes H \simeq H^0(E\otimes L^n) \otimes H  \to H^0(E\otimes L^m)\label{eq:simpsonquot}\tag{$\diamondsuit$}
 \end{equation}
is surjective, and thus gives a point in a Grassmannian of $V\otimes H$.  The different choices of isomorphism correspond to the orbits of this point under the natural $GL(V)$-action,  thus the moduli space desired is the quotient with respect to $GL(V)$. This quotient can be constructed using GIT, and it is at this stage the stability condition enters.\medskip

The insight of \'Alvarez-C\'onsul--King is that it is possible to delay the point at which one picks the isomorphism in \eqref{eq:simpsonquot}, and thus give a more ``functorial'' construction.   So, instead of \eqref{eq:simpsonquot} we consider the multiplication map
\begin{equation}
 H^0(E\otimes L^n) \otimes H \to H^0(E\otimes L^m)\label{eq:kingrep}\tag{$\heartsuit$}
\end{equation}
as representation of a certain quiver.  In fact, this is a representation of the so-called Kronecker-quiver given by
$$ \bullet \overset{H}{\longrightarrow} \bullet,$$
where the notation means there are $\dim H$ arrows between the two vertices;  so, by definition a representation of this quiver is precisely a morphism of vector spaces $V'\otimes H\to W$ for some vector spaces $V'$ and $W$, just as in \eqref{eq:kingrep}.  For a given sheaf $E$ one can show that for $m\gg n\gg 0$ this representation recovers $E$.   In fact more is true, and this association gives a fully faithful embedding from the category of (suitably regular) sheaves into the category of representations of this quiver.   One can then appeal to previous work of King \cite{King}, which uses GIT to produce a projective moduli space of semistable representations of a given quiver without oriented cycles.  
Thus, the task becomes to relate stability of the sheaf $E$ with stability of the corresponding representation, which makes up a substantial part of the work in \cite{ConsulKing}.\medskip

Now, for stability with respect to several ample line bundles we will do the same, only with a more complicated quiver.  For simplicity, suppose we have only two line bundles $L_1$ and $L_2$, and for $i,j =1,2$ let $H_{ij} = H^0(L_i^{-n}\otimes L_j^m)$.    Then, given a sheaf $E$ we will consider the diagram of linear maps
\[\begin{xymatrix}{
   H^0(E\otimes L_1^n)  \ar[rr]|{H_{11}} \ar[rrd]|<<<<<{H_{12}} & & H^0(E \otimes L_1^m) \\
 H^0(E\otimes L_2^n)   \ar[rr]|{H_{22}} \ar[rru]|<<<<<{H_{21}}  & &  H^0(E \otimes L_2^m).
}
  \end{xymatrix}
\]
Here, all the maps are given by natural multiplication; for example, the top row is the linear map $H^0(E\otimes L_1^n) \otimes H_{11}  = H^0(E\otimes L_1^n)\otimes H^0(L_1^{m-n})\to H^0(E\otimes L_1^m)$.  Thinking of this as a representation of an appropriate quiver, we will show that the stability of this representation is, under suitable hypotheses, the same as the stability of the sheaf $E$ as defined above. Thus, we can again appeal to \cite{King} to get the desired  moduli space.  Note that the representation of the Kronecker quiver computed from a sheaf $E$ with respect to the line bundle $L_j$ appears in the $j$-th row of our quiver. Heuristically, the diagonal maps ensure that these representations are related in the correct way, as given by the appropriate multiplication maps. What makes this quiver adapted to the variation problem at hand is the existence of a non-trivial space of stability conditions that one can vary to weigh the constribution of the individual rows (whereas for the Kronecker-quiver there is only one) and hence to interpolate between the semistability conditions coming from the single Kronecker quivers. 

\subsection*{Comparison with other works}

The variation of the moduli space of (slope) semistable sheaves on a smooth surface has attracted a lot of interest due to the connection with Donaldson invariants, for example \cite{Donaldson,FriedmanMorgan,HuLi,Qin,Gottsche,Yamada1,Yamada2} in addition to the references above.  For the most part, these works aim to describe explicitly how the moduli spaces change as the polarisation varies (often for specific classes of surfaces and for particular topological types) to understand precisely how these invariants change.  Thus, they avoid framing the problem as one of Variation of GIT.  The most prominent exception to this is the work of Matsuki-Wentworth \cite{MatsukiWentworth} who use GIT to completely solve the problem on surfaces for Gieseker-semistability.

As discussed above, if one wishes to understand the variation problem for the moduli spaces on bases of higher dimension one has to confront the fact that the walls in the ample cone may no longer be rational linear or locally finite (as described by explicit ``pathologies'' due to Qin \cite[Sect.~2.3]{Qin} and Schmitt \cite[Ex.~1.1.5]{Schmitt}).  The main result of \cite{Schmitt} deals with the case that the polarisation crosses a single wall in the ample cone, under the rather restrictive assumption that this wall contains a rational point.  We will see that we are able to relax this assumption, at least on threefolds.

The notion of Gieseker-semistability can be extended in many ways, and Rudakov \cite{Rudakov} was the first to place these in the context of general abelian categories.  This has since been built on by Joyce's epic [\citen{JoyceI}-\citen{JoyceIV}]  who uses this to understand the information held by Donaldson-Thomas invariants by describing the ``wall crossing'' formulae that govern their change as the stability condition varies.   The multi-Gieseker-semistability considered here is certainly a special case of one of Joyce's stability conditions (who uses the word ``permissible'' for what we refer to as ``bounded'').   Joyce's work does not consider (or really has use for) the coarse moduli spaces, and instead works throughout with the relevant kinds of stacks. Our two main technical results (namely the Embedding Theorem and Comparison of Semistability Theorem) that allow us to pass from sufficiently regular (semistable) sheaves to (semistable) representations of a certain quiver can also be interpreted as a statement about the corresponding stacks, but it is not clear what use this might have.  We also comment that our results yield new geometric situations in which the simpler approach of Kiem and Li can be applied to give similar wall-crossing formulae (see paragraph four of \cite[p.~3]{KiemLi}, where the authors propose using that the moduli spaces in question appear as a GIT quotient).

Another extension of the notion of Gieseker-stability concerns the decorated sheaves of Schmitt \cite{Schmitt,  SchmittBuch} as well as work by Schmitt \cite{SchmittQuiver} and \'Alvarez-C\'onsul \cite{AlvarezQuiver} on quiver sheaves, by which we mean representations of a given quiver in the category of sheaves (see also \cite{LaudinSchmitt} for a survey).  In these works stability is also defined in terms of linear combinations of Hilbert polynomials (where the coefficients are even allowed to be polynomials), but all with respect to a \emph{single} polarisation.  It seems likely that at least part of this work could be generalised to include the case of several polarisations.  If this were done, it might be possible to apply it to get similar variation results (perhaps by some kind of ``diagonal argument'') but it is not at all obvious to the authors how this should be done.  More significantly, perhaps, the cited works do not address the additional problems encountered with boundedness when dealing with several polarisations.

The variation of the moduli space of Gieseker-semistable sheaves on surfaces due to Matsuki-Wentworth has recently been interpreted by Bertram in the context of Bridgeland stability conditions \cite{Bertram}. It is hence a natural question to ask if anything similar can be said for the generalisation we describe here.

\subsection*{Characteristic of the base field} We have stated most of our main results assuming the base field has characteristic zero.  As is well known, the construction of the moduli space of Gieseker-semistable sheaves in arbitrary characteristic was achieved by Langer \cite{Langer}, and much of the preliminary results discussed here also hold in that case (for instance, we state and prove our boundedness results in Section \ref{section:boundedness} in arbitrary characteristic).  In fact, the construction of \'Alvarez-C\'onsul--King \cite{ConsulKing}, on which our construction is based, also holds in arbitrary characteristic, up to a complication concerning the scheme structure of the moduli space produced (coming from the fact that in characteristic zero taking quotient rings and invariant subrings commute due to the existence of the Reynolds operator, see \cite[Prop.~6.3]{ConsulKing}). In positive characteristic, this complication leads to problems with our organisation of the arguments in Chapter 10. These problems can very likely be solved by taking into account the subtleties of Geometric Invariant Theory in positive characteristic. 

\subsection*{Preview}
In a sequel to this paper, \cite{GRT15}, we give further applications of the machinery developed here.  First, we prove that if $X$ is smooth of any dimension, and $L'$ and $L''$ are \emph{general} polarisations in $\Amp(X)$, then the moduli space $\mathcal M_{L'}$ and $\mathcal M_{L''}$ of torsion-free sheaves taken with regard to $L'$ and $L''$, respectively, are related by a finite number of Thaddeus-flips.  Second, we revisit the threefold case and show, again assuming that $L'$ and $L''$ are general, that one can even identify the intermediate spaces that appear in the sequence of Thaddeus-flips as moduli spaces of multi-Gieseker-semistable sheaves, thereby generalising the work of Matsuki--Wentworth from surfaces to threefolds. 

\subsection*{Acknowledgements } The authors wish to thank Arend Bayer for helpful conversations at a crucial stage of this project. We also thank Dominic Joyce, Jun Li, and Alexander Schmitt for discussions comparing this work to theirs.  We also thank the referee for her/his careful reading and critical comments.

DG has been partially supported by the DFG-Research Training Group 1821 ``Cohomological Methods in Algebraic Geometry'' (Albert-Ludwigs-Universit\"at Freiburg) as well as by the Eliteprogramm f\"ur Postdoktorandinnen und Postdoktoranden of the Baden-W\"urttemberg Stiftung. JR is supported by an EPSRC Career Acceleration Fellowship (EP/J002062/1). Part of this research was done during stays of MT at the Max-Planck-Institut f\"ur Mathematik in Bonn and at the Ruhr-Universit\"at Bochum. The latter visit was financed by SFB / TR 12 ``Symmetries and Universality in Mesoscopic Systems''.

\part{Multi-Gieseker-Stability}
 
\section{Preliminaries}

\subsection{Notation and Terminology: } We follow closely the notation and terminology used in the bible \cite{Bible}.   Given polynomials $p,q\in \mathbb R[X]$ we write $p\le q$ to mean $p(m)\le q(m)$ for $m\gg 0$ and similarly for strict inequality, which is the same as the lexicographic order on the vector of coefficients of the two polynomials.  We write ``for $m\gg n\gg 0$'' to mean there is an $n_0$ such that for all $n\ge n_0$ there is an $m_0\ge n$ such that for all $m\ge m_0$ the statement in question holds, and similarly for the expression ``$m\gg n\gg p\gg 0$''. All the sheaves considered in this paper will be coherent, and we will only emphasise this when appropriate; in particular, if $E$ is coherent then (semi)stability is to be tested with respect to coherent subsheaves of $E$.

A \emph{$\mathbb Q$-line bundle} is a formal tensor power $L = M^{\frac{p}{q}}$ 
 for some line bundle $M$ and rational number $\frac{p}{q}$, which we say is ample if $M$ is an ample line bundle and $\frac{p}{q}$ is a positive rational number.   We write $\deg_L(E)$ for the degree of a coherent sheaf $E$ with respect to an ample line bundle, which extends to ample $\mathbb Q$-line bundles by linearity.

\subsection{Preliminaries on sheaves}

For the construction of the moduli space we assume that $X$ is a projective scheme over an algebraically closed field of characteristic zero.
The \emph{dimension} of a coherent sheaf $E$ on $X$ is the dimension of its support $\{ x\in X : E_x\neq 0\}$, and we say that $E$ is \emph{pure} of dimension $d$ if all non-trivial coherent subsheaves $F\subset E$ have dimension $d$.  The \emph{saturation} of a subsheaf $E'\subset E$ is the minimal subsheaf $F$ containing $E'$ such that $E/F$ is either pure or zero.   Given an ample line bundle $L$ on $X$ the Hilbert polynomial of a coherent sheaf $E$ of dimension $d$ can be written uniquely as
\begin{equation}
P_E^L(m):=\chi(E\otimes L^m) = \sum_{i=0}^{d} \alpha_i^L(E) \frac{m^i}{i!}\label{eq:expansionhilbertpoly}
\end{equation}
for  some $\alpha^L_i(E)\in\mathbb{Q}$.  If $E$ is non-zero, then we denote the \emph{multiplicity} as
$r_E^L:=\alpha_{d}^L(E)$, 
which is a strictly positive integer.     The \emph{reduced Hilbert polynomial} of $E$ is
$$ p_E^L(m) := \frac{P_E^L(m)}{r_E^L}$$
and the \emph{slope} of $E$ is
$$ \hat{\mu}^L(E) := \frac{\alpha_{d-1}^L(E)}{r_E^L}.$$ 
Thus, by definition
\begin{equation}
 p_E^L(m) = \frac{m^d}{d!} + \hat{\mu}^L(E)\,\frac{m^{d-1}}{(d-1)!} + O(m^{d-2}).\label{eq:expandreduced}
\end{equation}

We say that $E$ is \emph{Gieseker-(semi)stable with respect to $L$} if it is pure and for all proper coherent subsheaves $F$ we have $p_F^L (\le) p_E^L$.  This and similar sentences should be read as two statements, namely that semistability means $p_F^L \le p_E^L$ and stability means $p_F^L < p_E^L$.     Observe that the definition of Gieseker-(semi)stability is unchanged if $L$ is scaled by a positive multiple, and so extends to the case that $L$ is an ample $\mathbb Q$-line bundle. 

\begin{remark}\label{rmk:rank}
  If the dimension $d$ of $E$ equals $\dim X$, then the \emph{rank} of $E$ is defined to be
$$\rk(E) :=\frac{ \alpha_d^L(E)}{\alpha_d^L(\mathcal O_X)}.$$
When $X$ is integral, there is an open dense $U\subset X$ on which $E$
is locally free and then $\rk(E)$ is the rank of the vector bundle
$E|_U$.
\end{remark}

\begin{example}[Riemann-Roch I]\label{ex:RRochI}
  Suppose that $X$ is smooth of dimension $d$ and let $L$ be an ample line bundle.  Then, by the Riemann-Roch theorem the multiplicity of  a torsion-free sheaf $E$ is 
$$ r_E^L = \rank(E) \int_X c_1(L)^d$$
and
\begin{align}\label{eq:rrochI}
 \chi&(E\otimes L^m)=   r_E^L \frac{m^d}{d!} + \int_X (c_1(E) +  \rank(E)\,\Todd_1(X)). c_1(L)^{d-1} \frac{m^{d-1}}{(d-1)!} 
+ O(m^{d-2}),
\end{align}
where $\Todd_1(X) = -c_1(K_X)/2$ is the degree 2 part of the Todd class of $X$.  Thus, 
$$\hat{\mu}^L(E) = \frac{1}{\int_X c_1(L)^d}\frac{\deg_L(E)}{\rank(E)} + \frac{\int_X \Todd_1(X) c_1(L)^{d-1} }{\int_X c_1(L)^d}.$$
Note that this differs from the usual slope $\mu^L(E) := \deg_L(E)/\rank(E)$ by an affine linear function that is independent of $E$.  So $\mu^L(F)(\le) \mu^L(E)$ if and only if $\hat{\mu}^L(F) (\le) \hat{\mu}^L(E).$
\end{example}

\begin{definition}
  Any pure sheaf $E$ of dimension $d$ admits a unique \emph{maximally
    destabilising subsheaf} $E_{\max}\subset E$ with the property that
  $\hat{\mu}^L(F)\le \hat{\mu}^L(E_{\max})$ for all $F\subset E$ with
  equality implying that $F\subset E_{\max}$ \cite[1.3.5,1.6.6]{Bible}.  We
  write
$$ \hat{\mu}_{\max}^L(E) := \hat{\mu}^L(E_{\max}).$$
\end{definition}

\begin{definition}\label{def:toptype}
Let $\tau$ be an element of $B(X)_\Q:=B(X)\otimes_\Z\Q$, where $B(X)$ is the group of cycles on $X$ modulo algebraic equivalence, see \cite[Def.~10.3]{F}. We say that a sheaf  $E$ on $X$ is \emph{of topological type} $\tau$ if its homological Todd class $\tau_X(E)$ equals  $\tau$. 
\end{definition}

\begin{remark}\label{rmk:Todd}
The homological Todd class $\tau_X(E)$ of a sheaf $E$ is usually considered as an element in $A(X)_\Q:=A(X)\otimes\Q$, cf.~\cite[Chapter 18]{F}. For our purposes, it is more convenient to mod out cycles that are algebraically equivalent to zero, i.e., to work with $B(X)_\mathbb{Q}$, which has the same formal properties as $A(X)_\mathbb{Q}$ by \cite[Prop.~10.3]{F}. With this definition, the topological type of the members of a flat family of coherent sheaves over $X$ parametrised by a connected noetherian scheme is constant, see \cite[Ex.~18.3.8]{F}. The knowledge of the topological type $\tau$ of a sheaf completely determines its Hilbert polynomial with respect to any ample line bundle by \cite[Example 18.3.6]{F}. Note also that in the case where $X$ is a smooth variety over $\mathbb{C}$ the class $\tau_X(E)$ determines the Chern character of $E$. If one prefers not to use this machinery, one can, for the most part of this paper, instead fix from the outset the Hilbert polynomials $k\mapsto \chi(E\otimes L_j^k)$ for all line bundles $L_j$ in question (the exception to this statement concerns the boundedness results in Section \ref{section:boundedness} in which the proofs we give use the topological type rather than the Hilbert polynomials).
\end{remark}

\begin{definition}
A set $\mathcal S$ of isomorphism classes of coherent sheaves on $X$ is said to be \emph{bounded} if there exists a scheme $S$ of finite type and a coherent $\mathcal O_{S\times X}$-sheaf $\mathcal E$ such that every $E\in \mathcal S$ is isomorphic to $\mathcal E_{\{s\}\times X}$ for some closed point $s\in S$.  
\end{definition}

\begin{definition} Let $L$ be a very ample line bundle on $X$.  We say a coherent sheaf $E$ is \emph{n-regular} with respect to $L$ if
$$H^i(E\otimes L^{n-i}) =0 \quad \text{ for all } i>0.$$
\end{definition}

When dealing with several line bundles the following definition is convenient:

\begin{definition}
Suppose that $\underline{L}= (L_1,\ldots,L_{j_0})$, where each $L_j$ is a very ample line bundle on $X$.  We say that a coherent sheaf $E$  is \emph{$(n,\underline{L})$-regular} if $E$ is $n$-regular with respect to $L_j$ for all $j \in \{1, \ldots, j_0\}$. 
\end{definition}

 Using \cite[Lemma 1.7.6]{Bible}, we see that the set of $(n,\underline{L})$-regular sheaves of a given topological type is bounded.  Conversely, it follows from the Serre Vanishing Theorem that, if $\mathcal S$ is a bounded family of sheaves, then for $n \gg 0$ each $E\in \mathcal S$ is $(n,\underline{L})$-regular.

\section{Stability with respect to several polarisations}\label{section:stabilitywrtseveralpolarisations}

In this section, we introduce a stability condition for coherent sheaves. This stability condition depends on a number of (fixed) line bundles $L_1, \dots, L_{j_0}$, as well as on a number of real parameters that will later allow us to interpolate between the different notions of Gieseker-stability with respect to the $L_j$.

\begin{definition}
By a \emph{stability parameter} we mean the data
$$\sigma = (\underline{L},\sigma_1,\ldots,\sigma_{j_0})$$
where $\underline{L} = (L_1,\ldots,L_{j_0})$ for some ample line bundles $L_j$ on $X$, and  $\sigma_j\in \mathbb R_{\ge 0}$ are such that not all the $\sigma_j$ are zero.  We say that $\sigma$ is \emph{rational} if all the $\sigma_j$ are rational, and that it is \emph{positive} if $\sigma_j>0$ for all $j$.
\end{definition}

In the subsequent discussion the vector $\underline{L}$ will be fixed, so by abuse of notation we will sometimes confuse $\sigma$ and the vector $(\sigma_1,\ldots,\sigma_{j_0})$.    Thus, we allow $\sigma$ to vary in a subset of $(\mathbb R_{\ge 0})^{j_0}\setminus\{0\}$.  We emphasise that whereas we allow the $\sigma_j$ to be irrational, we will always assume that the $L_j$ are genuine (integral) line bundles. For now we fix such a stability parameter $\sigma$.

\begin{definition}[Multi-Hilbert polynomial]
   The \emph{multi-Hilbert polynomial} of a coherent sheaf $E$ with respect to $\sigma$ is
   \begin{equation}\label{eq:defalpha}
     P_E^\sigma(m) := \sum\nolimits_{j} \sigma_j \chi(E\otimes L_j^m) .
   \end{equation}
 \end{definition}

If $E$ has dimension $d$, we can write
$$P_E^{\sigma}(m) =\sum_{i=0}^d \alpha^{\sigma}_i(E) \frac{m^i}{i!},$$
where from \eqref{eq:expansionhilbertpoly} the coefficients are given by  $\alpha^{\sigma}_i(E) = \sum_j \sigma_j \alpha_i^{L_j}(E)$.   We let
$$r^{\sigma}_E:=\alpha^\sigma_d(E) = \sum\nolimits_j \sigma_j r_{E}^{L_j},$$
which is strictly positive by the hypothesis on $\sigma$.

\begin{definition}
The \emph{reduced multi-Hilbert polynomial} of a coherent sheaf of dimension $d$ is defined to be
$$ p_E^{\sigma}(m) := \frac{ P_E^{\sigma}(m)}{r^{\sigma}_E}.$$
Thus, defining
$$ \hat{\mu}^{\sigma}(E) := \frac{\alpha_{d-1}^{\sigma}(E)}{r^{\sigma}_E}$$ 
we have
$$ p_E^{\sigma}(m) = \frac{m^d}{d!} + \hat{\mu}^{\sigma}(E)\,\frac{m^{d-1}}{(d-1)!} + O(m^{d-2}).$$
\end{definition}

\begin{remark}
We will later frequently use the observation that the quantities $P^{\sigma}_E$, $r_E^{\sigma}$, $\hat{\mu}^{\sigma}(E)$ and $p^{\sigma}_E$ are determined by $\sigma$ and the topological type of $E$.
\end{remark}

\begin{definition}\label{defi:multiGiesekerstability}[Multi-Gieseker-stability]
   We say that a coherent sheaf $E$ is \emph{multi-Gieseker-(semi)stable, or just \emph{(semi)stable},}  if it is pure and for all proper subsheaves $F\subset E$ we have 
  \begin{equation}
p_F^{\sigma} (\le) p_E^{\sigma}.
\end{equation}
\end{definition}

\begin{example}[Riemann-Roch II]\label{ex:RRochII}
If $X$ is smooth of dimension $\dim X=d$, and if $E$ is torsion-free (as in in the setup of Remark \ref{ex:RRochI}) we have
$$ r_E^\sigma = \rank(E) \sum\nolimits_j \sigma_j \int_X c_1(L_j)^d$$
and 
 \begin{equation}
 \hat{\mu}^{\sigma}(E) = C_1 \frac{\sum_j \sigma_j \deg_{L_j}(E)}{\rank (E)} + C_2, \label{eq:RRochII}
\end{equation}
where $C_1, C_2$ are given by 
\begin{equation*}
 C_1 = \frac{1}{\sum_j \sigma_j \int_X c_1(L_j)^d} \;\text{ and }\;C_2 = \frac{\sum\nolimits_j \sigma_j\int_X \Todd_1(X).c_1(L_j)^{d-1}}{\sum_j \sigma_j \int_X c_1(L_j)^d}
\end{equation*}
and are therefore in particular independent of $E$.   Moreover, $$p_E^{\sigma}(m) = C_1 \frac{\sum_j \sigma_j \chi(E\otimes L_j^m)}{\rank( E)}$$ and so the definition of stability given here agrees with that in the introduction.
\end{example}

\begin{remark}\label{rmk:veryample}
Clearly, stability of $E$ is unchanged if $\sigma$ is scaled by a positive multiple.    Moreover, it is likewise unchanged if each $L_j$ is replaced by $L_j^{s}$ for some integer $s$ (since, up to scaling by a positive constant, the reduced multi-Hilbert polynomial changes to $k\mapsto p_E^{\sigma}(sk)$).  Thus, there is no loss in generality if all the $L_j$ are assumed to be very ample.
\end{remark}

\begin{example}[Relation with usual Gieseker-stability]
Let $\sigma =\underline{e}_i$ where $\underline{e}_i = (0,\ldots,1,\ldots,0)$ is the standard basis vector.   Then,  $E$ is (semi)stable with respect to $\sigma$ if and only if it is Gieseker-(semi)stable with respect to $L_i$.    
\end{example}

\begin{example}[Picard number 1]
Suppose that the Picard number $\rho:=\rho(X)$ of $X$ is $1$. Then for a torsion-free sheaf, stability with respect to any stability parameter is the same as Gieseker-stability.  To see this, let $\sigma = (\underline{L},\sigma_1,\ldots,\sigma_{j_0})$ be any stability parameter.     Fix an ample generator $A$ of $\Pic(X)$ so $L_j = A^{a_j}$ for some $a_j\in \mathbb Z_{\ge 1}$.  Then, up to an unimportant affine transformation as in Example \ref{ex:RRochII}, 
$$p_E^{\sigma} = \frac{1}{\rank( E)} \sum\nolimits_j \sigma_j P_E^A(a_jm).$$
Since the $\sigma_j$ are non-negative, it follows that $p^{\sigma}_F(m)(\le) p^{\sigma}_E(m)$ for $m\gg 0$ if and only if $p^A_F(m)(\le)p^A_E(m)$ for $m\gg 0$.
\end{example}

\begin{example}[Surfaces]
  Let $X$ be a smooth surface and  suppose that $\sigma= (\underline{L} ,\sigma_1,\ldots,\sigma_{j_0})$ is a \emph{rational} stability parameter.    Then, $E$ is (semi)stable with respect to $\sigma$ if and only if it is Gieseker-(semi)stable with respect to the $\mathbb Q$-line bundle $\otimes_j L_j^{\sigma_j}$.   This is because on a surface
$$ P_E^L(m) = r_{E}^L \frac{m^2}{2} + \alpha_1^L(E) m + \alpha_2(E),$$
and $\alpha_2(E)$ is independent of $m$ (as is apparent, for example, using the Riemann Roch theorem).  Thus,
$$ p_E^{\sigma}(m) = \frac{m^2}{2} +  \hat{\mu}^{\sigma}(E)\, m + \frac{\sum_j \sigma_j}{\sum_j \sigma_j r_E^{L_j}} \alpha_2(E).$$
Consequently, $p_F^{\sigma}\le p_E^{\sigma}$ if and only if (1) $\hat{\mu}^{\sigma}(F)\le\hat{\mu}^{\sigma}(E)$ and (2) if equality holds then we have $\alpha_2(F)\le \alpha_2(E)$.   Now clearing denominators we can choose an integer $s$ so that $\tilde{L}:= \otimes_j L_j^{s\sigma_j}$ is an integral line bundle.  Then, looking back at Example \ref{ex:RRochII} and using that $\deg_L$ on a surface is linear in $L$ we see
$$ \hat{\mu}^{\sigma}(E)= C_1 \frac{\sum_j \sigma_j \deg_{L_j}(E)}{\rank (E)} + C_2 = C_1 \frac{\deg_{\tilde{L}}(E)}{s\cdot  \rank (E)} +C_2.$$  
Thus, $p_F^{\sigma}\le p_E^{\sigma}$ if and only if  $p_F^{\tilde{L}}\le p_E^{\tilde{L}}$. The proof of the statement for stability is similar.
\end{example}

We now collect some of the basic properties of stability, several of which are analogous to those for Gieseker-stability.

\begin{lemma}\label{lem:slope}
Suppose $ \hat{\mu}^{\sigma}(E)(\le )\mu$ for some real number $\mu$.   Then, there exists an $j\in\{1,\dots,j_0\}$ such that $\sigma_j\neq 0$ and $ \hat{\mu}^{L_j}(E)(\le )\mu$.    Moreover, the analogous statement holds with $(\ge)$ instead of $(\le )$.   
\end{lemma}
\begin{proof}
Suppose that $ \hat{\mu}^{L_j}(E)\ge \mu$ for all $j$  with $\sigma_j\neq 0$.   Then for all $j$ we have  $\sigma_j\alpha_{d-1}^{L_j}(E)\ge\sigma_jr_{E}^{L_j} \mu$, and summing over $j$ yields $\hat{\mu}^{\sigma}(E)\ge \mu$.     The other statements are proved similarly. 
\end{proof}

\begin{definition}
   We say that a proper coherent subsheaf $E'\subset E$ is \emph{destabilising} if $p^{\sigma}_{E'}\ge p_{E}^{\sigma}$.  So if $E$ is semistable, a proper subsheaf $E'\subset E$ is destabilising if and only if $p_{E'}^{\sigma} = p_E^{\sigma}$.  
\end{definition}

\begin{lemma}\label{lem:destabilisingimpliessaturated}
  Let $E$ be semistable and $E'\subset E$ be destabilising.  Then, $E'$ is saturated, and $E'\oplus (E/E')$ is semistable.
\end{lemma}
\begin{proof}
By hypothesis, we have $p^{\sigma}_{E'} = p^{\sigma}_E$.  Let $F$ be the saturation of $E'$.   Then $r^{\sigma}_{E'} = r^{\sigma}_F$, and for $n\gg 0$ we have
$$
p_{E}^{\sigma}(n) = p_{E'}^\sigma(n) = \frac{\sum_j \sigma_j h^0(E' \otimes L_j^n)}{r_E^{\sigma}} \le \frac{\sum_j \sigma_j h^0(F \otimes L_j^n)}{r_F^{\sigma}}=p^{\sigma}_F(n) \le p^{\sigma}_E(n)$$
where the last inequality uses that $E$ is semistable.   Thus equality holds throughout. As $h^0(E'\otimes L_j^n)\le h^0(F\otimes L_j^n)$ for all $j$, and as all the $\sigma_j$ are non-negative, by choosing some $j$ so that $\sigma_j\neq 0$ this implies $H^0(E'\otimes L_j^n) = H^0(F\otimes L_j^n)$.  But $F\otimes L_j^n$ is globally generated for $n\gg 0$, so this implies that $F\subset E'$, and hence in fact $E'$ is saturated.  For the second statement observe that as $E'$ has the same reduced multi-Hilbert polynomial as $E$ it must also be semistable.  Moreover, the quotient $E/E'$ is pure (as $E'$ is saturated) and also has the same reduced multi-Hilbert polynomial as $E$.  Hence the direct sum $E'\oplus E/E'$ is semistable.
\end{proof}

\begin{lemma}\label{lem:saturated}
  Let $E$ be a pure coherent sheaf of dimension $d$. Then, the following are equivalent:
  \begin{enumerate}
  \item $E$ is (semi)stable.
  \item For all proper saturated $F\subset E$ one has $p^{\sigma}_F(\le)p^{\sigma}_E$.
  \item For all proper quotients $E\to G$ with $\alpha^{\sigma}_d(G)>0$ one has $p^{\sigma}_E(\le)p^{\sigma}_G$.
  \item For all proper pure quotient sheaves $E\to G$ of dimension $d$ it holds that $p^{\sigma}_E(\le)p^{\sigma}_G$.
  \end{enumerate}
\end{lemma}

\begin{proposition}[Harder-Narasimham filtration]\label{prop:HN}
Let $E$ be a non-trivial pure sheaf of dimension $d$.  Then, there exists a unique \emph{Harder-Narasimhan} filtration
$$ 0= \HN_0(E) \subset \HN_1(E) \subset \ldots \subset \HN_l(E) = E$$ 
such that the factors $gr_i^{\HN}:=\HN_i(E)/\HN_{i-1}(E)$ are semistable sheaves of dimension $d$ with $p_i:=p_{gr_i^{\HN}}^{\sigma}$ satisfying
$$p_{\max}^{\sigma}(E) : = p_1 > \cdots >p_l=:p_{\min}^{\sigma}(E).$$
\end{proposition}

\begin{proposition}[Jordan-H\"older filtration]\label{prop:Sequivalence}
  Let $E$ be a semistable sheaf of dimension $d$.  Then, there exists a Jordan-H\"older filtration
$$ 0 = E_0 \subset E_1 \subset \cdots \subset E_l = E$$
whose factors $gr_i(E)= E_i/E_{i-1}$ are stable with reduced Hilbert polynomial $p^{\sigma}_{E_i} = p^{\sigma}_E$.  Moreover, up to isomorphism the sheaf $gr(E):= \oplus_i gr_i(E)$ does not depend on the choice of Jordan-H\"older filtration (but may, of course, depend on the stability parameter).
\end{proposition}

\begin{definition}[$S$-equivalence]
  We say two semistable sheaves $E_1$ and $E_2$ of the same topological type are \emph{$S$-equivalent} if $gr(E_1)$ is isomorphic to $gr(E_2)$.  
\end{definition}

The proofs of the above are exactly as in the usual case for Gieseker-stability \cite[1.2.6, 1.2.7, 1.3.4, 1.5.2]{Bible}, once it is observed that for any short exact sequence $0\to F\to E\to G\to 0$ we have $P^{\sigma}_E = P^{\sigma}_F + P^{\sigma}_G$ and $\alpha_d^{\sigma}(E) = \alpha_d^{\sigma}(F) + \alpha_d^{\sigma}(G)$.   

\begin{remark}
  Alternatively, one can appeal to the machinery of Rudakov
  \cite{{Rudakov}} concerning stability conditions on abelian
  categories.  The function $P_E^{\sigma}$ is an example of a
  ``generalised Gieseker-stability'' on the category of
  coherent sheaves on $X$ \cite[Definition 2.1, 2.7]{Rudakov}.
\end{remark}

\begin{lemma}\label{lem:openness}
For fixed stability parameter $\sigma$,  (semi)stability is an open property.
\end{lemma}
\begin{proof}
This is a small adaptation of \cite[2.3.1]{Bible}.  Suppose $E$ is a flat family of $d$-dimensional sheaves of a given topological type on $X$ parameterised by  a connected noetherian scheme $S$.  We have to show that the set of closed points $s\in S$ such that $E_s$ is semistable is open in $S$.  

 Let $\hat{\mu} = \hat{\mu}^{\sigma}(E_s)$ and write $p$ for the reduced multi-Hilbert polynomial of $E_s$ for some (and hence all)  $s\in S$.  Consider the set $\mathcal S$ of all pure $d$-dimensional sheaves $E''$ that arise as proper quotients $E_s\to E''_s$ for some closed point $s\in S$ with $\hat{\mu}^{\sigma}(E'')\le \hat{\mu}$.  By Lemma \ref{lem:slope}, for such an $E''\in \mathcal S$ there exists an $i\in\{1,\dots,j_0\}$ such that $ \hat{\mu}^{L_i}(E'')\le \hat\mu$.  Hence, the set $\mathcal S$ is bounded by Grothendieck's Lemma, see for example \cite[Lem.~1.7.9]{Bible} or \cite[Th\'eor\`eme 2.2]{Groth}.
 Thus, the set of polynomials $R$ that arise as Hilbert polynomials from the family in $\mathcal S$ is finite; i.e.,
$$\#\mathcal P := \#\{ R : R = P^{L_1}_{E''} \text{ for some }E''\in \mathcal S\} < \infty.$$ 

Now for each $R\in \mathcal P$ consider the relative Quot scheme 
$$Q(R) := Quot_{S}(E;R)\stackrel{\pi}{\to} S$$
 of $X\times S$ over $S$ parameterizing quotients $E_s\to E''$ with Hilbert polynomial $R$ (taken with respect to $L_1$).  This Quot scheme has a finite number of connected components, and by flatness of the universal family,  the topological type of any quotient $E_s\to E''$ is constant within each component (see Remark \ref{rmk:Todd}) 

Now fix $R\in \mathcal P$.  Then, there is finite set $\underline{\mathcal P}=\underline{\mathcal P}(R)$ of vectors of polynomials  $\underline{P}'' = (P''_1,\ldots,P''_{j_0})$ and a disjoint union
$$Q(R) = \coprod_{\underline{P}'' \in \underline{\mathcal P}} Q(R;\underline{P}''),$$
where $Q(R;\underline{P}'')$ consists of the union of those components of $Q(R)$ consisting of quotients $E''$ of whose Hilbert polynomials satisfy  $P''_j(k) = \chi(E''\otimes L_j^k)$ for all $j \in \{1, \dots, j_0\}$.

For $\underline{P}''=(P''_0,\ldots,P''_{j_0})\in \underline{\mathcal P}$, let $p''$ be the reduced multi-Hilbert polynomial associated to $\underline{P}''$ (and $\sigma$).  Consider the set
$$\underline{\mathcal P}_0 : = \underline{\mathcal P}_0(R):= \{ \underline{P}''\in \underline{\mathcal P}(R) : p''< p \},$$
where, we recall, $p$ denotes the reduced multi-Hilbert polynomial of the sheaves $E_s$.  Then, a sheaf $E_s$ is unstable if and only if there exists a quotient $E_s\to E''$ with $E''$ of this type for some $\underline{P}''\in \underline{\mathcal P}_0(R)$ and some $R\in \mathcal P$.  But as the Quot scheme is projective over $S$, the image $\pi(Q(R,\underline{P}''))$ is closed in $S$, and thus the set of semistable points are those that do not lie in the union of the finitely many closed sets $\pi(Q(R,\underline{P}'')\subset S \text{ for } R\in \mathcal P, \underline{P}''\in \mathcal P_0$. Thus, semistability is open, as claimed.  The proof for openness of stability is the same, up to replacing $\underline{\mathcal P}_0$ with the set $\{ \underline{P}''\in \underline{\mathcal P}(R) : p''\le p \}$.
\end{proof}

\section{Slope stability}

On occasion (particularly in connection with boundedness questions), we will relate multi-Gieseker-stability to a notion of slope stability.  For this,  we assume $X$ to be smooth of dimension $d$ and the sheaves in question to be torsion-free.

\begin{definition}[Slope stability]\label{def:slopstability}
  Let $\gamma\in N_1(X)_{\mathbb R}$. The \emph{slope} of a torsion-free coherent sheaf $E$ \emph{with respect to $\gamma$} is the quantity
$$\mu_{\gamma}(E) : = \frac{\int_X c_1(E)\cdot  \gamma}{\rank (E)}.$$
If $\alpha \in N^1(X)_{\mathbb R}$, the \emph{slope} of $E$ \emph{with respect to $\alpha$} is
$$\mu^{\alpha}(E) : = \frac{ \int_X c_1(E) \cdot \alpha^{d-1}}{\rank (E)}.$$
We say that a coherent torsion-free sheaf on $E$ is \emph{slope (semi)stable} with respect to $\gamma\in N_1(X)_{\mathbb R}$ (resp.\ $\alpha\in N^1(X)_{\mathbb R}$) if for all proper subsheaves $F\subset E$ of lower rank we have
$$ \mu_{\gamma}(F) (\le) \mu_{\gamma}(E) \;\;\text{ (resp.\ }\mu^{\alpha}(F) (\le) \mu^{\alpha}(E)\text{)}.$$
\end{definition}
\begin{remark}
In order to obtain a theory enjoying the expected basic properties, some positivity of the class $\alpha$ or $\gamma$ has to be assumed. For example, we will only deal with movable curve classes $\gamma$.
\end{remark}
There is a simple connection between slope stability and multi-Gieseker-stability which follows from Riemann-Roch, cf.~Example \ref{ex:RRochII}:

\begin{lemma}[Comparison between slope and multi-Gieseker-stability]\label{lem:slopeimpliesmulti}
Suppose $X$ is smooth of dimension $d$ and $(\underline{L} ,\sigma_1,\ldots,\sigma_{j_0})$ is a stability parameter, and let
$$ \gamma = \sum\nolimits_j \sigma_j c_1(L_j)^{d-1}.$$
Then, for any torsion-free coherent sheaf $E$ the following implications hold
\begin{center}
slope stable with respect to $\gamma$ $\Rightarrow$ stable with respect to $\sigma$ $\Rightarrow$ semistable with respect to $\sigma$ $\Rightarrow$ slope semistable with respect to $\gamma$.
\end{center}
\end{lemma}
\section{Chamber structures}\label{sec:chamber}

We next show that any set of stability parameters admits a chamber decomposition such that the notion of stability is unchanged within a chamber.  For this we require that $X$ be a projective integral scheme.  

\begin{definition}
  Let $\Sigma\subset \mathbb R^{j_0}\setminus \{0\}$ be convex.  A \emph{chamber structure} on $\Sigma$ consists of a collection $\{ W_j\}_{j\in J}$ of real hypersurfaces in $\Sigma$ called \emph{walls}.  Given such data, we call a subset $\mathcal C\subset \Sigma$ a \emph{chamber}  if for all $j\in  J$ either $\mathcal C\subset W_{j}$ or $\mathcal C\cap W_{j}=\emptyset$, and if in addition $\mathcal C$ is maximally connected with this property.  We say a chamber structure is \emph{linear} (resp.\ \emph{rational linear}) if all the hypersurfaces $W_j$ are linear (resp.\ rational linear).
\end{definition}
Now fix $\underline{L} $ as before and let $\tau\in B(X)_{\mathbb Q}$. This section is devoted to the proof of the following proposition.

\begin{proposition}[Existence of chamber structures on sets of stability parameters]\label{prop:weakchamberstructure} 
Let $X$ be a projective integral scheme and $\Sigma\subset (\mathbb R_{\ge 0})^{j_0}\setminus\{0\}$ be convex.   Then, for all integers $p$ the set $\Sigma$ admits a rational linear chamber structure cut out by finitely many walls such that if $\sigma',\sigma''$ belong to the same chamber $\mathcal{C}$, then
\begin{enumerate}
\item any $(p,\underline{L})$-regular torsion-free sheaf $E$ of topological type $\tau$  is $\sigma'$-(semi)stable if and only if it is $\sigma''$-(semi)stable, and
\item any two $(p,\underline{L})$-regular torsion-free sheaves $E$ and $E'$ of topological type $\tau$ that are semistable with respect to both $\sigma'$ and $\sigma''$ are $S$-equivalent with respect to $\sigma'$ if and only if they are $S$-equivalent with respect to $\sigma''$.
\end{enumerate}
\end{proposition}

\begin{remark}
  In general, this chamber structure will depend on the integer $p$ chosen.  However, if one is in the situation that the set of semistable sheaves in question is bounded, then one should take $p$ to be large enough so that all such sheaves are $(p,\underline{L})$-regular. 
\end{remark}

The above proposition is to be understood as including the possibility that $\sigma'$ or $\sigma''$ are irrational.  Since each wall in this chamber structure is a rational hyperplane, any chamber $\mathcal C$ contains a rational point; i.e., $\mathcal C \cap (\mathbb Q_{\ge 0})^{j_0}$ is non-empty.  Thus, we see that nothing is lost by restricting to rational $\sigma$, as made precise in the following.

\begin{corollary}\label{cor:irrational}
  Let $\sigma'\in \Sigma$.  Then, there exists a $\sigma''\in \Sigma \cap (\mathbb Q_{\ge 0})^{j_0}\setminus\{0\}$ such that any $(p,\underline{L})$-regular torsion-free sheaf of topological type $\tau$ is (semi)stable with respect to $\sigma'$ if and only if it is (semi)stable with respect to $\sigma''$, and similarly for $S$-equivalence classes.
\end{corollary}

Turning to the proof of Proposition~\ref{prop:weakchamberstructure}, fix an integer $p$.   We will use the notation of Section~\ref{section:stabilitywrtseveralpolarisations}, so that the multi-Hilbert polynomial of a sheaf $E$ of dimension $d$ is written as
$$P_E^{\sigma}(m) = \sum_{i=0}^d \alpha_i^{\sigma}(E) \frac{m^i}{i!},\;\;\text{ where } \alpha_i^{\sigma}(E) = \sum\nolimits_j \sigma_j \alpha_i^{L_j}(E).$$ 
Consider next the family of subsheaves 
\begin{equation*}
  \mathcal S := \left\{F \left|\begin{split}  & E \text{ is torsion-free, }
  (p,\underline{L})\text{-regular, of top.~type } \tau \text{, and }  F \text{ is a  }\\ & \text{ saturated subsheaf of $E$ with }
  \hat{\mu}^{\sigma}(F) \geq \hat{\mu}^{\sigma}(E) \text{ for some
  }\sigma \in \Sigma 
\end{split} \right.\right\}.
\end{equation*}

\begin{lemma}
The set $\mathcal{S}$ is bounded.
\end{lemma}
\begin{proof}
 If $F \in \mathcal{S}$, then by Lemma \ref{lem:slope} there exists some $j \in \{1, \ldots, j_0\}$ such that $\hat{\mu}^{L_j}(F) \ge \hat{\mu}^{\sigma}(E)$.
Since the set of $(p,\underline{L})$-regular sheaves of topological type $\tau$ is bounded, and since the quotients $E/F$ are torsion-free or zero, it follows from Grothendieck's Lemma \cite[Lem.~1.7.9]{Bible} that $\mathcal{S}$ is contained in a finite union of bounded families, hence is itself bounded.
\end{proof}
Now, for each $1 \leq i \leq d-1$ and each $F\in \mathcal S$, set
\begin{equation}\label{eq:wif}
 W_{i,F} := \Bigl\{ \sigma\in \Sigma : \sum\nolimits_j \sigma_j \Bigl( \frac{\alpha_i^{L_j}(F)}{\rank(F)}  - \frac{\alpha_i^{L_j}(E)}{\rank(E)}\Bigr) =0\Bigr\}.
\end{equation}
Note that as $X$ is integral and the sheaves $E$ and $F$ are torsion-free, the ranks of $E$ and $F$ are defined independently of $L_j$, see Remark \ref{rmk:rank}.

Clearly, $W_{i,F}$ is either empty, all of $\Sigma$, or a rational hyperplane in $\Sigma$.  If $W_{i,F}$ is empty or all of $\Sigma$, then it is discarded.    We observe that $W_{i,F}$ depends only on the topological type of $F\in \mathcal S$, and thus we have a finite number of (non-trivial) rational linear walls as $F$ varies over the bounded set $\mathcal S$ and as $i$ varies between $1$ and $d-1$. In this way, we obtain a rational linear chamber structure on $\Sigma$.

\begin{lemma}\label{lem:weakchamberhelper}
Suppose that $E$ is torsion-free of topological type $\tau$ and $(p,\underline{L})$-regular, and $F\subset E$ with $F\in \mathcal S$.  Let $\sigma'$ and $\sigma''$ be points in a chamber $\mathcal C$.  Then, $p_{F}^{\sigma'} (\le) p_E^{\sigma'}$ if and only if $p_{F}^{\sigma''} (\le) p_E^{\sigma''}$.
\end{lemma}
\begin{proof}
We deal with the case of non-strict inequality.  Suppose for contradiction this is not the case, so swapping $\sigma'$ and $\sigma''$ if necessary
\begin{align}\label{eq:chamber}
p^{\sigma'}_F\le p^{\sigma'}_E \;\;\;\;\;\; \text{ and }\;\;\;\;\;\;
p^{\sigma''}_F > p^{\sigma''}_E.
\end{align}
Now write
$$ p^{\sigma}_F - p^{\sigma}_E  = \sum_{i=0}^d c^{\sigma}_i \frac{m^i}{i!} \;\;\;\text{for some  }c^{\sigma}_i \in \mathbb{R}.$$
Let $i$ be the smallest integer such that $c^{\sigma'}_j =c^{\sigma''}_j=0$ for all $j>i$.  Then, by \eqref{eq:chamber} and by the definition of ordering of polynomials we get $c^{\sigma'}_i\le 0$  and $c^{\sigma''}_i\ge 0$ (but not both being equal to zero by choice of $i$).  
Now, let $f\colon \Sigma\to \mathbb R$ be the linear function given by
$$ f(\sigma) = \sum\nolimits_j \sigma_j \left(\frac{\alpha_i^{L_j}(F)}{\rank(F)} - \frac{\alpha_i^{L_j}(E)}{\rank(E)}\right),$$
so by definition $W_{i,F} = \{ \sigma : f(\sigma)=0\}$.  
Thus, we have 
\begin{align*}
  c_i^{\sigma} = \frac{\sum_j \sigma_j \alpha_i^{L_j}(F)}{\sum_j \sigma_j \alpha_d^{L_j}(F)} - \frac{\sum_j \sigma_j \alpha_i^{L_j}(E)}{\sum_j \sigma_j \alpha_d^{L_j}(E)}
=C f(\sigma),
\end{align*}
where $C: = \bigl(\sum_j \sigma_j \alpha_d^{L_j} (\mathcal{O}_X)\bigr)^{-1}>0$.  Since either $c_i^{\sigma'}$ or $c_i^{\sigma''}$ are non-zero, we conclude that either $\sigma'$ or $\sigma''$ is not in $W_{i,F}$, and thus $\mathcal C$ is not contained in $W_{i,F}$.    On the other hand, there must be a point $\tilde{\sigma}$ on the line segment between $\sigma'$ and $\sigma''$ such that $f(\tilde{\sigma})=0$.  But $\mathcal C$ is convex so contains this line segment,  which implies $\mathcal C\cap W_{i,F}$ is non-empty, and this is absurd since $\mathcal C$ is meant to be a chamber.  The case for strict inequality is proved in precisely the same way.
\end{proof}

\begin{proof}[Proof of Proposition~\ref{prop:weakchamberstructure}]
Let $\sigma'$ and $\sigma''$ be in the same chamber, and suppose that $E$ is of  topological type $\tau$, $(p,\underline{L})$-regular and semistable with respect to $\sigma'$.  Let $F\subset E$ be saturated.  If $\hat{\mu}^{\sigma''}(F)<\hat{\mu}^{\sigma''}(E)$, then $F$ does not destabilise $E$ with respect to $\sigma''$.  Otherwise, $F\in \mathcal S$ and so Lemma \ref{lem:weakchamberhelper} implies $p_F^{\sigma''}\le p_E^{\sigma''}$.  Thus, $E$ is also semistable with respect to $\sigma''$.  The proof of the statement about $S$-equivalence is the same, for if $F\subset E$ is a saturated subsheaf that is destabilising with respect to $\sigma'$, by definition, we have $p_F^{\sigma'} = p_E^{\sigma'}$.  So, again by Lemma~\ref{lem:weakchamberhelper} we have $p_F^{\sigma''} = p_E^{\sigma''}$, and hence $F$ is also destabilising with respect to $\sigma''$.  Thus, any maximal chain of destabilising subsheaves (with respect to $\sigma'$) is such a maximal chain also when semistability is defined by $\sigma''$, so the corresponding graded objects are isomorphic.  At the same time, this proves the statement about stability, as if $E$ is semistable then it is stable if and only if $gr(E)$ is isomorphic to $E$.
\end{proof}

\part{Construction of Moduli Spaces}

\section{Functorial approach to the moduli problem}\label{sec:categories}Following the ideas of \'Alvarez-C\'onsul--King presented in \cite{ConsulKing}, see also the survey \cite{ConsulKingSurvey}, we will embed the category of sheaves of interest into a category of representations for certain quivers. We first introduce in Section~\ref{subsect:quivers} the relevant concepts from the representation theory of quivers, and then prove the fundamental functorial embedding result, Theorem~\ref{thm:categoryembedding}, in Section~\ref{subsect:embeddingfunctor}. In Section~\ref{sec:families} we show this extends to flat families of sheaves, and this result in turn is used to identify the image of the embedding functor in the relevant category of representations.
\subsection{Quivers and their representations}\label{subsect:quivers}
We will use the standard notations used in representation theory of quivers, as fixed for example in \cite[Sect.~3]{King}.   We denote by $\mathbf{Vect}_k$ the category of vector spaces over a field $k$.

\subsubsection{The quiver $Q$}
Given a $j_0\in \mathbb N^+$ we define a labelled quiver 
$$Q = (Q_0, Q_1, h,t: Q_1 \to Q_0, H: Q_1 \to \mathbf{Vect}_k)$$ as follows.  Let 
\begin{equation*}
 Q_0 := \{v_i, w_j \mid i,j = 1, \dots, j_0\}
\end{equation*}
be a set of pairwise distinct vertices, and
\begin{equation}
Q_1 := \{\alpha_{ij} \mid i,j= 1, \dots, j_0 \},
\end{equation}
the set of arrows, whose heads and tails are given by
\begin{align*}
 h (\alpha_{ij}) = w_j, \\
 t (\alpha_{ij}) = v_i.
\end{align*}
The arrows will be each labelled by a vector space, encoded by a function $H: Q_1 \to \mathbf{Vect}_k$ written as $H(\alpha_{ij}) = H_{ij}$, which will be fixed later.  This quiver can be pictured as follows (where for better readability we restrict to the case $j_0=3$):
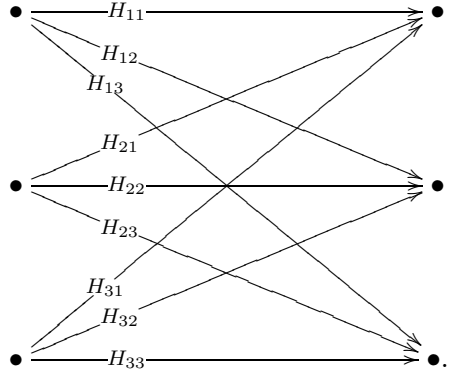
\[\begin{xymatrix}
{ \bullet \ar[rrrrr]|<<<<<<<<<<<<{H_{11}}\ar[rrrrrdd]|<<<<<<<<<<<<{H_{12}}\ar[rrrrrdddd]|<<<<<<<<<<<<{H_{13}}&  & & &  &\bullet \\
                                        &  &  & &  &     \\
  \bullet \ar[rrrrruu]|<<<<<<<<<<<<{H_{21}} \ar[rrrrr]|<<<<<<<<<<<<{H_{22}} \ar[rrrrrdd]|<<<<<<<<<<<<{H_{23}}& & & & & \bullet \\
                            &   & &  &   &    \\
  \bullet \ar[rrrrruuuu]|<<<<<<<<<<<<{H_{31}} \ar[rrrrr]|<<<<<<<<<<<<{H_{33}} \ar[rrrrruu]|<<<<<<<<<<<<{H_{32}}& & & & &\bullet    .
}
  \end{xymatrix}
\]
\begin{remark}
In the ``Method of construction''-section of the introduction, we gave a heuristic argument as to why it is a natural idea to consider this particular quiver. One could probably also work with a smaller quiver that does not include all of the ``diagonal'' arrows.  However, we prefer to work with the quiver defined above, as it reflects the a priori symmetry between the line bundles $L_j$ in the variation problem for Gieseker moduli spaces.
\end{remark}

\subsubsection{Representations of $Q$}\label{subsubsect:Qreps}A \emph{representation} of $Q$ over a field $k$ is a collection $V_i, W_j, i,j= 1, \ldots, n$ of $k$-vector spaces together with $k$-linear maps $\phi_{ij}: V_i \otimes H_{ij} \to W_j$.

Let now $X$ be a projective scheme of finite type over an algebraically closed field $k$ of characteristic zero. Here, and henceforth, given line bundles $L_j$ on $X$ for $j=1,\ldots,j_0$ and integers $m > n$, we consider the sheaf
\[ T := \bigoplus\nolimits_{j= 1}^{j_0} L_j^{-n}  \oplus L_j^{-m}, \]
together with the finite-dimensional $k$-algebra 
\[A := L \oplus H \subset \mathrm{End}_X(T)\]
generated by the projection operators onto the summands $L_i^{-n}$ and $L_j^{-m}$ of $T$ (collected in the subalgebra $L$) and 
\begin{equation}\label{eq:definitionOfH}
 H = \bigoplus\nolimits_{i,j=1}^{j_0} H_{ij} := \bigoplus\nolimits_{i,j=1}^{j_0} H^0(X, L_i^{-n}\otimes L_j^{m}) =  \bigoplus\nolimits_{i,j=1}^{j_0}\Hom(L_j^{-m}, L_i^{-n}).
\end{equation}
Note that $T$ is a left $A$-module and that $H$ is an $L$-bimodule.

The category of representations of the quiver $Q$ with $H(\alpha_{ij}) =  H_{ij}$ is equivalent to the category of modules over $A$. An $A$-module structure on a vector space $M$ can be specified by a direct sum decomposition $M =  \bigoplus_{j=1}^{j_0} V_j \oplus W_j$, together with a right $L$-module map $\alpha: M \otimes_L H \to M$. By abuse of notation, we will sometimes write such an $A$-module $M$ only as $M = \bigoplus_{j=1}^{j_0} V_j\oplus W_j$ representing the decomposition of $M$ under the action $L$ and suppressing the action of $H$. Similarly, the left $A$-module structure on $T$ is given by the decomposition \eqref{eq:definitionOfH} and the multiplication maps 
\[\mu_{ij}: H_{ij} \otimes L_j^{-m} \to L_i^{-n},\] or equivalently the left $L$-module map $\mu: H \otimes_L T \to T$.  If  $M =  \bigoplus_{j=1}^{j_0} V_j \oplus W_j$ is an $A$-module then a submodule $M'\subset M$ is given by a direct sum $M' = \bigoplus_{j=1}^{j_0} V'_j \oplus W'_j$ such that $V'_j\subset V_j$ and $W_j'\subset W_j$ with the additional property that the image of $V'_i\otimes H_{ij}$ under the linear map $V_i\otimes H_{ij}\to W_j$ is contained in $W_j'$.

The representations of most interest to us are the ones of the form $\Hom(T,E)$, where $E$ is a coherent sheaf on $X$. On the one hand, this naturally comes equipped with a right-module structure over $A \subset \Hom(T,T)$, given by (pre-)composition of maps. On the other hand, we have the obvious decomposition
\[\Hom(T,E) = \bigoplus\nolimits_{j} H^0(E\otimes L_j^{n}) \oplus  H^0(E\otimes L_j^{m}),\]
together with the natural multiplication maps $H^0(E \otimes L_i^{n}) \otimes H_{ij} \to H^0(E \otimes L_j^m)$.

\subsection{Stability of quiver representations}\label{subsubsect:stabilityandmodulispaces}
Next, in order to construct moduli spaces paramet\-rising representations of our given quiver, we introduce a notion of stability.
\begin{definition}[Dimension vector]
Let $M = \bigoplus_j V_j\oplus W_j$ be an $A$-module.  We call the vector $$\underline{d} := (\dim V_1, \dim W_1,\dots, \dim V_{j_0}, \dim W_{j_0})=: (d_{11}, d_{12}, \ldots, d_{{j_0 1}}, d_{{j_0 2}})$$ the \emph{dimension vector} of $M$. 
\end{definition}

We wish to define the notion of stability of $A$-modules of a given dimension $\underline{d}$, where $d_{j1}$ and $d_{j2}$ are strictly positive for all $j$.  To do so, fix  $\sigma = (\sigma_1,\ldots,\sigma_{j_0})$ with $\sigma_j\in \mathbb R_{\ge 0}$ not all equal to zero.  Define a vector $\theta_\sigma = (\theta_{11}, \theta_{12}, \ldots, \theta_{{j_0 1}}, \theta_{{j_0 2}})$ by
\begin{equation}\label{eq:smallthetadef}
 \theta_{j1} := \frac{\sigma_j}{\sum \sigma_i d_{i1}}\text{ and }\theta_{j2} := \frac{- \sigma_j}{\sum \sigma_i d_{i2}}\text{ for }j= 1, \ldots, j_0.
\end{equation}
and for any $A$-module $M' = \bigoplus V_j'\oplus W_j'$ we set
\begin{equation}\label{eq:bigthetadef}
 \theta_\sigma(M') = \sum\nolimits_{j} \theta_{j1} \dim V_j' + \sum\nolimits_{j} \theta_{j2} \dim W_j',
\end{equation}
which makes $\theta_{\sigma}$ an additive function from the set $\mathbb{Z}^{2 j_0}$ of possible dimension vectors to $\mathbb R$.  Note that if $M$ is an $A$-module of dimension vector $\underline{d}$,  $\theta_{\sigma}(M)=\sum_j (\theta_{j1} d_{j1} + \theta_{j2} d_{j2})=0$.

\begin{definition}[Semistability for $A$-modules]\label{def:semistabilitymodule}
Let $M$ be an $A$-module with dimension vector $\underline{d}$.  We say that $M$ is \emph{(semi)stable (with respect to $\sigma$)} if for all proper submodules $M'\subset M$ we have $\theta_\sigma(M')(\le)0$.
\end{definition}
This definition of stability for $A$-modules is that of King \cite{King}, generalised here to allow the possibility that $\sigma_j$ are not necessarily integral. Observe that if $\sigma$ is rational, $\theta_{\sigma}$ takes values in $\mathbb Q$, and hence, by clearing denominators, we can arrange it to take values in $\mathbb Z$. This brings us back to the original setup of \cite{King} and will allow us to apply the results proven there.

Every $\sigma$-semistable $A$-module has a Jordan-H\"older filtration with respect to $\theta_\sigma$, cf.~\cite[p.~521/522]{King}, and we call two modules \emph{$S$-equivalent} if the graded modules associated to the respective filtrations are isomorphic.

\begin{remark}
 For an interpretation of the above discussion in terms of stability conditions on the abelian category of representations of a given quiver, see \cite[Sect.~3.4]{Ginzburg}.  We observe that the assignment $\sigma \mapsto \theta_\sigma$ is not in general linear, which will be relevant when we discuss the variation of moduli spaces in the sequel to this paper.
\end{remark}

\subsection{Moduli spaces of quiver representations}\label{subsect:quivermodulispaces}

The main result of \cite{King} links the notion of stability discussed above to GIT-stability. We have taken the following discussion of King's results from \cite[Sect.~3.5]{Ginzburg}. Given a dimension vector $\underline{d} \in \mathbb{Z}^{2 j_0}_{> 0}$, we consider the reductive group $G_{\underline{d}} := \prod_{j} (GL_{d_{j1}}(k) \times GL_{d_{j2}}(k))$ and the $G_{\underline{d}}$-module
\[\mathrm{Rep}(Q, \underline{d}) := \bigoplus\nolimits_{i,j} \mathrm{Hom}_k (k^{d_{i1}} \otimes H_{ij}, k^{d_{j2}}), \]
on which $G_{\underline{d}}$ acts by ``base change'' automorphisms. Now, given an integral vector $\theta \in \mathbb{Z}^{2 j_0}$, we introduce a rational character 
\[\chi_{\theta}: G_{\underline{d}} \to k^*,\;\; g = (g_{j1}, g_{j2})_{j= 1, \ldots, j_0} \mapsto \prod\nolimits_{j} (\det (g_{j1})^{-\theta_{j1}}\cdot\det (g_{j2})^{-\theta_{j2}}).\] 
This character defines a linearisation of the $G_{\underline{d}}$-action in the trivial line bundle over the affine space $\mathrm{Rep}(Q, \underline{d})$. Note that the character $\chi_{\theta}$ vanishes on the subgroup $k^* \subset G_{\underline{d}}$ of diagonally embedded invertible scalar matrices (which acts trivially on $\mathrm{Rep}(Q, \underline{d})$ if and only if $\sum_j (\theta_{j1} d_{j1} + \theta_{j2} d_{j2}) = 0$).

The character $\chi_\theta$ defines a set of GIT-semistable points $\mathrm{Rep}(Q, \underline{d})^{\chi_\theta\text{-ss}}$ and a corresponding GIT-quotient $\pi:\mathrm{Rep}(Q, \underline{d})^{\chi_\theta\text{-ss}} \to \mathrm{Rep}(Q, \underline{d})^{\chi_\theta\text{-ss}}\hq G$. We call two points $p,p' \in \mathrm{Rep}(Q, \underline{d})^{\chi_\theta\text{-ss}}$ \emph{GIT-equivalent} if and only if the points $\pi(p)$ and $\pi(p')$ agree.

Using the notation introduced above, the fundamental result that compares GIT-stability and semistability for representations of quivers can now be stated as follows.
\begin{theorem}[Prop.~3.1 and Prop.~3.2 of \cite{King}]\label{thm:quivermodulicitation}
 For any dimension vector $\underline{d}$ and any $\theta \in \mathbb{Z}^{2j_0}$ such that $\sum_j (\theta_{j1} d_{j1} + \theta_{j2} d_{j2}) = 0$, we have
\begin{enumerate}
 \item A representation in $\mathrm{Rep}(Q, \underline{d})$ is $\chi_{\theta}$-GIT-(semi)stable if and only if every submodule $M'$ of the corresponding $A$-module $M$ satisfies $\theta(M') (\leq) 0$.
 \item If $\theta_\sigma$ is computed from a given $\sigma$ as explained above, then a representation in $\mathrm{Rep}(Q, \underline{d})$ is $\chi_{\theta}$-GIT-(semi)stable if and only if the corresponding $A$-module $M$ is (semi)stable with respect to $\sigma$ in the sense of Definition~\ref{def:semistabilitymodule}.
 \item Two $\chi_{\theta}$-semistable representations are GIT-equivalent if and only if they are $S$-equi\-va\-lent.
\end{enumerate}
\end{theorem}

\begin{remark}
Later we will deal with fractional characters of $G_{\underline{d}}$. The discussion naturally extends to this slightly more general setup.
\end{remark}

\subsection{The embedding functor}\label{subsect:embeddingfunctor}
We will now show that a category of sufficiently regular sheaves embeds into the category of representations of the quiver $Q$. We refer the reader to \cite[Chap.~4]{MacLane} for the basics concerning adjoint functors used below.

  Let $X$ be a projective scheme of finite type over an algebraically closed field $k$ of characteristic zero, and $n$ a non-negative natural number.  Suppose that $\underline{L}= (L_1,\ldots,L_{j_0})$ where each $L_j$ is a very ample line bundle on $X$.  We also fix a $\tau\in B(X)_{\mathbb Q}$

\begin{theorem}[Embedding regular sheaves into the category of representations of $Q$]\label{thm:categoryembedding}
For $m \gg n$, the functor \[\Hom(T, -): \mathbf{mod}\text{-}\mathcal{O}_X \to \mathbf{mod}\text{-}A\] is fully faithful on the full subcategory of $(n,\underline{L})$-regular sheaves of topological type $\tau$. In other words, if $E$ is an $(n,\underline{L})$-regular sheaf of topological type $\tau$, the natural evaluation map $\varepsilon_E: \Hom(T, E) \otimes_A T \to E$ is an isomorphism. 
\end{theorem}
\begin{proof}
We first describe how to construct the tensor product $M \otimes_A T$ for a given $A$-module $M$. Similar to the case of the Kronecker quiver spelled out in \cite[Sect.~3.2]{ConsulKing}, one shows that $M \otimes_A T$ is constructed as the cokernel of the map
\[\begin{xymatrix}
   {M\otimes_L  H \otimes_L T \ar[rrr]^{1 \otimes \mu - \alpha \otimes 1} &  & &  M \otimes_L T.
}
  \end{xymatrix}
\]
Writing out the $L$-module structures of $M$ and $T$ explicitly as direct sum decompositions yields the following exact sequence
\begin{equation}\label{eq:tensordisentangled}\bigoplus\nolimits_{i,j} V_i \otimes H_{ij} \otimes L_j^{-m}  \longrightarrow \bigoplus\nolimits_{j} \left(\begin{aligned}
                                                                               V_j &\otimes L_j^{-n}\\
										&\,\,\oplus\\
									      W_j &\otimes L_j^{-m} 
                                                                              \end{aligned}\right)
  \longrightarrow M \otimes_A T \longrightarrow 0\end{equation}
Now, let $E$ be an  $(n,\underline{L})$-regular sheaf of topological type $\tau$. The $(n,\underline{L})$-regularity implies that $E$ has presentations 
\begin{equation}\label{eq:evaluationsurjective}0 \to F_i \to V_i \otimes L_i^{-n} \to E \to 0 \quad \quad \text{for }i = 1, \ldots, j_0,
\end{equation}
where $V_i = H^0(E\otimes L_i^n)$. 

As the set of $(n,\underline{L})$-regular sheaves of topological type $\tau$ is bounded, the set of the corresponding $F_i$'s  is bounded as well, and hence for all $m \gg n$ the $F_i\otimes L_j^{m}$, $i, j = 1, \ldots, j_0$, are globally generated. Consequently, we obtain surjections $U_{ij} \otimes L_j^{-m} \twoheadrightarrow F_i$, where $U_{ij}=  H^0(F_i\otimes L_j^{m})$ are the appropriate spaces of sections. 

On the other hand, twisting \eqref{eq:evaluationsurjective} with $L_j^m$ yields the short exact sequence
\[
 0 \to F_i \otimes L_j^m \to  V_i \otimes L_i^{-n}  \otimes L_j^m \to E \otimes L_j^m \to 0.
\]
Again using boundedness of the $F_i$'s, increasing $m$ if necessary, and recalling that $H_{ij} = H^0 (L_i^{-n}\otimes L_j^m)$, we therefore obtain short exact sequences
\begin{equation}\label{eq:linearsequence}
 0 \to U_{ij} \to V_i \otimes H_{ij} \to W_j \to 0,
\end{equation}
where $W_j = H^0 (E\otimes L_j^m)$.

By putting \eqref{eq:evaluationsurjective} and \eqref{eq:linearsequence}$\otimes L_j^{-m}$ together, for each pair $(i,j)$ we obtain a commutative diagram of exact sequences
\begin{equation}\label{eq:bigdiagram}
\begin{gathered}
 \begin{xymatrix}
  {  0 \ar[r]&  F_i  \ar[r] &      V_i\otimes L_i^{-n} \ar[r] &  E  \ar[r] & 0 \\
     0\ar[r]  &  U_{ij} \otimes L_j^{-m}  \ar@{->>}[u]\ar[r]&   V_i \otimes H_{ij}\otimes L_j^{-m} \ar[u]\ar[r]&  W_j \otimes L_j^{-m}  \ar[u]\ar[r]& 0 .
}
 \end{xymatrix}
\end{gathered}
\end{equation}
We conclude by a diagram chase that the square on the right hand side is a pushout; i.e., we have an exact sequence 

\begin{equation}\label{eq:rowquotient}
\bigoplus\nolimits_{j} V_j \otimes H_{jj} \otimes L_j^{-m}   \longrightarrow \bigoplus\nolimits_{j} \left(\begin{aligned}
                                                                               V_j &\otimes L_j^{-n}\\
										&\,\,\oplus\\
									      W_j &\otimes L_j^{-m}.
                                                                              \end{aligned}\right)
  \longrightarrow E^{\oplus j_0} \longrightarrow 0.\end{equation}

We are now carrying out the construction of $\Hom(T,E)\otimes_A T$, as written out in \eqref{eq:tensordisentangled}. First, we conclude from \eqref{eq:rowquotient}  that we have an exact sequence
\begin{equation}\bigoplus\nolimits_{i\neq j} V_i \otimes H_{ij} \otimes L_j^{-m}  \longrightarrow E^{\oplus j_0}
  \longrightarrow\Hom(T,E)\otimes_A T \longrightarrow 0,\end{equation}
cf.~\cite[end of the proof of Thm.~3.4]{ConsulKing}. For each $(i,j)$, apply \eqref{eq:bigdiagram} once more, to conclude that the image of 
$V_i \otimes H_{ij} \otimes L_j^{-m} $ in $E^{\oplus j_0}$ is equal to the $(i,j)$-th anti-diagonal 
\[\Delta_{ij} := \{(0,\ldots ,0 ,e,0,\ldots, 0,-e,0,\ldots,0)\mid e\in E\} \cong E.\]
 Consequently, the image of $\bigoplus_{i\neq j} V_i \otimes H_{ij} \otimes L_j^{-m}$ is equal to the (direct) sum
$\sum_{j=0}^{j_0-1} \Delta_{j, j+1} \subset  E^{\oplus j_0}$.
We infer that the natural evaluation map from $E^{\oplus j_0}$ to $E$ induces an isomorphism $\Hom(T,E)\otimes_A T \cong E$ , as claimed.
\end{proof}

\subsection{Families of sheaves, families of representations, and the image of the embedding functor}\label{sec:families}
In this section we follow the exposition of \cite[Sect.~4]{ConsulKing} closely.

Let $S$ be a scheme. A \emph{flat family $\mathscr{E}$ over $S$ of sheaves on $X$} is a sheaf $\mathscr{E}$ on $X \times S$ that is flat over $S$. On the other hand, a \emph{flat family $\mathscr{M}$ over $S$ of right $A$-modules} is a sheaf $\mathscr{M}$ of right modules over the sheaf of algebras $\mathscr{A} := \mathcal{O}_S \otimes A$ on $S$ that is locally free as a sheaf of $\mathcal{O}_S$-modules.

Let $\pi\colon X \times S \to S$ and $p_X\colon X \times S \to X$ be the canonical projections. The adjoint pair formed by $\Hom(T, -)$ and $-\otimes_A T$ extends to an adjoint pair of functors between the category $\mathbf{mod}\text{-} A \otimes \mathcal{O}_S$ of sheaves of right $A$-modules on $S$ that are coherent as $\mathcal{O}_S$-modules and the category $\mathbf{mod}\text{-}\mathcal{O}_{X \times S}$ of sheaves on $X \times S$:
\begin{equation}
\begin{gathered}
 \begin{xymatrix}
  {
\mathbf{mod}\text{-} A \otimes \mathcal{O}_S \ar[d]_{- \otimes_\mathscr{A} T}\\
\mathbf{mod}\text{-}\mathcal{O}_{X \times S}\ar[u]_{\sheafHom_X(T, -)}
}
 \end{xymatrix}
\end{gathered}
\end{equation}
Here, for a sheaf $\mathcal{E}$ on $X\times S$ and a sheaf $\mathcal{M}$ of right $A$-modules on $S$ we are using the abbreviations 
\[\sheafHom_X(T, \mathcal{E}):= \pi_* \bigl(\sheafHom_{X \times S}(p_X^*T, \mathcal{E}) \bigr) \cong \pi_* \bigl(\mathscr{E} \otimes_{\mathcal{O}_{X \times S}} T^\vee  \bigr)\]
and 
\[\mathscr{M} \otimes_\mathscr{A} T := \pi^* \mathscr{M} \otimes_{\mathscr{A}} p_X^* T.\]
\begin{proposition}[Family version of the embedding functor is fully faithful]\label{prop:embedding_familyversion}

Let $n$ be a natural number.  Then, for  $m\gg n$ as in Theorem~\ref{thm:categoryembedding} the following holds: if $S$ is any scheme, $\sheafHom_X(T, -)$ is a fully faithful functor from the full subcategory of $\mathbf{mod}$-$\mathcal{O}_{X\times S}$ consisting of $S$-flat families of $(n, \underline{L} )$-regular sheaves of topological type $\tau$ to the full subcategory of $\mathbf{mod}$-$\mathscr{A}$ consisting of $S$-flat families  of right $A$-modules.
\end{proposition}
\begin{proof}
 The proof is almost the same as the one of \cite[Prop.~4.1]{ConsulKing}; we just have to make some small adjustments due to the fact that we are dealing with several line bundles at the same time. 

In order to see that flatness is preserved by the functor $\sheafHom_X(T, -)$, it suffices to know that for every sheaf $\mathscr{E}_s$ in an $S$-flat family $\mathscr{E}$ we have
\begin{equation}\label{eq:vanishingforflatness}
 H^1(\mathscr{E}_s \otimes L_j^n) = H^1(\mathscr{E}_s \otimes L_j^m)= 0 \quad \quad  \text{for all}\,j = 1, \ldots, j_0.
\end{equation}
Indeed, once we have this, it follows that $R^1\pi_*\bigl(\sheafHom_{X \times S}(p_X^*T, \mathscr{E}) \bigr)$ vanishes, which in turn implies by \cite[Thm.~12.11(b)]{Hartshorne} that $\sheafHom_X(T, \mathscr{E})$ is locally free. The required vanishing \eqref{eq:vanishingforflatness} follows from the assumption that each $\mathscr{E}_s$ is $(n, \underline{L} )$-regular (and hence $(m, \underline{L} )$-regular for every $m\geq n$ by \cite[Lem.~3.2(1)]{ConsulKing}) and \cite[Lem.~3.2(2)]{ConsulKing}. We then conclude the proof as in \cite[proof of Prop.~4.1]{ConsulKing} by noting that we may now apply Theorem~\ref{thm:categoryembedding} fibrewise, using general results about cohomology and flat base extensions \cite[Sect.~III.12]{Hartshorne}.
\end{proof}
Next, we are going to identify the image of the embedding functor; in fact, it forms a locally closed subfunctor of $\mathbf{mod}$-$\mathscr{A}$. This is the content of the following proposition.
\begin{proposition}[Identifying the image of the embedding functor]\label{prop:identifyingtheimage}
For $m\gg n$ as in Theorem~\ref{thm:categoryembedding} the following holds: If $B$ is any scheme and $\mathscr{M}$ a $B$-flat family of right-$A$-modules of dimension vector $$\underline{d} = \bigl(P_1(n), P_1(m), \ldots, P_{j_0} (n), P_{j_0} (m)\bigr),$$ then there exists a (unique) locally closed subscheme $\iota\colon B^{[reg]}_{\tau} \hookrightarrow B$
with the following properties.
\begin{enumerate}
 \item[(a)] $\iota^* \mathscr{M} \otimes_\mathscr{A} T$ is a $B^{[reg]}_{\tau}$-flat family of $(n,\underline{L})$-regular sheaves on $X$ of topological type $\tau$, and the unit map
 \[\eta_{\iota^*\mathscr{M}}\colon  \iota^* \mathscr{M} \to \sheafHom_X(T, \iota^* \mathscr{M} \otimes_\mathscr{A} T)\]
is an isomorphism.
 \item[(b)] If $\sigma: S \to B$ is such that $\sigma^* \mathscr{M} \cong \sheafHom_X (T,\mathscr{E})$ for an $S$-flat family $\mathscr{E}$ of $(n,\underline{L})$-regular sheaves on $X$ of topological type $\tau$, then $\sigma$ factors through $\iota\colon B^{[reg]}_{\tau} \hookrightarrow B$ and $\mathscr{E} \cong \sigma^*\mathscr{M} \otimes_\mathscr{A} T$.
\end{enumerate}

\end{proposition}
\begin{proof}
The proof is almost literally the same as the one of \cite[Prop.~4.2]{ConsulKing} and so will not be repeated here.
We remark, however, that the cited proof actually gives a subscheme $B^{[reg]}$ that parametrises a flat family of $(n,\underline{L})$-regular sheaves with a given set of Hilbert polynomials with respect to $L_1, \dots, L_{j_0}$. There may be several topological types that lead to the same such set of polynomials, but since the topological type is constant in flat families over connected bases, see Remark~\ref{rmk:Todd}, we are free to select just those components of $B^{[reg]}$ that correspond to sheaves of topological type $\tau$, giving the desired subscheme $B^{[reg]}_{\tau}$.
\end{proof}
\section{Boundedness}\label{section:boundedness}

In this section $X$ will denote a smooth $n$-dimensional projective variety over an algebraically closed field $k$ which we can allow to be of arbitrary characteristic. Again fix a vector $\underline{L} = (L_1,\ldots,L_{j_0})$ of ample line bundles and a topological type $\tau \in B(X)_\mathbb{Q}$. 

\begin{definition}[Bounded sets of stability parameters]
  We say that a set $\Sigma\subset (\mathbb R_{\ge 0})^{j_0}\setminus\{0\}$ of stability parameters is  \emph{bounded} (with respect to the data $\tau,\underline{L} $) if the set of all sheaves of topological type $\tau$ that are semistable with respect to some $\sigma\in \Sigma$ is bounded. We say that an individual stability parameter $\sigma$ is \emph{bounded} if the singleton $\{\sigma\}$ is bounded.
\end{definition}

We will give several conditions under which this boundedness holds. The main situations where our results apply are gathered in Corollaries~\ref{thm:positiveimpliesstronglybounded} and~\ref{cor:boundednesssurfaceorpicard2} at the end of the section. We start with a general fact concerning bounded sets of sheaves.

\begin{proposition}\label{prop:stronglybounded}
Let  $\mathcal S$ be a set of isomorphism classes of pure $d$-dimen\-sional coherent sheaves on $X$ all of topological type $\tau$.  Then the following are equivalent:
\begin{enumerate}
\item The set $\mathcal S$ is bounded.
\item For all $j$, the quantity $\hat{\mu}^{L_j}_{\max}(E)$ is bounded uniformly over all $E\in \mathcal S$.
\item For some $j$, the quantity $\hat{\mu}^{L_j}_{\max}(E)$ is bounded uniformly over all $E\in \mathcal S$.
\end{enumerate}
\end{proposition}
\begin{proof} 
Clearly (2) implies (3), and  the statement that (3) implies (1) is classical (see \cite[Thm.~3.3.7]{Bible} in characteristic 0 and \cite[Thm.~4.2]{Langer} in arbitrary characteristic). It remains to show that (1) implies (2). So, assume that (1) holds and fix some $j$. As $\mathcal{S}$ is bounded, there exists a scheme $S$ of finite type and a sheaf $\mathscr E$ on $X\times S$ such that each element of $\mathcal{S}$ is contained in the family $\mathscr{E}$ (up to isomorphism).   We have to show that the quantity $\mu_{\max}^{L_j}(E_s)$ is bounded uniformly over all closed points $s\in S$. We proceed by induction on $n=\dim S$, the case $n=0$ being trivial. 

We first claim that without loss of generality we may assume that $\mathscr E$ is flat over $S$ and that $S$ is integral.  To see this, first take a flattening stratification \cite[Lemma 2.1.6]{Bible} to write $S=\bigcup S_i$ as a finite union of locally closed subschemes $S_i$ over which $\mathscr E$ is flat.  So, by replacing $S$ with one of the $S_i$ we may assume that $\mathscr E$ is flat.  Then, pulling back to $S_{\text{red}}$ (which has the same closed points) one may assume that $S$ is reduced, and finally pulling back to an irreducible component we may assume that $S$ is integral.

We now use the relative Harder-Narashimhan filtration to deduce there is a dense open $U_j\subset S$ so that $\hat{\mu}_{\max}^{L_j}(E_s)$ is independent of $s\in U_j$. 
 In detail, by a variant of \cite[Thm.~2.3.2]{Bible}, there exists an integral scheme $T$ and a birational morphism $g\colon T\to S$ and a filtration 
$$ 0 = \HN_0(\mathscr{ E}) \subset \HN_1(\mathscr{ E})\subset \cdots \subset \HN_l(\mathscr E) = \mathscr{E}_T$$ 
such that (1) the factors $\HN_i(\mathscr E)/\HN_{i-1}(\mathscr E)$ are flat over $T$ for $i=1,\ldots,l$, and (2) there is a dense open subset $U'\subset T$ over which $g$ is an isomorphism, such that $\HN_{*}(\mathscr E)_{t} = g^* \HN_{*}(\mathscr E_{g(t)})$ for all closed points $t\in U'$.  Here, $\HN$ denotes the Harder-Narashimhan filtration taken with respect to the slope function $\hat{\mu}^{L_j}$.   By definition, the maximal destabilising subsheaf of $\mathscr{E}_{s}$ is $\HN_1(\mathscr{E}_s)$.  So, letting $U_j:=g(U')$ we see that if $s=g(t)\in U$, then $\mu_{\max}^{L_j}(\mathscr{ E}_s) = \mu^{L_j}(\HN_1(\mathscr{E}_s)) = \mu^{L_j}(\HN_1(\mathcal E)_t)$, which is independent of $t$ by flatness (and so independent over all $s\in U_j$), as claimed.

Now applying the induction hypothesis to the complement of $U_j$, which has strictly lower dimension than $S$, we conclude that $\mu_{\max}^{L_j}(\mathscr {E}_s)$ is also bounded over $s\in S\setminus\ U_j$, and thus bounded over all of $S$ as required.
\end{proof}

\begin{remark}[Relative Harder-Narasimhan filtration]
  In the above proof we have used the relative Harder-Narasimhan filtration taken with respect to the slope function $\hat{\mu}^{L_j}$.  The cited theorem \cite[Thm.~2.3.2]{Bible} is stated for Gieseker-stability, but the same proof works for slope stability (one merely has to replace their definition of $A_4$ with the set of those polynomials $P''\in A$ such that $\hat{\mu}(P'')\le\hat{\mu}(P)$, and use that all the basic properties proved for Gieseker-stability also hold for slope stability, cf.~\cite[Thm.~1.6.6]{Bible}).
\end{remark}

We need to introduce some notation. For simplicity and as all our boundedness statements will be clearly true when $X$ is a curve,  we shall assume from now on that  $X$ is $n$-dimensional with $n\ge 2$.
Letting $\Amp(X)_{\mathbb R}$ denote the ample cone of $X$ in $N^1(X)_\R:=N^1(X)\otimes_\Z \mathbb R$, we define the \emph{(strongly) positive cone} as 
$$ \Pos(X)_{\mathbb R} := \{ \gamma \in N_1(X)_\mathbb R \ | \  \gamma
 = D^{n-1} \text{ for some } D\in \Amp(X)_{\mathbb R}\}.$$
The cone $\Pos(X)_{\mathbb R}$ plays an important role in our boundedness considerations.  
Note that it is not convex in general. In order to study its convex hull $\Conv(\Pos(X)_\R)$, we are led to introduce some further ``positive cones''. Their basic properties will be derived from the following version of the Hodge Index Theorem for real ample classes. 

\begin{theorem}[Hodge Index Theorem]\label{thm:Hodge}
 Let $X$ be a projective smooth variety of dimension $n\ge 2$ and denote by $\rho$ its Picard number. Then, for any $L\in \Amp(X)_\R$ the quadratic intersection form $q(\alpha):=q_L(\alpha):=\alpha^2L^{n-2}$ is non-degenerate on $N^1(X)_\R$ of signature $(1,\rho-1)$. 
\end{theorem}

\begin{proof}
Suppose first that $L$ is a rational class in $\Amp(X)_\Q$. Then, by taking hyperplane sections we reduce ourselves to the classical statement of the Hodge Index Theorem for surfaces, cf. \cite[V.Thm.~1.9]{Hartshorne}. The fact that $q_L$ is non-degenerate then follows as in \cite[Sect.~3.8]{Debarre}.

Let now $L$ be an arbitrary class in $\Amp(X)_\R$. It is enough to check that $q_L$ is non-degenerate, since the signature will then be independent of $L\in\Amp(X)_\R$, {and therefore equal to $(1, \rho -1)$ as for rational classes}. 
We proceed by induction on $n$. When $X$ is a surface, the assertion is clear. So, suppose that $n>2$ and that the quadratic form $q_L$ is degenerate. Then, its associated symmetric matrix admits an eigenvector $D\in N^1(X)_\R$ associated with the eigenvalue $0$. Hence, $L^{n-2}D=0$ in $N_1(X)_\R$. Choose finitely many very ample smooth hypersurfaces $H_i$, $i = 1, \dots, k$, in $X$ such that $L$ is a (real) convex combination of their associated classes $L_i\in\Amp(X)$. Then, {for any $i \in \{1, \dots, k\}$, the equality} $L^{n-2}D=0$ implies $(L|_{H_i})^{n-2} D|_{H_i}=0$ and further by the induction hypothesis that $(L|_{H_i})^{n-3}( D|_{H_i})^2\le 0$. Here, equality occurs if and only if $ D|_{H_i}=0$ in $N^1(H_i)_\R $.
But this would imply that $DL_i^{n-1}= D|_{H_i}(L_i |_{H_i})^{n-2} =0$ and  that $D^2L_i^{n-2}=( D|_{H_i})^2(L_i |_{H_i})^{n-3}=0$ and as the theorem holds for the rational ample  classes $L_i $ we get that $D=0$ in $N^1(X)_\R$ which contradicts the choice of $D$. Thus, $L^{n-3}D^2L_i=(L|_{H_i})^{n-3}( D|_{H_i})^2< 0$ for each $i \in \{1, \dots, k\}$, and since $L$ is a convex combination of the $L_i$-s we infer that $L^{n-2}D^2<0$, which is again a contradiction to the choice of $D$. This ends the proof.
\end{proof}
\begin{corollary}[Special case of hard Lefschetz theorem]\label{cor:HardLefschetz}
 Let $X$ be a projective smooth variety of dimension $n\ge 2$ and let $L\in\Amp(X)_\R$. Then, the linear map\[N^1(X)_\R \to N_1(X)_\R,\; D \mapsto L^{n-2}.D\] is an isomorphism. 
\end{corollary}

For any ample class $L\in \Amp(X)_\R$ we  set
$$K^+_L(X):=\{ \beta\in N^1(X)_\R \ | \ \beta^2L^{n-2}>0, \ \beta L^{n-1}>0\}.$$
This is an open cone in $N^1(X)_\R$ containing $\Amp(X)_\R$. It is the ``positive component'' of the quadric cone $K_L(X):=\{ \beta\in N^1(X)_\R \ | \ \beta^2L^{n-2}>0\}$. 
By the Hodge Index Theorem one may find linear coordinates in $N^1(X)_\R$ with respect to which $K^+_L(X)$ becomes $\{ x\in\R^\rho \ | \ \sum_{i=1}^{\rho-1}x_i^2<x_\rho^2  , \ x_\rho>0\}.$ From this it follows that   $K^+_L(X)$ is self-dual in the sense that
$$ K^+_L(X)=\bigl\{\beta\in N^1(X)_\R \ | \ \beta\alpha L^{n-2}>0, \ \text{for all }\alpha\in \overline{ K^+_L(X)}\setminus\{0\}\bigr\}.$$
{
Note that by this self-duality property one can also write 
\begin{align*}
  K^+_L(X)&=\bigl\{\beta\in N^1(X)_\R \ | \ \beta^2 L^{n-2}>0, \ \exists\  \alpha\in\Amp(X)_\R: \ \beta\alpha^{n-1}>0\bigr\} \\
 &=
\bigl\{\beta\in N^1(X)_\R \ | \ \beta^2 L^{n-2}>0, \ \beta\alpha^{n-1}>0\ \ \text{for all } \alpha\in\Amp(X)_\R \bigr\}.
\end{align*}}
It is another direct consequence of the Hodge Index Theorem that for any $\beta\in K^+_L(X)$ the square root of the quadratic form $\alpha\mapsto-\alpha^2L^{n-2}$ gives a norm on the hyperplane 
$\beta^\perp:=\{\gamma\in N^1(X)_\R  \ | \ \beta\gamma L^{n-2}=0\}$. Furthermore, we set $$C^+(X)\;\;:= \bigcup_{L\in\Amp(X)_\R}L^{n-2}\cdot  K^+_L(X)\;\;\; \subset\;\; N_1(X)_\R$$
and note that \begin{equation}\label{eq:posisincplus}\Pos(X)_\R\subset C^+(X).\end{equation}

We will denote by $\Delta(F)$ the discriminant of a torsion-free coherent sheaf $F$  on $X$\:
$$ \Delta (F)=\frac{1}{\rk (F)}\bigl(c_2(F)-\frac{\rk(F)-1}{2\,\rk (F)}c_1^2(F)\bigr).$$
Furthermore, for two torsion-free coherent sheaves $G$ and $G'$ on $X$  we will use the notation
$$\xi_{G',G}:=\frac{c_1(G')}{\rank(G')}-\frac{c_1(G)}{\rank(G)}\in N^1(X)_\Q.$$

One key ingredient to our boundedness results is the following version of Bogomolov's Inequality:
\begin{theorem}[Bogomolov Inequality]\label{thm:Bogomolov}
 Let $X$ be an $n$-dimensional smooth projective variety over an algebraically closed field,
 $n\ge 2$, and $L\in \Amp(X)_\Q$. Then, there exists a non-negative constant $\beta=\beta(L)$ depending only on $X$ and on $L$ such that for any torsion-free sheaf $F$ of rank $r$ on $X$ with $$\Delta(F)L^{n-2}+r^2(r-1)^2\beta<0,$$ there exists some non-trivial saturated subsheaf $F'$ of $F$ with $\xi_{F',F}\in K^+_L(X)$. In particular, $F$ is not semistable with respect to any polarization in $L^{n-2}K^+_L(X)$. \qedhere
\end{theorem}
\begin{proof}
  This is proved for $n=2$ in \cite[Thm.~7.3.3]{Bible} in characteristic 0 and in \cite[Thm.~3.12]{Langer} in arbitrary characteristic and for arbitrary $n$. We remark that in characteristic $0$ the quantity $\beta(L)$ can be taken to be zero.
\end{proof}

\begin{lemma}\label{lemma:properly-ss}
Let $\gamma$ be any class in $N_1(X)$ and $\alpha\in\Amp(X)_\Q$. If a torsion-free sheaf $E$ is slope-semistable with respect to $\gamma$ but not
 with respect to $\alpha$, then $E$ is properly semistable with respect to a class $\gamma_t:=(1-t)\gamma+t\alpha^{n-1}$, for some $t\in{[0,1)}$. 
\end{lemma}
\begin{proof}
{We look at the set $T:=\{ t\in[0,1] \ | \ E \text{ is not stable w.r.t.\ } \gamma_t\}$. If $T=[0,1]$ then $E$ is properly $\gamma$-semistable, and the claim is proven. 
So, suppose that $T\neq[0,1]$ and that $E$ is $\gamma$-stable. For any saturated subsheaf $S$ of $E$ with $0<\rank (S)<\rank (E)$ we define an  affine-linear function $f_S:[0,1]\to \R$ by $f_S(t):=\mu_{\gamma_t}(S)-\mu_{\gamma_t}(E)=\gamma_t.\xi_{S,E}$. As $E$ is $\gamma$-stable, $f_S(0)<0$ for all such $S$. If $f_S$ vanishes somewhere on $[0,1]$, then $f_S(1)\ge 0$, and hence $\mu^\alpha(S)\ge\mu^\alpha(E)$. Consequently, the set of such $S$ is bounded by Grothendieck's Lemma \cite[Lem.~1.7.9]{Bible}. On the other hand, it is non-empty, as $E$ is $\alpha$-unstable. Hence, there exist only finitely many functions of type $f_S$ that vanish at some point of $[0,1]$, say $f_{S_1},\dots ,f_{S_m}$. We set $f$ to be their maximum, $f(t):=\max\{f_{S_1}(t),\dots,f_{S_m}(t)\}$. }

{It is immediately seen that $f$ is continuous, strictly increasing, that $f(0)<0$, $f(1)>0$ and that $T=\{f\ge 0\}$. In particular, $T=[t_0,1]$ is a closed interval. Moreover, $E$ is properly semistable with respect to $\gamma_{t_0}$. Indeed, it is semistable since $f_S(t_0)\le f(t_0)=0$ for any saturated subsheaf $S$ of $E$ with $0<\rank (S)<\rank (E)$, and it is not stable since there exists an $S_i$ with $0\le i\le m$ such that $f_{S_i}(t_0)=f(t_0)=0$.  }
\end{proof}

\begin{theorem}[Boundedness I]\label{thm:K^+}
 Let $X$ be a smooth $n$-dimensional projective variety and let $\Gamma\subset C^+(X)$ be a compact subset.   Then, the set of torsion-free sheaves $E$ with fixed topological type that are slope semistable with respect to some class in $\Gamma$ is bounded. 
\end{theorem}
\begin{proof}
Without loss of generality we may assume that there exists an $L\in \Amp(X)_\Q$ such that $\Gamma\subset L^{n-2}K^+_L(X)$. {Indeed, as $\Gamma$ is compact, it is covered by finitely many of the open cones $L^{n-2}K^+_L(X)$, and hence decomposes into a finite number of compact sets, each of which is contained in a cone of the form  $L^{n-2}K^+_L(X)$.} By enlarging $\Gamma$ if necessary, we may further assume that $\Gamma$ is star-shaped with respect to $L^{n-1}\in  L^{n-2}K^+_L(X)$. 
 
Let $E$ be a torsion-free  sheaf and $\gamma\in\Gamma$. 
There exists a class $\alpha\in K^+_L$ such that $\gamma=\alpha L^{n-2}$.
We denote by $\gamma_t:=(1-t)\gamma+t L^{n-1}=((1-t)\alpha+tL)L^{n-2}$ the  classes on the segment $[\gamma, L^{n-1}]$.
We write $\mu_{\max}(E)$ for  the maximal slope of a subsheaf of $E$ with respect to $L$.\medskip

\noindent{\em Claim:} If $E$ is $\gamma_t$-semistable for some $t\in[0,1]$, then $\mu_{\max}(E)$ is bounded by a function depending only on $\rank(E)$, $c_1(E)$, $\Delta(E)L^{n-2}$ and $\Gamma$.
\medskip

We shall prove this claim by induction on the rank of $E$. The assertion is clearly true when the rank is one. 
So, suppose that $\rank(E)>1$ and that the claim holds for all lower ranks.
If $E$ is slope semistable with respect to $L$, things are clear. Otherwise, by Lemma \ref{lemma:properly-ss} there exists some $t_0\in[0,1]$ and a proper 
saturated subsheaf $E_0$ of $E$ having the same  $\gamma_{t_0}$-slope as $E$. Thus, $\xi_{E_0,E}.\gamma_{t_0}=0$. We write $E'_0:=E/E_0$. Then, both $E_0$ and $E'_0$ are  $\gamma_{t_0}$-semistable, hence $\Delta(E_0)L^{n-2}$ and $\Delta(E'_0)L^{n-2}$ are  no smaller than
 $-r^2(r-1)^2\beta$ by Bogomolov's Inequality, Theorem~\ref{thm:Bogomolov} above. 
 This, the fact that $\xi_{E_0,E}\in ((1-t_{0})\alpha+t_0L)^\perp$, and the identity
$$-\frac{\rank(E_0)\rank(E'_0)}{\rank(E)} \xi_{E_0,E}^2=\frac{\Delta(E)}{\rank(E)}-\frac{\Delta(E_0)}{\rank(E_0)}-\frac{\Delta(E'_0)}{\rank(E'_0)}$$
imply that  $\xi_{E_0,E}$ belongs to a bounded ball inside $((1-t_{0})\alpha+t_0L)^\perp$ whose radius only depends on $\Delta(E)L^{n-2}$. 
Since the norm induced by the square root of the quadratic form $\delta\mapsto-\delta^2L^{n-2}$ on  $((1-t)\alpha+tL)^\perp$ varies continuously in $t$, and since $\Gamma$ is compact, we deduce that  $\xi_{E_0,E}$ belongs to a bounded and hence finite subset of $N^1(X)_\R$ depending only on $\Delta(E)L^{n-2}$, and of course on $\Gamma$. Thus, if
$c_1(E)$ is fixed,
 $c_1(E_0)$ and hence $c_1(E'_0)$ may only acquire a finite number of values. Moreover, the same identity, Bogomolov's Inequality for $E_0$ and $E'_0$, and the negativity of $ \xi_{E_0,E}^2L^{n-2}$ imply
 $$-r^2(r-1)^2\beta\le\frac{\Delta(E_0)L^{n-2}}{\rank(E_0)}\le\frac{\Delta(E)L^{n-2}}{\rank(E)}+r^2(r-1)^2\beta,$$
 $$-r^2(r-1)^2\beta\le\frac{\Delta(E'_0)L^{n-2}}{\rank(E'_0)}\le\frac{\Delta(E)L^{n-2}}{\rank(E)}+r^2(r-1)^2\beta.$$
 By the induction assumption it now follows that  $\mu_{\max}(E_0)$ and $\mu_{\max}(E'_0)$ are bounded by a function depending only on  $\rank(E)$, $c_1(E)$, $\Delta(E)L^{n-2}$, and $\Gamma$. Since $$\mu_{\max}(E)\le\max\bigl(\mu_{\max}(E_0),\mu_{\max}(E'_0)\bigr),$$ the claim is proved. The statement of the theorem is a direct consequence of the {\em Claim} by \cite[Thm.~4.2]{Langer}, see also \cite[Thm.~3.3.7]{Bible}. 
\end{proof}

\begin{remark}
{ A precursor of the argument used in the proof of Theorem~\ref{thm:K^+} can be found in \cite[proof of Lem.~6.4]{GrebToma}.}
\end{remark}

Our main boundedness results will be consequences of the preceding boundedness result in combination with the following convexity result. 
\begin{theorem}[Convexity of $C^+(X)$]\label{thm:Conv(Pos)}
  Let $X$ be a smooth projective variety of dimension at most three or of Picard number at most two. Then, the cone $C^+(X)$ is convex. In particular, $\Conv(\Pos(X))\subset C^+(X)$. 
\end{theorem}
\begin{proof}We first organise our linear algebra data:
We set $n=\dim (X)$ and $\rho=\dim N^1(X)_\R$.  For any basis $L_1,\dots,L_\rho$ of $N_1(X)_\R$ and any $L\in\Amp(X)_\R$ the matrix $A=(L_iL_jL^{n-2})_{1\le i,j\le\rho}$ is symmetric non-degenerate of signature $(1,\rho-1)$ by the Hodge Index Theorem, Theorem~\ref{thm:Hodge} above. We denote by $q(\alpha)=q_L(\alpha)=\alpha^2L^{n-2}$ the quadratic form induced by $L$ on $N^1(X)_\R$ and by $\tilde q=\tilde q_L:N^1(X)_\R\times N^1(X)_\R\to\R$ its associated bilinear form. Let $\beta=\sum_{i=1}^\rho x_iL_i$ be any class in $N^1(X)_\R$ and $x=(x_1,\dots,x_\rho)\in\R^\rho$ its coordinate vector. Then, we have $q(\beta)=x^\top Ax$. Let now $\gamma=\beta L^{n-2}$ and set $y=(y_1,\dots,y_\rho)$, where $y_j:=\gamma L_j$.
Then, $Ax=y$. 

 {The condition that $\gamma $ belong to $L^{n-2}K^+_L(X)$ translates into $q(\beta)>0$ and $\gamma \alpha>0$ for some $\alpha\in\Amp(X)_\R$. For the specific classes $\gamma$ discussed later on, it will be clear that the second inequality always holds.} The first one may be rewritten in coordinates as $x^\top Ax>0$ or equivalently as $y^\top  A^{-1}y>0$. Thus the quadratic form $q=q_L$ on $N^1(X)_\R$ is transported to $N_1(X)_\R$ by Hard Lefschetz, Corollary~\ref{cor:HardLefschetz}, and its matrix with respect to the chosen basis is $A^{-1}$. We denote also this quadratic form by $q$.

The statement of the theorem is clear when $\rho\le 2$. From now on, we will consider the case $\rho\ge3$ and $n=3$.

Let $\gamma_\infty\in L_1K^+_{L_1}(X)$, $\gamma_0 \in L_2K^+_{L_2}(X)$, 
where $L_1,L_2$ are real ample classes on $X$. Let $\gamma=\gamma_t=t\gamma_\infty +\gamma_0 $. Here, we may take $0\le t\le \infty$. We want to show that $\gamma_t$ belongs to $C^+(X)$ for all $t$. Since this is clear when $L_1$ and $L_2$ are proportional or when $\gamma_\infty $ and $\gamma_0 $ are proportional, we shall suppose in the sequel that neither of them are. 
When $\gamma_\infty $ and $\gamma_0 $ are not proportional, we may suppose moreover that 
\begin{equation}\label{eq:wlogcrossterms}
 (\gamma_\infty L_1)(\gamma_0 L_2)-(\gamma_\infty L_2)(\gamma_0 L_1)\neq0
\end{equation}
 by slightly perturbing $L_1$ or $L_2$. This assumption says that the subspace spanned by 
$L_1$ and $L_2$ is complementary to $V:=\{ D\in N^1(X)_\R \ |  \ D\gamma_\infty =0, \ D\gamma_0 =0\}$. We remark that 
\begin{equation}\label{eq:negativeonD}
 D^2L_1<0 \text{ and } D^2L_2<0 \text{ for all }D\in V \setminus\{0\},
\end{equation}
since by self-duality of the positive cones $K^+_{L_1}$ and $K^+_{L_2}$, the vector $D$ cannot belong to either of them. Thus, also for any convex combination $L$ of $L_1$ and $L_2$, one has $D^2L<0$ and $D\notin K^+_L(X)$. For later use, we also note that by self-duality again $\gamma_\infty L>0$, $\gamma_0 L>0$ for any ample class $L$.

  We now choose a basis $D_1,\dots,D_{\rho-2}$ of $V$, set $L_j=D_{2+j}$ for $1\le j\le \rho-2$ and apply the previous notations for the basis  $L_1,\dots,L_\rho$. {In particular, for each ample $L$, consider the matrix $A = A_L$.} Note that $(-1)^{\rho-1}\det(A)>0$. 
We have $A^{-1}=(\det A)^{-1} \adj (A)$, where $\adj(A)=(C_{i,j})_{i,j}$ is the cofactor matrix  of $A$. Let $y=(y_1,\dots,y_\rho)$, where $y_i =y_i(t)$ be the coordinate vector of $\gamma=\gamma_t$, as before. Note that our choice of basis implies $y_1,y_2>0$, $y_3=\dots=y_\rho=0$. We want to show that for any $t$ there is a choice of $L\in [L_1,L_2]$ with corresponding matrix $A$ such that
\begin{equation}\label{eq:positive}
 y^\top A^{-1}y>0.
\end{equation}
The expression \eqref{eq:positive} is quadratic and homogeneous in $y_1$ and $y_2$ and its positivity will not change if we divide it by $y^2_2$. 
Writing $\lambda=\lambda(t) =\frac{y_1}{y_2}$, we get a quadratic function of $\lambda$ whose positivity we want to examine: $$f_L(\lambda)=(\lambda,1,0,\dots,0) (-1)^{\rho-1}\adj(A)(\lambda,1,0,\dots,0)^\top=(-1)^{\rho-1}(C_{1,1}\lambda^2+2C_{1,2}\lambda+C_{2,2}).$$ 
Note that the leading coefficient $(-1)^{\rho-1}C_{1,1}$ of $f_L$ is negative, as it is $(-1)^{\rho-1}$ times a principal $(\rho-1)\times(\rho-1)$-minor of $A$, whose signature on the subspace spanned by $L_2$ and  $V$  is $(1,\rho-2)$ by \eqref{eq:negativeonD}. The same is true for $(-1)^{\rho-1}C_{2,2}$, again by \eqref{eq:negativeonD}. The discriminant of $f$ is $4(C_{1,2}^2-C_{1,1}C_{2,2})$, which is $-4$ times a $2\times 2$-minor of $\adj (A)$. By \cite[Ch.III, Exercise~11.9]{BourbakiAlgebra1-3}, this minor equals $\det(A)$ times the $(\rho-2)\times(\rho-2)$-minor of $A$ formed on complementary position. {The corresponding matrix represents the} restriction of our quadratic form to $V$, which has signature $(0,\rho-2)$. 
Thus, the discriminant of $f_L$ is positive for any convex combination $L$ of $L_1$ and $L_2$. So, for any such $L$ the function $f_L$ {will take positive values on a} non-empty open interval $I_L \subset \mathbb R$.
Therefore, the subset
$$\bigl\{(\lambda,L)\in \R\times[L_1,L_2] \ | \ f_L(\lambda)>0\bigr\}\;\;=\bigcup_{L\in[L_1,L_2]}I_L\times\{ L\}\;\; \subset\;\; \R\times[L_1,L_2]$$
is connected, and hence so is its projection $\bigcup\nolimits_{L\in[L_1,L_2]}I_L$ on $\R$. As
$\lambda(0)\in I_{L_2}$ and $\lambda(\infty)\in I_{L_1}$, it follows that any $\lambda$ between $\lambda(0)$ and $\lambda(\infty)$ belongs to the positivity interval $I_L$ of some suitable convex combination $L$ of $L_1$ and $L_2$.
Writing out $\lambda(t)=\frac{t(\gamma_\infty L_1)+\gamma_0 L_1}{t(\gamma_\infty L_2)+\gamma_0 L_2}$, we see that $$\lambda'(t)=\frac{(\gamma_\infty L_1)(\gamma_0 L_2)-(\gamma_\infty L_2)(\gamma_0 L_1)}{(t(\gamma_\infty L_2)+\gamma_0 L_2)^2},$$ which does not vanish by our assumption \eqref{eq:wlogcrossterms}. Hence, $\lambda$ is a  monotonous function of $t$, and $\lambda(t)$ must lie between $\lambda(0)$ and $\lambda(\infty)$ for all $t\in [0,\infty]$. This proves the theorem.
\end{proof}

\begin{corollary}[Boundedness II]\label{thm:positiveimpliesstronglybounded}
Let $X$ be a smooth projective variety of dimension $n$, $\tau \in B(X)_\mathbb{Q}$, and $L_1, \dots, L_{j_0}$ ample line bundles on $X$.   Furthermore, suppose that $\Sigma\subset (\mathbb R_{\ge 0})^{j_0}$ is a  closed convex polyhedral cone with the origin removed.
If 
$$ \sum\nolimits_j \sigma_j \,c_1(L_j)^{n-1} \in \Pos(X)_{\mathbb R} \quad \text{ for all }(\sigma_1, \dots, \sigma_{j_0}) \in \Sigma, $$
then $\Sigma$ is a bounded set of stability parameters with respect to $\tau$ and $\underline{L}=(L_1, \dots, L_{j_0})$.
  \end{corollary}
\begin{proof}
We define $\hat \Sigma := \{\sum\nolimits_j \sigma_j L_j^{n-1} \mid (\sigma_1, \dots, \sigma_{j_0}) \in \Sigma  \} \subset \mathrm{Pos}_\mathbb{R}(X)$. Since $\Sigma$ is a cone over a compact base, the same is true for $\hat \Sigma$. 
If $\Gamma$ is such a base for $\hat \Sigma$, {it suffices to show that $\Gamma$ is a bounded set of stability parameters.} This however follows directly from Lemma \ref{lem:slopeimpliesmulti} and Theorem \ref{thm:K^+}, since $\Pos(X)_\R\subset C^+(X)$ by \eqref{eq:posisincplus}.\end{proof}
\begin{corollary}[Boundedness III]
\label{cor:boundednesssurfaceorpicard2}
  Let $X$ be a smooth projective variety, $\tau \in B(X)_\mathbb{Q}$, and $L_1, \dots, L_{j_0}$ ample line bundles on $X$. In addition, suppose that
  \begin{enumerate}
  \item the rank of the torsion-free sheaves under consideration is at most two, or
  \item the dimension of $X$ is at most three, or
  \item the Picard rank of $X$ is at most two.
  \end{enumerate}
Then, the whole set $(\mathbb R_{\ge 0})^{j_0}\setminus\{0\}$ of stability parameters is bounded with respect to $\tau$ and $\underline{L}=(L_1, \dots, L_{j_0})$.
\end{corollary}
\begin{proof}
By Lemma \ref{lem:slopeimpliesmulti} it is sufficient to show that in any of the above situations the set of torsion-free coherent sheaves of fixed topological type $\tau$ and which are slope-semistable with respect to some class in the compact set $\Conv(\{ L_1^{n-1},\dots,L_{j_0}^{n-1}\})\subset N_1(X)_\R$ is bounded.
When $\dim(X)\le 3$ or when $\rho(X)\le 2$, this follows directly from Theorems \ref{thm:K^+} and  \ref{thm:Conv(Pos)}.

So, suppose now that $E$ is a torsion-free sheaf of  rank two and of fixed topological type $\tau$, which is slope-semistable with respect to some $\gamma=\sum\nolimits_j \sigma_j L_j^{n-1}\in\Conv(\{ L_1^{n-1},\dots,L_{j_0}^{n-1}\} )$. The classes $ L_j^{n-1} $ are in $C^+(X)$, and one can see that they also belong to some compact connected subset $\Gamma$ of $C^+(X)$, by considering connecting paths between the $L_j^{n-1}$-s for instance.  

Set $\alpha=L_1$. If $E$ is slope-semistable with respect to $\alpha$, then $E$ lies in a bounded set of sheaves by \cite[Thm.~3.3.7]{Bible}. If not, then by Lemma \ref{lemma:properly-ss} there exists $t \in [0,1]$ such that  $E$ is properly semistable with respect to $\gamma_t:= (1-t)\gamma + t \alpha^{n-1}$. Hence, there exists a proper saturated subsheaf $E_0$ of $E$ having the same  $\gamma_{t}$-slope as $E$; i.e.,\ $\xi_{E_0,E}.\gamma_{t}=0$. 
 Let $E'_0:=E/E_0$. 
 Then,  $E_0$ and $E'_0$ are  rank one torsion-free sheaves having the same $\gamma_{t}$-slope. Since $\gamma_t$ lies in $\Conv(\{ L_1^{n-1},\dots,L_{j_0}^{n-1}\} )$, the hyperplane $H:=\{ \delta\in N_1(X)_\R \ | \ \delta\xi_{E_0,E}=0\}$ {separates $\alpha$ from one of the other $L_j^{n-1}$-s}. Since $\Gamma$ is connected and contains $L_1^{n-1},\dots,L_{j_0}^{n-1}$, there exists some class $\gamma'$ in $\Gamma\cap H$. Figure 1 illustrates the situation. It follows that $c_1(E_0)\gamma'=c_1(E'_0)\gamma'$, and thus $E$ is slope-semistable with respect to $\gamma'\in \Gamma$. Boundedness then follows from Theorem \ref{thm:K^+}.
\end{proof}
\vspace{-0.4cm}
\begin{center}
\begin{figure}[h]\label{figure1}
\begin{tikzpicture}
 \fill (5,0) circle (2pt) node [below right]{$\alpha$};
 \draw (0,0) .. controls(2.5,0.7) .. (5,0) ;
 \draw[style=thick](5,0)  -- (1.5,3);
 \draw (5,0) .. controls(4,1.9) .. (5,5);
 \draw[style=thick](5,5) -- (3.5, 4) node[below]{$\Gamma$} -- (1.5,3) ;
 \draw (5,5) .. controls(1.5,3.6)  .. (0,5);
 \draw[style=thick](0,5)  -- (1.5,3);
 \fill (0,5) circle (2pt) node [left]{$L_j^{n-1}$};
 \draw (0,5)  .. controls(1.3,2.9) .. (0,0);
 \draw[style=thick](0,0)  -- (1.5,3);
 \fill (4.9, 4) circle (2pt) node [above right]{$\gamma$};
 \draw[style=thick] (4.9, 4)  -- (5,0);
 \fill (4.96,1.6) circle (2pt) node[above right] {$\gamma_t$};
 \draw[style=thick] (6.96, 1) -- (-1.04, 3.4)node [left]{$H$} ; 
 \node [above] (gamma') at (2.5, 2.3){$\gamma'$};
\end{tikzpicture}
\caption{A cross section through $N_1(X)_\R$ and $\mathrm{Pos}_\R(X)$.}
\end{figure}
 \end{center}
 \begin{remark}\label{rmk:Joyce}
   As pointed out by Joyce, the above boundedness results have implications for a certain ``technical difficulty'' concerning wall crossing formulae for Donaldson-Thomas invariants discussed in \cite[p.~27]{JoyceSong} and \cite[Sect.~5]{JoyceIV}.
 \end{remark}

\section{The Le Potier--Simpson Theorem}\label{sect:lepotier}
 We now fix a projective scheme $X$ over an algebraically closed field of characteristic 0, a vector $\underline{L} = (L_1,\ldots,L_{j_0})$ of very ample line bundles on $X$ and a topological type $\tau \in B(X)_\mathbb{Q}$.  We also fix a bounded stability parameter $\sigma = (\underline{L},\sigma_1,\ldots,\sigma_{j_0})$.
 
 Our goal is to show that semistability of a sheaf can be detected by the spaces of sections of its subsheaves, which we do with the following version of the Le Potier--Simpson estimate.  We write $[x]_+ = \max\{x,0\}$. 

\begin{theorem}[Le Potier--Simpson estimate]\label{thm:lepotiersimpson}
  Let $X$ be a projective scheme and $L$ be a very ample line bundle on $X$.  Let $E$ be a pure $d$-dimensional sheaf and set 
$$C_E^{L}:= (r_E^L)^2 + \frac{1}{2}(r_E^L+d)-1.$$
Then, for any $n>0$,
$$ h^0(E\otimes L^n) \le \frac{r_E^L-1}{d!} \left[\hat{\mu}_{\max}^L(E) + C_E^{L} +n\right]_+^d + \frac{1}{d!}\left[\hat{\mu}^L(E) +C_E^{L} +n\right]_+^d.$$
\end{theorem}
\begin{proof}
  This is proved in \cite[Cor.~3.3.8]{Bible}.
\end{proof}

We emphasise that $C_E^L$ depends only on the Hilbert polynomial $P_E^L$ of $E$, and if $E'\subset E$ then $C_{E'}^L\le C_{E}^L$.  We next prove a tailored version of a theorem originally due to Le Potier \cite{LePotier} and Simpson \cite{Simpson}. 

By hypothesis, the set of semistable sheaves of topological type $\tau$ is bounded.  Thus, for all $p$ sufficiently large any semistable sheaf is $(p,\underline{L})$-regular. We fix such a $p \in \mathbb{N}$.

\begin{theorem}[Le Potier--Simpson Theorem]\label{thm:detectstabilitysections}
For all $n\gg p$  the following are equivalent for any pure $d$-dimensional sheaf $E$ of topological type $\tau$.
\begin{enumerate}
\item $E$ is (semi)stable
\item $E$ is $(p,\underline{L})$-regular and for all proper $E'\subset E$ we have
  \begin{equation}
 \frac{\sum_j \sigma_j h^0(E'\otimes L_j^n)}{r_{E'}^{\sigma}}(\le) p_E^{\sigma}(n)\label{eq:detectstabilitysubsheaf}
 \end{equation}
\item $E$ is $(p,\underline{L})$-regular and for all proper saturated $E'\subset E$ with $\hat{\mu}^{\sigma}(E')\ge \hat{\mu}^{\sigma}(E)$ the inequality \eqref{eq:detectstabilitysubsheaf} holds.
\end{enumerate}
Moreover, if $E$ is semistable of topological type $\tau$, and $E'\subset E$ is a proper subsheaf, then equality holds in \eqref{eq:detectstabilitysubsheaf} if and only if $E'$ is destabilising.  
\end{theorem}

\begin{proof}
  Our proof follows \cite[Theorem 4.4.1]{Bible}, only with a more careful bookkeeping of the constants involved.      The set of all $(p,\underline{L})$-regular sheaves of topological type $\tau$ is bounded.  Thus by Proposition \ref{prop:stronglybounded} we know there is a positive constant $C_1$ such that
\begin{equation}\label{eq:maxslopeestimated}
 \max_j  \, \hat{\mu}_{\max}^{L_j}(E) \le C_1
\end{equation}
for all such sheaves.     Define $ \overline{C} = \max_j \{ C_{E}^{L_j}\}$, where the quantity on the right is as in Theorem \ref{thm:lepotiersimpson}, and observe that $\overline{C}$ depends only on $\tau$.   We then pick a positive constant $C_2$ large enough so that for any $(p, \underline{L})$-regular sheaf $E$ of topological type $\tau$ we have
\begin{equation}
C_2\ge -\hat{\mu}^{\sigma}(E)+1\label{eq:C20}
\end{equation}
and
\begin{equation}\label{eq:C21}
  \left(1-\frac{\sum_j \sigma_j}{r_E^{\sigma}}\right) (C_1 + \overline{C})+\frac{\sum_j \sigma_j}{r_E^{\sigma}}(-C_2 + \overline{C})\le \hat{\mu}^{\sigma}(E)-1
\end{equation}
(this is clearly possible, as $r_E^{\sigma}$ and $\hat{\mu}^{\sigma}(E)$ depend only on $\sigma$ and $\tau$).

Now, let
 $\mathcal S$ be the set of all saturated subsheaves $F\subset E$, where $E$ is $(p,\underline{L})$-regular of topological type $\tau$, and 
 \begin{equation}
\hat{\mu}^{L_j}(F)\ge -C_2 \text{ for some } j\in \{1,\ldots,j_0\}.\label{eq:definitionofS}
\end{equation}
We claim that $\mathcal S$ is bounded, which uses in an essential way the hypothesis that each sheaf in $\mathcal S$ is saturated.   For $j_0=1$, this follows immediately from Grothendieck's Lemma \cite[Lem.~1.7.9]{Bible}, since the family of $(p,\underline{L})$-regular sheaves $E$ of topological type $\tau$ is bounded. For higher $j_0$, the set $\mathcal S$ is thus a finite union of bounded sets of sheaves, and so itself bounded.  

We claim that for all $n\gg p$ the following hold:
{\renewcommand{\theenumi}{\roman{enumi}}
\begin{enumerate}
\item For any $F\in \mathcal S$, $ p_{F}^{\sigma}(n) \sim p_E^{\sigma}(n) \text{ if and only if } p_{F}^{\sigma}\sim p_E^{\sigma}$,
where $\sim$ stands for $<, \le$ or $=$.
\item All sheaves $F\in \mathcal S$ are $(n,\underline{L})$-regular.
\item $n>C_2-\overline{C}$.
\item For any sheaf $E$ of topological type $\tau$, 
$$\left(1-\frac{\sum_j \sigma_j}{r_{E}^{\sigma}}\right) \frac{1}{d!}\left(C_1 + \overline{C}  +n \right)^d + \frac{\sum_j \sigma_j}{r_{E}^{\sigma}}
  \frac{1}{d!}\left(-C_2 + \overline{C}  +n\right)^d \le p_E^{\sigma}(n)-1.$$
\end{enumerate}}
To see this is true, observe that there are only a finite number of Hilbert polynomials $\chi(F\otimes L_j^k)$ among the sheaves in $\mathcal S$, so (i) holds for all $n$ sufficiently large.  That (ii) holds follows from boundedness of $\mathcal{S}$, and (iii) can obviously be achieved.  The condition (iv) is true for all $n$ sufficiently large by inequality \eqref{eq:C21}, since  $p_E^{\sigma}(m) = \frac{1}{d!}m^d + \frac{1}{(d-1)!}\hat{\mu}^{\sigma}(E) m^{d-1} + O(m^{d-2})$, where the $O(m^{d-2})$-terms and $r_{E}^{\sigma}$ depend only on the type of $E$.\medskip

\noindent {\bf Proof that (1) implies (2):} Let $E$ be (semi)stable of topological type $\tau$, and let $P^{\sigma} := P_E^{\sigma}$, which depends only on $\sigma$ and $\tau$.  By our choice of $p$, the semistability hypothesis implies that $E$ is $(p,\underline{L})$-regular. Now, let $E'\subset E$ be a proper subsheaf.  We split into two cases:
{\renewcommand{\theenumi}{\Alph{enumi}}
\begin{enumerate}
\item $\hat{\mu}^{L_j}(E') < - C_2$ for all $j$,
\item $\hat{\mu}^{L_j}(E') \ge - C_2$ for some $j$.
\end{enumerate}}
Suppose $E'$ is of type (B) and let $F$ be the saturation of $E'$ in $E$.  Then, as $\hat{\mu}^{L_j}(F)\ge \hat{\mu}^{L_j}(E')$, we see that $F\in \mathcal S$.  In particular, $F$ is $(n,\underline{L})$-regular by (ii).  Since $E$ is (semi)stable, we have $p^{\sigma}_{F}(\le) p^{\sigma}_E$ and so by (i)
\begin{equation}
\frac{\sum_j \sigma_j h^0(E'\otimes L_j^n)}{r_{E'}^\sigma} \le \frac{\sum_j \sigma_j h^0(F\otimes L_j^n)}{r_{F}^\sigma} = p^{\sigma}_{F}(n) \le p^{\sigma}_{E}(n).\label{eq:proofofcaseB}
\end{equation}
Thus, we have \eqref{eq:detectstabilitysubsheaf} for all subsheaves $E'$ of type (B).

We now deal with the case of equality for sheaves of type (B).    Suppose first equality holds in  \eqref{eq:detectstabilitysubsheaf}, where $E'$ is of type (B).    Then, the equality in \eqref{eq:detectstabilitysubsheaf} implies equality in \eqref{eq:proofofcaseB}, and so as $r_{E'}^\sigma = r_{F}^\sigma$ and $h^0(E'\otimes L_j^n) \le h^0(F\otimes L_j^n)$ for all $j$, we conclude that there must be some $j$ for which $H^0(F\otimes L_j^n)=H^0(E'\otimes L_j^n)$.  But $F$ is $(n,\underline{L})$-regular, and so in particular $F\otimes L_j^n$ is globally generated, and thus $F\subset E'$.    Hence, $E'$ is saturated and so lies in $\mathcal S$.  Thus, $E'$ is $(n,\underline{L})$-regular by (ii) and the assumed equality in \eqref{eq:detectstabilitysubsheaf} is precisely that $p_{E'}^{\sigma}(n) = p_E^{\sigma}(n)$.  Thus, (i) implies that $p_{E'}^{\sigma} = p_{E}^{\sigma}$, from which we conclude that $E'$ is destabilising.   Conversely, assume $E'\subset E$ is destabilising.  Then, by Lemma \ref{lem:destabilisingimpliessaturated} the direct sum $G:=E'\oplus (E/E')$ is semistable.  But $G$ has topological type $\tau$, so by hypothesis $G$ is $(p,\underline{L})$-regular, and thus $E'$ is also $(p,\underline{L})$-regular.  Hence, $E'$ is also $(n,\underline{L})$-regular, and so equality holds in   \eqref{eq:detectstabilitysubsheaf} as $p_{E'}^{\sigma} = p_E^{\sigma}$.

To deal with sheaves of type (A) we use the Le Potier-Simpson estimate.  First consider a fixed $j$.  As $E$ is pure of dimension $d$, the same is true of $E'$.  Moreover, as $E'$ is a subsheaf of $E$,
\begin{equation}\label{eq:maxslopeestimatedC1}
 \hat{\mu}_{\max}^{L_j}(E')\le \hat{\mu}_{\max}^{L_j}(E)\le C_1
\end{equation}
by \eqref{eq:maxslopeestimated}, and
\begin{equation}\label{eq:barCestimate}
  C_{E'}^{L_j} \le C_{E}^{L_j} \le \overline{C}
\end{equation}
 by the definition of $\overline{C}$. Thus, from Theorem \ref{thm:lepotiersimpson} applied to $E'$ and $L_j$ we have that for any $n>0$
\begin{align*}
  h^0(E'\otimes L_j^n) &\le \frac{r_{E'}^{L_j}-1}{d!} \left[
    \hat{\mu}_{\max}^{L_j}(E') + C_{E'}^{L_j}  + n\right]_+^d +
  \frac{1}{d!} \left[\hat{\mu}^{L_j}(E') + C_{E'}^{L_j}  +n
  \right]_+^d\\
&\le \frac{r_{E'}^{L_j}-1}{d!} \left[C_1 + \overline{C}  +n \right]_+^d + \frac{1}{d!} \left[-C_2 + \overline{C}  +n\right]_+^d ,
\end{align*}
where we have used inequalities \eqref{eq:maxslopeestimatedC1} and \eqref{eq:barCestimate}, the assumption that $E'$ is of type (A), and the simple fact that $x\le y$ implies $[x]_+\le [y]_+$.  Now, by condition (iii) above, the term in the last square brackets is positive.  So, we in fact have
\begin{align}\label{eq:onej}
  h^0(E'\otimes L_j^n)&\le \frac{r_{E'}^{L_j}-1}{d!} \left(C_1 + \overline{C}  +n \right)^d + \frac{1}{d!} \left(-C_2 + \overline{C}  +n\right)^d.
\end{align}
Recall that $r_{E'}^{\sigma} = \sum_j \sigma_j r_{E'}^{L_j}$.  So multiplying \eqref{eq:onej} by $\sigma_j$, then summing over all $j$, and dividing by $r_{E'}^{\sigma}$ yields
\begin{align*}
   \frac{\sum_j \sigma_j  h^0(E'\otimes L_j^n)}{r_{E'}^{\sigma}}&\le \left(1-\frac{\sum_j \sigma_j}{r_{E'}^{\sigma}}\right) \frac{1}{d!}\left(C_1 + \overline{C}  +n \right)^d + \frac{\sum_j \sigma_j}{r_{E'}^{\sigma}}
  \frac{1}{d!}\left(-C_2 + \overline{C}  +n\right)^d.
\end{align*}
Notice that $r_{E'}^{\sigma} \ge \sum_j \sigma_j$, so the above is a convex combination of $\frac{1}{d!}(C_1+\overline{C} +n)^d$ and $\frac{1}{d!}(-C_2+\overline{C}+n)^d$, and that as $C_1,C_2$ are positive, we clearly have $(C_1 + \overline{C}+n)^d > (-C_2 + \overline{C} +n)^d$.    So as $r_{E'}^{\sigma}\le r_E^{\sigma}$ we can replace all the multiplicities of $E'$ in the previous equation with those of $E$ and only improve the inequality, i.e.
\begin{align}
   \frac{\sum_j \sigma_j  h^0(E'\otimes L_j^n)}{r_{E'}^{\sigma}}&\le \left(1-\frac{\sum_j \sigma_j}{r_{E}^{\sigma}}\right) \frac{1}{d!}\left(C_1 + \overline{C}  +n \right)^d + \frac{\sum_j \sigma_j}{r_{E}^{\sigma}}
  \frac{1}{d!}\left(-C_2 + \overline{C}  +n\right)^d\nonumber\\
&\le p_{E}^{\sigma}(n)-1,\label{eq:later}
\end{align}
where the last inequality uses  condition (iv).   Hence the desired inequality \eqref{eq:detectstabilitysubsheaf} holds strictly and so \eqref{eq:detectstabilitysubsheaf} holds strictly.     Finally, we observe that if $E'$ is of type (A) then $\hat{\mu}^{\sigma}(E')\le -C_2\le \hat{\mu}^{\sigma}(E)-1$ by \eqref{eq:C20}, and so $E'$ is not destabilising.  Thus, along with the paragraph immediately after \eqref{eq:proofofcaseB} we see that for any $E'\subset E$,  equality holds in \eqref{eq:detectstabilitysubsheaf} if and only if $E'$ is destabilising if and only if $E'$ is saturated and destabilising.\medskip

\noindent{\bf Proof that (3) implies (1)}:
Note that (2) obviously implies (3).  So, suppose (3) holds and that $E$ is a pure $d$-dimensional sheaf of topological  type $\tau$ and that (3) holds.  To show that $E$ is (semi)stable it is sufficient by Lemma~\ref{lem:saturated} to prove that $p_{E'}^{\sigma}(\le) p_E^{\sigma}$ for all saturated subsheaves $E'\subset E$.  So let $E'\subset E$ be saturated.  If $\hat{\mu}^{\sigma}(E')<\hat{\mu}^{\sigma}(E)$, then clearly $E'$ does not destabilise.    So, suppose that $\hat{\mu}^{\sigma}(E') \ge \hat{\mu}^{\sigma}(E)$. Then, (3) says that \eqref{eq:detectstabilitysubsheaf} holds for this $E'$. 

Now part of the hypothesis in (3) is that $E$ is $(p,\underline{L})$-regular. This together with Lemma \ref{lem:slope} implies that there exists a $j \in \{1, \dots, j_0\}$ such that
$$ \hat{\mu}^{L_j}(E') \ge \hat{\mu}^{\sigma}(E) \ge -C_2,$$
where the last inequality comes from \eqref{eq:C20}. Looking at the defining inequality \eqref{eq:definitionofS} of $\mathcal{S}$, we conclude that $E'\in \mathcal S$. Thus, $E'$ is $(n,\underline{L})$-regular by (ii).  

As \eqref{eq:detectstabilitysubsheaf} holds for $E'$, we get
$$p_{E'}^{\sigma}(n) = \frac{\sum_j \sigma_j h^0(E'\otimes L_j^n,)}{r_{E'}^{\sigma}} \le p_E^{\sigma}(n).$$
Hence, by (i) we deduce $p_{E'}^{\sigma}(\le)p_E^{\sigma}$, and so $E$ is semistable; i.e., (3) implies (1).   
\end{proof}

\begin{corollary}\label{cor:lepotiersimpson}
For all $n\gg p\gg 0$  the following are equivalent for any pure $d$-dimensional sheaf $E$ of topological type $\tau$:
\begin{enumerate}
\item $E$ is semistable.
\item $E$ is $(p,\underline{L})$-regular and for all proper  $E'\subset E$ we have the inequality of polynomials
  \begin{equation}\label{eq:corlepotiersimpson3}
 \sum\nolimits_j \sigma_j h^0(E'\otimes L_j^n) P_E^{\sigma} \le P_E^{\sigma}(n) P_{E'}^{\sigma}.
\end{equation}
\item  $E$ is $(p,\underline{L})$-regular and for all proper saturated $E'\subset E$  with $\hat{\mu}^{\sigma}(E')\ge \hat{\mu}^{\sigma}(E)$ the inequality \eqref{eq:corlepotiersimpson3} holds. 
\end{enumerate}
Moreover if $E$ is semistable and $E'\subset E$ is a proper subsheaf then equality holds in \eqref{eq:corlepotiersimpson3} if and only if $E'$ is destabilising. 
\end{corollary}
\begin{proof}
For all $p\gg 0$ the set of semistable sheaves of topological type $\tau$ are $(p,\underline{L})$-regular.  So let $n\gg p\gg 0$ be as in the proof of the previous Theorem.  

To show (1) implies (2) set $r= \sum_j \sigma_j r_E^{L_j}$ and $r'=\sum_j \sigma_j r_{E'}^{L_j}$  (which of course also depend on $\sigma$)  and let $E$ be semistable. So by Theorem~\ref{thm:detectstabilitysections}, for  $E'\subset E$ we have the following inequality of polynomials
\begin{equation}
\sum\nolimits_j \sigma_j h^0(E'\otimes L_j^n) r \le P_E^{\sigma}(n) r'.\label{eq:corlepotiersimpson2}
\end{equation}
If this inequality is strict, then we infer that \eqref{eq:corlepotiersimpson3} holds strictly, since $r$ and $r'$ are the leading order terms in $P_E^{\sigma}$ and $P_{E'}^\sigma$, respectively. On the other hand, if equality holds, then by the last statement in Theorem~\ref{thm:detectstabilitysections} we have that $E'$ is destabilising, so $p^{\sigma}_{E'}=p^{\sigma}_E$.  Then, $P_E^{\sigma}/r = P_{E'}^\sigma/r'$, and so equality in \eqref{eq:corlepotiersimpson2} implies equality in \eqref{eq:corlepotiersimpson3}.

Clearly (2) implies (3), so assume that (3) holds.  Suppose $E'\subset E$ is a proper subsheaf with $\hat{\mu}^{\sigma}(E')\ge \hat{\mu}^{\sigma}(E)$, so by hypothesis \eqref{eq:corlepotiersimpson3} holds.   This implies the corresponding inequality in the leading order coefficients of the two polynomials, and this leading order term  is precisely \eqref{eq:corlepotiersimpson2}.  Thus, by the implication ``(3) $\Rightarrow$ (1)'' of Theorem \ref{thm:lepotiersimpson} we deduce that $E$ is semistable.

The final statement is proven similarly: If $E$ is semistable, and $E'\subset E$ is such that equality holds in \eqref{eq:corlepotiersimpson3}, then equality holds in its leading order term \eqref{eq:corlepotiersimpson2}, and thus by the final statement in Theorem \ref{thm:detectstabilitysections} the subsheaf $E'$ is destabilising.  Conversely, by the same theorem, if $E'$ is destabilising, then $p_{E'}^{\sigma} = p_E^{\sigma}$ and equality holds in \eqref{eq:corlepotiersimpson2}, which implies equality in \eqref{eq:corlepotiersimpson3}. 
\end{proof}

\section{Comparison of semistability}\label{sect:ComparisonOfSemistability}

Our goal is to compare stability of a sheaf $E$ with stability of the module $\Hom(T,E)$ introduced in Section~\ref{subsubsect:Qreps}. 

 We continue using the notation of the previous section; so $\tau\in B(X)_{\mathbb Q}$, each $L_j$ is very ample and $\sigma$ is a bounded stability parameter. 
 Moreover, we re-invoke the notations of Section~\ref{sec:categories}. In particular, for integers $m > n$, we consider the sheaf $$T := \bigoplus\nolimits_{j} \bigl(L_j^{-n}  \oplus L_j^{-m} \bigr),$$  
whose dependence on $m,n$ will be suppressed throughout the discussion, and for a coherent sheaf $E$ the representation $\Hom(T, E)$ of the algebra $A = L \oplus \bigoplus_{i,j=1}^{j_0}\Hom(L_j^{-m}, L_i^{-n})$ associated with our quiver $Q$, cf.~Section~\ref{subsubsect:Qreps}.

Recall that a stability parameter $\sigma = (\underline{L},\sigma_1,\ldots,\sigma_{j_0})$ is called positive if all the $\sigma_j$ are positive.

\begin{theorem}[Comparison of semistability and JH filtrations]\label{thm:mainsemistabilitycomparison}
 For all integers $m\gg n\gg p\gg 0$ the following holds  for any sheaf $E$ on $X$ of topological type $\tau$:
 \begin{enumerate}
 \item  $E$ is semistable if and only if it is pure, $(p,\underline{L})$-regular, and $\Hom(T,E)$ is semistable.
 \item Suppose $\sigma$ is positive.  If $E$ is semistable, then
$$ \Hom(T,gr E)\simeq gr \Hom(T,E),$$
where $gr$ denotes the graded object coming from a Jordan-H\"older filtration of $E$ or $\Hom(T,E)$, respectively.  In particular, two semistable sheaves $E$ and $E'$ are $S$-equivalent if and only if $\Hom(T,E)$ and $\Hom(T,E')$ are $S$-equivalent.  
 \end{enumerate}

\end{theorem}
Parts (1) and (2) are proven as Theorems~\ref{thm:semistability} and \ref{thm:sequivalence} below. 
\begin{remark}
   In fact, for the ``only if'' statement in (1) one can in fact choose $n=p$.     In the case $j_0=1$ considered by \cite{ConsulKing} it is proved moreover that one can take $n=p$ for the converse direction.  However, we have not been able to prove that this is the case for higher $j_0$. The issue arises in the converse direction of the Le Potier--Simpson theorem in which we needed to assume a priori that $E$ lies in a bounded family (for example that it is $(p,\underline{L})$-regular) to deduce the ``(3) $\Rightarrow$ (1)''-direction.
\end{remark}

The proof of the theorem is adapted from \cite[Section 5]{ConsulKing}.   We begin, as the authors of \cite{ConsulKing} do,  by making explicit our requirements on the integers $p$, $n$, and $m$.  First, we choose $p$ sufficiently large so that: \medskip

\noindent (\hypertarget{C1}{C1})
Every sheaf $E$ of topological type $\tau$ that is semistable with respect to $\sigma$ is $(p,\underline{L})$-regular.\medskip

This is possible by the boundedness assumption on $\sigma$.  Then, given such a $p$, we choose $n\gg p$ sufficiently large so\medskip

\noindent (\hypertarget{C2}{C2})
 The conclusion of Corollary \ref{cor:lepotiersimpson} (coming from the Le Potier-Simpson Theorem) holds for any pure sheaf of topological type $\tau$.
\medskip

Now choose $m\ge n$ large enough so  that three further conditions (\hyperlink{C3}{C3}), (\hyperlink{C4}{C4}), and (\hyperlink{C5}{C5}) hold:\medskip

\noindent (\hypertarget{C3}{C3}) For all $j=1,\ldots,j_0$ the line bundle $L_j^{-n}$ is $(m,\underline{L} )$-regular.\medskip

To discuss the final two conditions we make some definitions.  Let $E$ be any sheaf that is $(n,\underline{L})$-regular and has topological type $\tau$. For each $j$ let
$$ \epsilon_j \colon H^0(E\otimes L_j^n)\otimes L_j^{-n} \to E$$
be the natural (surjective) evaluation maps. 

\begin{definition}\label{def:Esum}
  For an $(n,\underline{L})$-regular sheaf $E$ and subspaces $V_j'\subset H^0(E\otimes L_j^n)$ let $E'_j$ and $F'_j$ be the image and kernel of $\epsilon_j$ restricted to $V'_j$, so there is a short exact sequence $ 0\to F_j' \to V'_j\otimes L_j^{-n} \to E_j'\to 0$. Then, define a subsheaf of $E$ by $$ E_{\text{sum}} := E_{\text{sum}} (V_1',\ldots,V'_{j_0}) := E'_{1} + \cdots + E'_{j_0}$$
and let $K=K(V_1,\ldots,V_{j_0})$ be the kernel of the surjection $\bigoplus_j E'_j \to E_{\text{sum}}.$  We let $\mathcal S_1$ be the set of all sheaves $E_j',F'_j,E_{sum}$ and $K$ that arise in this way.
\end{definition}

Since the set of $(n,\underline{L})$-regular sheaves of topological type $\tau$ is bounded, and since for each such $E$ the possible $V_j'$ all live in a bounded family,  $\mathcal S_1$ is a bounded family.   Now let $\mathcal S_2$ be the set of saturated subsheaves $E'\subset E$ where $E$ is $(p,\underline{L})$-regular of topological type $\tau$ and $\hat{\mu}^{\sigma}(E')\ge \hat{\mu}^{\sigma}(E)$.  Then, as in the proof of Lemma \ref{lem:openness} since each sheaf in $\mathcal S_2$ is assumed to be saturated, Grothendieck's Lemma implies $\mathcal S_2$ is also bounded.\medskip

\noindent (\hypertarget{C4}{C4}) All the sheaves in $\mathcal S_1\cup \mathcal S_2$ are $(m,\underline{L} )$-regular. \medskip

\noindent (\hypertarget{C5}{C5})
  Let $P_j(k) = \chi(E\otimes L_j^k)$ for a sheaf of topological type $\tau$, so $P_E^{\sigma} = \sum_j \sigma_j P_j$.    Then for any integers $c_j\in \{ 0,\ldots,P_j(n)\}$ and sheaves $E'\in \mathcal S_1\cup \mathcal S_2$ the polynomial relation $ P_E^{\sigma} \sum_j \sigma_j c_j \sim P_{E'}^{\sigma} P_E^{\sigma}(n)$ is equivalent to the relation $ P_E^{\sigma} (m) \sum_j \sigma_j c_j \sim P_{E'}^{\sigma}(m) P_{E}^{\sigma}(n)$, where $\sim$ is any of $\le$ or $<$ or $=$.\medskip

This last condition is possible since there are only a finite number of different topological types arising from the different sheaves $E'$ in the bounded family $\mathcal S_1\cup \mathcal S_2$.   So, as the $c_j$ are all bounded, the above gives a finite number of numerical conditions on these polynomials.  Each such condition can be satisfied, since an inequality between polynomials $p,q\in \mathbb R[l]$ is equivalent to the same inequality holding with $l=m$ for some/all sufficiently large $m\in \N$.

\begin{remark}\label{rmk:howlargem}
For later reference we emphasise that what we have actually shown is that if $n$ is chosen so that condition (\hyperlink{C2}{C2}) holds (i.e., so the Le Potier-Simpson Theorem, Theorem~\ref{thm:detectstabilitysections}, holds) then conditions (\hyperlink{C3}{C3}) - (\hyperlink{C5}{C5}) hold for all $m$ sufficiently large.
\end{remark}

\subsection{Slope of Modules}
We next recast the stability of an $A$-module in terms of a ``slope'' function. Let $\sigma = (\underline {L}, \sigma_1, \dots, \sigma_{j_0})$ be a stability parameter.

\begin{definition}[Slope of an $A$-module]\label{def:slopemodule}
Let $M = \bigoplus_j V_j\oplus W_j$ be an $A$-module with either $\sum_j \sigma_j \dim V_j>0$ or $\sum_j \sigma_j \dim W_j>0$.  Then, the \emph{slope} of $M$ is
$$ \mu(M):=\mu_{\sigma}(M) := \frac{\sum_j \sigma_j \dim V_j}{\sum_j \sigma_j \dim W_j},$$
which takes values in the ordered interval $[0,\infty]$.
\end{definition}

We now fix the dimension vector of the modules we wish to consider. For this, let $p,n,m$ satisfy conditions (\hyperlink{C1}{C1})-(\hyperlink{C5}{C5}) and let $\underline{d} = (d_{11}, d_{12}, \ldots, d_{{j_01}}, d_{{j_0}2})$, where
\begin{equation}
d_{j1} = h^0(E\otimes L_j^n) = P_E^{L_j}(n) \text{ and } d_{j2} = h^0(E\otimes L_j^m)=P_E^{L_j}(m)\label{eq:fixdimension}
\end{equation}
for any $(n,\underline{L})$-regular sheaf $E$ of topological type $\tau$.  Recall that if $M$ is an $A$-module with dimension vector $\underline{d}$ and $M'=\bigoplus_j V_j'\oplus W_j'$ is a submodule of $M$, we have set $ \theta(M'):=\theta_\sigma(M') = \sum\nolimits_{j} \theta_{j1} \dim V_j' + \sum\nolimits_{j} \theta_{j2} \dim W_j'$, where 
\begin{equation}
 \theta_{j1} := \frac{\sigma_j}{\sum \sigma_i d_{i1}},\;\theta_{j2} := \frac{- \sigma_j}{\sum \sigma_i d_{i2}}\quad \quad \text{ for }j= 1, \ldots, j_0.
\end{equation}
Also, $M$ was defined to be semistable (with respect to $\sigma$) if  $\theta_{\sigma}(M')\le 0$ for all submodules $M'$, and if $M$ is semistable we say that a proper submodule $M'$ is \emph{destabilising} if $\theta_{\sigma}(M')=0$.

\begin{lemma}[Detecting semistability via slopes]\label{lem:thetamu}
Let $M$ be an $A$-module of dimension vector $\underline{d}$, and $M'=\bigoplus V_j'\oplus W_j'$ be a submodule of $M$ with $\sum_j \sigma_j \dim W_j'>0$.  Then,
$$ \theta(M')(\le) 0 \text{ if and only if } \mu(M')(\le) \mu(M).$$
\end{lemma}
\begin{proof}
We have
\begin{align*}
\theta(M') &= \sum\nolimits_j \theta_{j1} \dim V_j' + \sum\nolimits_j \theta_{j2} \dim W_j'=\frac{\sum_j \sigma_j \dim V_j'}{\sum_i \sigma_i d_{i1}} - \frac{\sum_j \sigma_j \dim W_j'}{\sum_i \sigma_i d_{i2}} \\&= \frac{\sum_j \sigma_j \dim W_j'}{\sum_i \sigma_i d_{i1}} \bigl(\mu(M')  - \mu(M)\bigr),
\end{align*}
from which the statement follows immediately.
\end{proof}

\begin{definition}[Degenerate submodules]\label{def:irrelevantmodule}
Let $M'=\bigoplus_j V'_j\oplus W'_j$ be a submodule of an $A$-module $M$.  We say that $M'$ is \emph{degenerate} if $V'_j=\{0\}$ for all $j$, and $W'_i=\{0\}$ for all $i$ such that $\sigma_i\neq 0$.
\end{definition}
The reason to introduce this terminology is from the following statement, which reduces the check of semistability to a slope inequality among non-degenerate submodules:

\begin{lemma}\label{lem:slopemustability}
Let $E$ be $(n,\underline{L})$-regular of topological type $\tau$ and set $M=\Hom(T,E)$.  Let $M' =\bigoplus V_j'\oplus W_j'$ be a submodule of $M$.
\begin{enumerate}
\item If $M'$ is non-degenerate, then $\sum_j \sigma_j \dim W_j'>0$; i.e., the slope $\mu(M')$ is well-defined for all non-degenerate submodules of $M$.
\item $M$ is semistable if and only if  $\mu(M')\le \mu(M)$ for all non-degenerate submodules $M'$ of $M$.
\item Suppose $M$ is semistable.  Then, a proper submodule $M'$ is destabilising if and only if either it is degenerate or it is non-degenerate with $\mu(M') = \mu(M)$.
\end{enumerate}
\end{lemma}

\begin{proof}
Write $M = \Hom(T,E) = \bigoplus_j H^0(E\otimes L_j^n) \oplus H^0(E\otimes L_j^m)$ along with the natural multiplication morphisms $\phi_{ij}\colon H^0(E\otimes L_i^n) \otimes H_{ij} \to H^0(E\otimes L_j^m)$ where $H_{ij}=H^0(L_j^m \otimes L_i^{-n})$.  Observe that since $L_i^{-n}\otimes L_j^m$ is globally generated by (\hyperlink{C3}{C3}), if $s\in H^0(E\otimes L_i^n)$ is such that $\phi_{ij}(s\otimes h)=0$ for all $h\in H_{ij}$ then $s=0$.

To show the first statement, suppose $\sum_j \sigma_j \dim W_j'=0$.  Then, clearly $W'_{i}=\{0\}$ for all $i$ such that $\sigma_{i}\neq 0$.  Choose some $r$ so that $\sigma_{r}>0$, which implies $W'_{r}=\{0\}$.  Then, as $M'$ is a submodule of $M$, $\phi_{ir}(V'_i\otimes H_{ir}) \subset W'_{r} =\{0\}$ for all $i$.  By the previous paragraph this implies $V'_i=\{0\}$ for all $i$. We conclude that $M'$ is degenerate, as claimed in (1).

For the second statement, note that $\theta(M')=0$ if $M'$ is degenerate.  On the other hand, if $M'$ is non-degenerate, then by (1) and Lemma \ref{lem:thetamu} we have $\theta(M')\le 0$ if and only if $\mu(M') \le \mu(M)$, proving (2). Statement (3) can be proven with similar arguments.
\end{proof}

\begin{remark}
A pure sheaf $E$ is (semi)stable with respect to $\sigma$  if and only if for all proper subsheaves $F\subset E$ we have
$$
\frac{P^{\sigma}_F(n)}{P^{\sigma}_F(m)}
(\le) 
\frac{P^{\sigma}_E(n)}{P^{\sigma}_E(m)}
\quad \text{ for } m\gg n\gg 0.$$
This follows quickly from the following observation:  for two monic polynomials $P$ and $Q$ of the same degree one has $Q(\le)P$ if and only if $P(m)/Q(m)(\le)P(n)/Q(n)$ for all $m\gg n\gg 0$. In order to see this, just write $P=Q+R$ and note that when $m$ tends to infinity $R(m)/Q(m)$ tends to $0$ through positive (resp. negative) values depending on the positivity of $R$. 

  Although we will not use this statement directly, it illustrates
  the relationship between stability of sheaves and quiver
  representation.  Since higher cohomology will vanish for large $n,m$, it says that
  a sheaf $E$ is (semi)stable if and only if for all proper subsheaves $F\subset E$
$$
\frac{\sum_j \sigma_j h^0(F\otimes L_j^n)}{\sum_j \sigma_j h^0(F\otimes
  L_j^m )} (\le) \frac{\sum_j \sigma_j h^0(E\otimes L_j^n)}{\sum_j \sigma_j
  h^0(E\otimes L_j^m )} \quad \text{ for all } m\gg n\gg 0$$
which, by definition, holds if and only if $\mu(\Hom(T,F))(\le) \mu(\Hom(T,E))$.  Thus, the main task of the subsequent sections will be to ensure that one can take $m,n$ uniformly over all (relevant) sheaves, and to prove that to test for stability of $\Hom(T,E)$ it is sufficient to consider only submodules of the form $\Hom(T,F)$ for some subsheaf $F\subset E$.
\end{remark}

\subsection{Tight Submodules}

Our next task is to simplify the stability condition on a module of the form $\Hom(T,E)$, where $E$ is an $(n,\underline{L})$-regular sheaf.  Roughly speaking, we show that in order to test $\Hom(T,E)$ for stability it is sufficient to restrict our attention to ``tight submodules'' as in the following definition, and moreover that these special submodules essentially arise as $\Hom(T,E')$ for some subsheaf $E'\subset E$.

\begin{definition}
  Let $M'=\bigoplus_j V_j'\oplus W_j'$ and  $M''=\bigoplus_j V_j''\oplus W_j''$ be two submodules of a given $A$-module $M$.  We say that $M'$ is \emph{subordinate} to $M''$ if
  \begin{equation}\label{eq:subordinate}
 V_j'\subset V_j'' \text{ and } W''_j\subset W'_j \quad \text{ for all } j.
\end{equation}
We say that $M'$ is \emph{tight} if whenever $M'$ is subordinate to a submodule $M''$ we have
\begin{equation}
  V_j' = V_j'' \text{ and } W_j' = W_j'' \text{ for all } j \text{ such that } \sigma_j\neq 0.\label{eq:tight}
\end{equation}
\end{definition}

Directly from the definition we have the following:
\begin{lemma}\label{lem:tightslope}
Suppose that $M'$ and $M''$ are such that $\mu(M')$ and $\mu(M'')$ are well defined.   If $M'$ is subordinate to $M''$, then $\mu(M')\le \mu(M'')$, and if moreover $M'$ is tight, then $\mu(M')=\mu(M'')$.
\end{lemma}

\begin{lemma}\label{lem:subordtight}
Let $\widetilde{M}=\bigoplus_j \widetilde{V}_j \oplus \widetilde{W}_j$ be a submodule of an $A$-module $M$. Then $\widetilde{M}$ is subordinate to some  tight submodule $M'=\bigoplus_j V_j'\oplus W_j'$ of $M$. 
\end{lemma}
\begin{proof}
Associate to  $M = \bigoplus_j V_j \oplus W_j$ the collection of linear maps 
$$ \phi_{jk} \colon V_j \otimes H_{jk} \to W_k \quad j,k=1,\ldots,j_0$$
coming from its $A$-module structure.
Define $ W_j' := \sum_l \phi_{lj}(\widetilde{V}_l\otimes H_{lj})$ and $$ V_j' := \{ v\in V_j : \phi_{jk}(v\otimes h)\in W_k' \text{ for all } k \text{ and all } h\in H_{jk}\}.$$ So, by definition  $M':=\bigoplus V_j'\oplus W_j'$ is a submodule of $M$. Notice that as $\widetilde{M}$ is a submodule we have $W_j'\subset \widetilde{W}_j$.  On the other hand, if $v\in \widetilde{V}_j$ then for any $k$ and any $h\in H_{jk}$ we have ${\phi_{jk}}(v\otimes h)\in {\phi_{jk}}(\widetilde{V}_{j}\otimes H_{jk}) \subset W_{k}'$,
where the last inclusion comes from the definition of $W_k'$.  Thus, $\widetilde{V}_j\subset V_j'$, and so $\widetilde{M}$ is subordinate to $M'$.

A similar elementary argument shows that $M'$ is tight.  For suppose that $M'$ is subordinate to a submodule $M'' = \bigoplus V_j'' \oplus W_j''$.  Then $\widetilde{V}_l\subset V_{l}'\subset V_{l}''$, so
$$W_{j}' \subset \sum\nolimits_l \phi_{lj}(V_{l}''\otimes H_{lj}) \subset W_j'',$$
and so $W_{j}'= W_j''$.  On the other hand, if $v\in V_j''$, then for any $k$ and any $h\in H_{jk}$ we have $\phi_{jk}(v\otimes h) \in W_j'' = W_j'$, and so $V_{j}''=V'_{j}$ as well.
\end{proof}

\begin{lemma}\label{lem:sufficienttight}
Let $E$ be $(n,\underline{L})$-regular of topological type $\tau$ and set $M=\Hom(T,E)$.  Then, $M$ is semistable if and only if $\mu(M')\le \mu(M)$ for all tight non-degenerate submodules $M'$ of $M$.
\end{lemma}
\begin{proof}
To prove this, suppose $\mu(M')\le \mu(M)$ for all non-degenerate tight submodules $M'$, the other implication being clear.  If $\widetilde{M}$ is any non-degenerate submodule, then it is subordinate to some tight $M'$ by Lemma \ref{lem:subordtight}.  If $M'$ is degenerate, then $\widetilde{V}_j\subset V_j'=\{0\}$ for all $j$, so $\mu(\widetilde{M})=0\le\mu(M)$, and we are done. Otherwise, it follows from Lemma \ref{lem:tightslope} that $\mu(\widetilde{M}) \le \mu(M')\le \mu(M)$.  Thus, the result is a consequence of Lemma \ref{lem:slopemustability}(2).  
\end{proof}

\begin{proposition}\label{prop:tightmodulescomefromsheaves}
  Let $E$ be an $(n,\underline{L})$-regular sheaf of type $\tau$.  Suppose that $M'=\bigoplus_j V_j'\oplus W_j'$ is a submodule of $\Hom(T,E)$ and set $$E':=E_{sum}(V_1',\ldots,V_{j_0}').$$
Then, $M'$ is subordinate to $\Hom(T,E')$.  If moreover $M'$ is tight and non-degenerate, then $\mu(M') = \mu(\Hom(T,E'))$.
\end{proposition}
\begin{proof}
Let $E_j'$ be the sheaf generated by $V_j'\otimes L_j^{-n}$, so $E' = E'_1+ \cdots + E'_{j_0}$.    Note that we certainly have
\begin{equation}
  \label{eq:sub1}
  V_j'\subset H^0(E_j'\otimes L_j^n)\subset H^0(E'\otimes L_j^n).
\end{equation}
Now applying (\hyperlink{C4}{C4}) to the short exact sequences $0\to F_j\to V_j'\otimes L_j^{-n} \to E_j'\to 0$ we have that $F_j$ is $(m,\underline{L} )$-regular, and hence for any $j,i$ the multiplication map
$$ V_j'\otimes H^0(L_i^m\otimes L_j^{-n}) \to H^0(E_j'\otimes L_i^m)$$
is surjective.  Now consider the short exact sequence $0\to K \to \bigoplus\nolimits_j E_j' \to E'\to 0$. Again from (\hyperlink{C4}{C4}) the sheaves $K$ and $E'$ are $(m,\underline{L} )$-regular.  So, for any $i$ the composition
$$ \bigoplus\nolimits_j V_j'\otimes H^0(L_i^m\otimes L_j^{-n}) \to \bigoplus\nolimits_j H^0(E_j'\otimes L_i^m) \to H^0(E'\otimes L_i^m)$$
is surjective.  But this composition is just the direct sum of the natural multiplication maps $V_j'\otimes H^0(L_i^m\otimes L_j^{-n})\to H^0(E\otimes L_i^m)$, whose image lies in $W_i'$, since $M'$ is a submodule of $\Hom(T,E)$.    Thus, we conclude
\begin{equation}
H^0(E'\otimes L_i^m) \subset W_i' \text{ for all } i.\label{eq:sub2}
\end{equation}
Now \eqref{eq:sub1} and \eqref{eq:sub2} together imply that $M'$ is subordinate to $\Hom(T,E')$, which proves the first statement.    

If $M'$ is tight and non-degenerate, then the equality of slopes follows from Lemma \ref{lem:tightslope}. Indeed, $\Hom(T,E')$ is non-degenerate, as by (\hyperlink{C4}{C4}) $E'$ is $(m,\underline{L})$-regular and so certainly $H^0(E\otimes L_j^m)\neq 0$ for all $j$.
\end{proof}

\subsection{Sheaves and Modules: Semistability}
Our aim in this section is to prove Theorem~\ref{thm:semistability} below, which compares semistability of sheaves $E$ to semistability of modules of the form $\Hom(T, E)$.
\begin{lemma}\label{lem:Homsemistable}
  Suppose that $E$ is $(n,\underline{L})$-regular of topological type $\tau$.  Then, the following are equivalent:
\begin{enumerate}
\item $\Hom(T,E)$ is semistable.
\item For any subsheaf $E'\subset E$ we have
  \begin{equation}
    \sum\nolimits_j \sigma_j h^0(E'\otimes L_j^n) P_E^{\sigma}(m) \le \sum\nolimits_j \sigma_j h^0(E'\otimes L_j^m) P_E^{\sigma}(n).\label{eq:stabilitycondition2}
\end{equation}
\item The inequality \eqref{eq:stabilitycondition2} holds for any subsheaf of the form $E'=E_{sum}(V_1',\ldots,V_{j_0}')$ with $V_j'\subset H^0(E\otimes L_j^n)$, $j= 1, \dots, j_0$.
\end{enumerate}
\end{lemma}
\begin{proof}
If $\Hom(T,E)$ is semistable and $E'\subset E$, then
\begin{align*}0\ge \theta(\Hom(T,E')) &= \frac{\sum_j \sigma_j h^0(E'\otimes L_j^n)}{\sum_i \sigma_i d_{i1}} - \frac{\sum_j \sigma_j h^0(E'\otimes L_j^m)}{\sum_i \sigma_i d_{i2}}\\
&=\frac{\sum_j \sigma_j h^0(E'\otimes L_j^n)}{P_{E}^\sigma(n)} - \frac{\sum_j \sigma_j h^0(E'\otimes L_j^m)}{P_{E}^\sigma(m)},
\end{align*}
where the last equality uses the regularity assumption on $E$.  Thus, (1) implies (2). Clearly, (2) implies (3), so assume that (3) holds.    If $M'=\bigoplus_j V_j'\oplus W_j'$ is any tight non-degenerate submodule of $\Hom(T,E)$, then by Proposition \ref{prop:tightmodulescomefromsheaves} we know that $\mu(M') = \mu(\Hom(T,E'))$ for some subsheaf $E' \subset E$ of the form $E'=E_{sum}(V_1',\ldots,V_{j_0}')$.   Therefore, by (3) we get $\mu(M') \le \mu(M)$.    Since this holds for any non-degenerate tight submodule, using Lemma \ref{lem:sufficienttight} we conclude that $\Hom(T,E)$ is semistable.
\end{proof}

\begin{theorem}[Comparison of semistability]\label{thm:semistability}
Suppose $n,m,p$ satisfy conditions (\hyperlink{C1}{C1})-(\hyperlink{C5}{C5}). Then, a sheaf $E$ of topological type $\tau$ is semistable if and only if it is pure and $(p,\underline{L})$-regular, and $\Hom(T,E)$ is semistable.
\end{theorem}
\begin{proof}
Suppose first that $E$ is semistable.  Then, by definition it is pure and  $(p,\underline{L})$-regular by (\hyperlink{C1}{C1}) and thus also $(n,\underline{L})$-regular.  Let $E'=E_{sum}(V_1',\ldots,V_{j_0}')$ for some $V_j'\subset H^0(E\otimes L_j^n)$.    By (\hyperlink{C5}{C5}),  Corollary~\ref{cor:lepotiersimpson}, and (\hyperlink{C4}{C4}), we have
$$ \sum\nolimits_j \sigma_j h^0(E'\otimes L_j^n)P_E^{\sigma}(m) \le P_{E'}^{\sigma}(m) P_E^{\sigma}(n)  = \sum\nolimits_j \sigma_j h^0(E' \otimes L_j^m) P_E^{\sigma}(n),$$
and thus $\Hom(T,E)$ is semistable by Lemma \ref{lem:Homsemistable}.

Conversely, suppose that $E$ is pure and $(p,\underline{L})$-regular, and that $\Hom(T,E)$ is semistable.   Let $E'\subset E$ be a saturated subsheaf with $\hat{\mu}^{\sigma}(E')\ge \hat{\mu}^{\sigma}(E)$.  Then, by Lemma \ref{lem:Homsemistable} we know that
$$ \sum\nolimits_j \sigma_j h^0(E'\otimes L_j^n) P_E^{\sigma}(m) \le \sum\nolimits_j \sigma_j h^0(E'\otimes L_j^m) P_E^{\sigma}(n).$$
But since $E'$ is a saturated subsheaf with this assumed lower bound in its slope we have $E\in \mathcal S_2$ (as defined just before condition \hyperlink{C4}{C4}).  Thus, we can use (\hyperlink{C5}{C5}) to deduce that the previous inequality implies the inequality of polynomials
$$ \sum\nolimits_j \sigma_j h^0(E'\otimes L_j^n) P_E^{\sigma} \le  P_E^{\sigma}(n) P_{E'}^{\sigma}.$$
Hence, applying the implication ``(3) $\Rightarrow$ (1)'' of Corollary~\ref{cor:lepotiersimpson} we conclude that $E$ is semistable, as required.
\end{proof}

\subsection{Sheaves and Modules: $S$-equivalence}

Having compared semistability of sheaves with semistability of modules, we now turn to Jordan-H\"older filtrations. To obtain a similar comparison result, we need to assume that $\sigma$ is a positive stability parameter (and thus there are no non-trivial degenerate submodules). 

\begin{lemma}\label{lem:regularitydestabilising}
Suppose that $E$ is semistable of topological type $\tau$.

\begin{enumerate}
\item  If $E'\subset E$ is destabilising subsheaf then  $E'$ is $(p,\underline{L})$-regular and $\Hom(T,E')$ is a destabilising subsheaf of $\Hom(T,E)$.
\item If $E'_1$ and $E'_2$ are two destabilising subsheaves of $E$ and $E'_1\subset E'_2$ then
$$ \Hom(T,E'_2) / \Hom(T,E'_1) \simeq \Hom(T,E_2'/E_1').$$
\end{enumerate}
\end{lemma}
\begin{proof}
For the first statement, since $E$ is semistable and $E'$ has the same reduced multi-polynomial, we have that $E'$ is also semistable.  Letting $E'':= E/E'$, one checks easily that the same holds for $E''$.  Thus, $E'\oplus E''$ is a semistable sheaf of topological type $\tau$, and so by (\hyperlink{C1}{C1}) is $(p,\underline{L})$-regular, and thus the same holds for $E'$.  Consequently, $E'$ is $(n,\underline{L})$-regular and $(m,\underline{L})$-regular, and so
$$\mu(\Hom(T,E')) = \frac{\sum_j \sigma_j h^0(E'\otimes L_j^n)}{\sum_j \sigma_j h^0(E'\otimes L_j^m)} = \frac{P^{\sigma}_{E'}(n)}{P^{\sigma}_{E'}(m)}=  \frac{P^{\sigma}_{E}(n)}{P^{\sigma}_{E}(m)} = \mu(\Hom(T,E)),$$
so $\Hom(T, E')$ destabilises $\Hom(T, E)$, as claimed.

For the second statement, we know that $E_1'$ and $E_2'$ are $(n,\underline{L})$-regular by part (1), and so $\Ext^1(T,E_1)=0$.  Thus, applying $\Hom(T,-)$ to the short exact sequence $0\to E_1'\to E_2'\to E_2'/E_1'\to 0$ yields the short exact sequence $0\to \Hom(T,E_1')\to \Hom(T,E_2') \to \Hom(T,E_2'/E_1')\to 0$, which completes the proof.
\end{proof}

\begin{lemma}\label{lem:istight}
Let $\sigma$ be a positive stability parameter and suppose that $M'$ is a destabilising submodule of a semistable module $M$.  Then, $M'$ is tight.
\end{lemma}
\begin{proof}
Suppose $M'$ is subordinate but not equal to a submodule $M''$ of $M$ contradicting tightness, then there is a $j$ with $\sigma_j\neq 0$ such that either $V_{j}'\subsetneq V''_j$ or $W_{j}''\subsetneq W'_j$.  But in both cases this implies $\mu(M)=\mu(M')<\mu(M'')$ which is impossible as $M$ is semistable.   
\end{proof}

\begin{lemma}\label{lem:destabilisingmodulesgivesdestabilisingsheaves}
Let $\sigma$ be a positive stability parameter and suppose $E$ is semistable of topological type $\tau$ and $M'=\bigoplus V_j'\oplus W_j'$ is a destabilising submodule of $\Hom(T,E)$.  Then, $E' := E_{sum}(V_1',\ldots,V_{j'})$ is either a destabilising subsheaf of $E$ or equals $E$.
\end{lemma}
\begin{proof}
From Theorem \ref{thm:semistability} we know that $\Hom(T,E)$ is semistable.  By the regularity of $E$ we know $\mu(\Hom(T,E)) = P_E^{\sigma}(n)/P_E^{\sigma}(m)$.   Now, by Lemma \ref{lem:istight} we have that $M'$ is tight and so by Proposition \ref{prop:tightmodulescomefromsheaves} subordinate to $\Hom(T,E')$. Moreover, $\mu(M') = \mu(\Hom(T,E'))$.  Taken together, these observations say that we have
$$
\frac{\sum_j \sigma_j h^0(E'\otimes L_j^n)}{\sum_j \sigma_j h^0(E'\otimes L_j^m)} = \frac{P_E^{\sigma}(n)}{P_E^{\sigma}(m)}.\label{eq:sameslope2}
$$
This equality together with  (\hyperlink{C5}{C5}) applied to $E'$ and to $c_j=h^0(E'\otimes L_j^n)\le h^0(E\otimes L_j^n)=P_j(n)$ give 
$$\sum\nolimits_j \sigma_j h^0(E'\otimes L_j) P_E^{\sigma} = P_E^{\sigma}(n) P_{E'}^{\sigma}.$$
Thus, by the final statement of Corollary \ref{cor:lepotiersimpson} we conclude that $E'$, if not equal to $E$, will destabilise $E$.
\end{proof}

\begin{theorem}[Comparison of Jordan-H\"older filtrations]\label{thm:sequivalence}
Let $\sigma$ be a positive stability parameter, and suppose $n$, $m$ and $p$ are chosen such that the Embedding  Theorem \ref{thm:categoryembedding} holds and such that conditions (\hyperlink{C1}{C1})-(\hyperlink{C5}{C5}) are satisfied.  Then, for any semistable sheaf $E$ of topological type $\tau$,
$$ \Hom(T,gr E)\simeq gr \Hom(T,E).$$
Thus, if $E$ is stable, then $\Hom(T,E)$ is stable.  Moreover, two such semistable sheaves $E$ and $E'$ are $S$-equivalent if and only if $\Hom(T,E)$ and $\Hom(T,E')$ are $S$-equivalent.
\end{theorem}

\begin{proof}
Set $M= \Hom(T,E)$, which is semistable by Theorem \ref{thm:semistability}.    Let $0=E_0\subsetneq E_1\subsetneq \cdots\subsetneq E_l=E$ be a Jordan-H\"older filtration of $E$.  So, by definition each $E_i$ is destabilising, and the filtration is maximal with this property.  By Lemma \ref{lem:regularitydestabilising} each $E_i$ is $(p,\underline{L})$-regular and the module $M_i:=\Hom(T,E_i)$  is destabilising in $M$.  Thus, we have a filtration $ 0 = M_0 \subsetneq M_1 \subsetneq \cdots \subsetneq M_l = M$ of destabilising submodules. 

We claim that this is in fact a Jordan-H\"older filtration of $M$, i.e., that it is maximal among such filtrations. To this end suppose that 
\begin{equation}
 M_p\subset M'\subset M_{p+1}\label{eq:inclusionJHM}
\end{equation}
for some destabilising submodule $M'=\bigoplus_j V_j'\oplus W_j'$ of $M$.  By Lemma \ref{lem:destabilisingmodulesgivesdestabilisingsheaves} the sheaf $E':=E_{sum}(V_1',\ldots,V_{j_0}')$
is a destabilising subsheaf of $E$  or equals $E$.

 We claim that 
\begin{equation}
  \label{eq:inclusion}
  E_p\subset E'\subset E_{p+1}.
\end{equation}

To prove this, observe $M'$ is tight from Lemma \ref{lem:istight} and subordinate to $\Hom(T,E')$ by Proposition \ref{prop:tightmodulescomefromsheaves}. Hence, by the definition of being tight \eqref{eq:tight} we have $V'_j = H^0(E'\otimes L_j^n)$ and $W_j' = H^0(E'\otimes L_j^m)$ for all $j$ (here again we use that $\sigma$ is positive, so this holds for all $j$), and thus $M'=\Hom(T,E')$.   Therefore, the above inclusion \eqref{eq:inclusionJHM} in particular implies
$$ H^0(E_p\otimes L_j^m) \subset H^0(E'\otimes L_j^m)\subset H^0(E_{p+1}\otimes L_j^m).$$
But both the sheaves $E_p$ and $E'$ are $(m,\underline{L})$-regular, and so both of  $E_p\otimes L_j^m$ and $E'\otimes L_j^m$ are globally generated, which gives \eqref{eq:inclusion}.

Hence, by maximality of the original Jordan-H\"older filtration given by the $E_i$,  we must have either $E'=E_p$ or $E'=E_{p+1}$.  Thus either $M'=M_p$ or $M'=M_{p+1}$ 
so the $M_p$ do in fact give a Jordan-H\"older filtration of $M$ as claimed.

Using the above, we compute
\begin{align*}
 gr \Hom(T,E)  &= \oplus_p M_p/M_{p+1} = \oplus_p \Hom(T,E_p)/\Hom(T,E_{p+1})\\
& = \oplus_p \Hom(T,E_{p}/E_{p+1}) = \Hom(T, gr E),
\end{align*}
where the penultimate equality comes from the second statement in Lemma \ref{lem:regularitydestabilising}.

Finally, if $E$ is stable, then $gr(E)=E$, so $gr \Hom(T,E)  = \Hom(T,E)$, which is thus stable.  For the statement about $S$-equivalence, clearly if $E$ and $E'$ are $S$-equivalent then the same is true for $\Hom(T,E)$ and $\Hom(T,E')$.  On the other hand, if $\Hom(T,E)$ and $\Hom(T,E')$ are $S$-equivalent, then $\Hom(T,gr(E))$ is isomorphic to $\Hom(T,gr(E'))$, which implies that $gr(E)$ is isomorphic to $gr(E')$,  since we are assuming $n$ and $m$ are chosen so the map $\acts \mapsto \Hom(T, \acts)$ is fully faithful by Theorem \ref{thm:categoryembedding}. This shows that $E$ and $E'$ are $S$-equivalent. Note that Theorem \ref{thm:categoryembedding} may be invoked since both $gr(E)$ and $gr(E')$ are semistable of topological type $\tau$ and hence $(p,\underline{L})$-regular by condition (\hyperlink{C1}{C1}). 
\end{proof}

\section{Construction of moduli spaces}\label{sect:constructionofmoduli}

In this section, we construct the moduli space for multi-Gieseker-semistable sheaves.   As before, let $X$ be a projective scheme of finite type over an algebraically closed field $k$ of characteristic zero, and $\sigma = (\underline{L},\sigma_1,\ldots,\sigma_{j_0})$ be a rational bounded stability parameter. Removing any of the $L_j$ for which $\sigma_j=0$ does not affect the definition of (semi)stability, cf.~Definition~\ref{defi:multiGiesekerstability}.  Thus, we may without loss of generality assume that $\sigma$ is positive.

The moduli functor we wish to consider is
$$\underline{\mathcal M}_{\sigma}\colon (Sch/k)^{\circ} \to (Sets),$$
assigning to a scheme $S$ the set $\underline{\mathcal M}_\sigma(S)$ of isomorphism classes of $S$-flat families of sheaves on $X$ that are semistable with respect to $\sigma$ and have topological type $\tau$. Moreover, if $f\colon S'\to S$ is a morphism, then $\underline{\mathcal M}_\sigma(f)$ is the map obtained by pulling back sheaves via $f\times \id_X$. Here, a \emph{moduli space of $\sigma$-semistable sheaves} is a scheme that corepresents $\underline{\mathcal M}^{\sigma}$ (for the basic terminology concerning moduli spaces and the corepresentation of functors adopted, see \cite[Sect.~4.4 and 4.5]{ConsulKing}). 

\subsection{Constructing quasi-projective moduli spaces by GIT}\label{subsect:GITmoduli}
First, we choose natural numbers $p,n,m \in \mathbb{N}$ such that the Comparison of semistability and Jordan-H\"older filtrations between sheaves and modules holds, i.e., such that Theorem~\ref{thm:mainsemistabilitycomparison} holds.   Moreover, by increasing $m$ if necessary, we may assume that the assertions of Theorem~\ref{thm:categoryembedding} and hence those of Propositions~\ref{prop:embedding_familyversion} and \ref{prop:identifyingtheimage} also hold. Note that by assumption, every semistable sheaf $E$ of topological type $\tau$ is  $(p,\underline{L})$-regular, and therefore also $(n, \underline{L})$- and $(m, \underline{L})$-regular. 

We now match up the discussion of the previous sections with the terminology introduced in Section~\ref{subsect:quivermodulispaces}. To ease notation let $P_j = P_E^{L_j}$ where $E$ is (any) sheaf of topological type $\tau$.  We consider the dimension vector  
$$\underline{d} = \bigl(P_1(n), P_1(m), \ldots, P_{j_0} (n), P_{j_0} (m)\bigr),$$ 
as introduced above \eqref{eq:fixdimension},
and let
\[R:= \mathrm{Rep}(Q, \underline{d}) =  \bigoplus\nolimits_{i,j=1}^{j_0} \mathrm{Hom}_k (k^{P_i(n)} \otimes H_{ij}, k^{P_j(m)} ) \]
be the representation space of the quiver $Q$ corresponding to the dimension vector $\underline{d}$. Here, as before we have used the notation $H_{ij} = H^0(L_i^{-n}\otimes L_j^m)$. The space $R$ carries a tautological family $\mathscr{M}$ of right $A$-modules, and therefore, Proposition~\ref{prop:identifyingtheimage} allows us to find a locally closed subscheme $$\bar \iota\colon~R^{[n\text{-}reg]}_{\tau} \hookrightarrow R$$ parametrising those modules that are in the image of the $\Hom(T, -)$-functor on the category of $n$-regular sheaves. Hence, $\bar \iota^*\mathscr{M} \otimes_{\mathscr{A}} T$ is an $R^{[n\text{-}reg]}_{\tau}$-flat family of $n$-regular sheaves of topological type $\tau$. Consequently, there exists an open subscheme $$\iota\colon~Q \hookrightarrow R^{[n\text{-}reg]}_{\tau}$$ that parametrises sheaves that are not only $(n,\underline{L})$-regular, but even $(p,\underline{L})$-regular.
We let
\begin{equation}\label{eq:iteratedinclusion}
Q^{[\sigma\text{-}s]} \subset Q^{[\sigma\text{-}ss]} \subset Q
\end{equation}
be the loci where the fibres of the tautological family $\mathscr{F} = \iota^* \mathscr{M} \otimes_\mathscr{A} T$ are stable and semistable, respectively. By Lemma~\ref{lem:openness}, these loci are open subschemes. The reductive group
\[G:=  \prod\nolimits_{j=1}^{j_0} \bigl(GL_k(P_j(n)) \times GL_k(P_j(m)) \bigr)\]
acts linearly on $R$ by conjugation, and the subschemes $R^{[n\text{-}reg]}_{\tau}$,  $Q$, $Q^{[\sigma\text{-}s]}$, and $Q^{[\sigma\text{-}ss]}$ are stable under the $G$-action. As explained in Section~\ref{subsubsect:stabilityandmodulispaces}, the vector $\sigma$ defines a rational character $\theta_\sigma$ of $G$ and hence a $G$-linearisation $\chi_\sigma:= \chi_{\theta_\sigma}$ of the trivial line bundle over $R$, and it follows from the work of King \cite{King} that we have the following connection between Geometric Invariant Theory and the moduli theory of semistable $A$-modules, cf.~\cite[Thm.~4.8]{ConsulKing}; see also the discussion and notation of Section~\ref{subsubsect:stabilityandmodulispaces}.

\begin{proposition}[Moduli of semistable modules via GIT]\label{prop:quivermoduli}
 If $R^{\sigma\text{-}ss}$ denotes the open subscheme of points that are GIT-semistable with respect to the linearisation $\chi_{\theta_{\sigma}}$, the good quotient 
\[\pi_A\colon R^{\sigma\text{-}ss} \to \mathcal{M}^{\sigma\text{-}ss}_A(\underline{d})\] exists, and the projective scheme $\mathcal{M}^{\sigma\text{-}ss}_A(\underline{d})$ is a moduli space for $\sigma$-semistable $A$-modules. The closed points of $\mathcal{M}^{\sigma\text{-}ss}_A(\underline{d})$ correspond to the $S$-equivalence classes of semistable $A$-modules. 

Moreover, if $R^{\sigma\text{-}s} \subset R^{\sigma\text{-}ss}$ denotes the open subscheme of stable points, the geometric quotient $R^{\sigma\text{-}s} /G =: \mathcal{M}^{\sigma\text{-}s}_A(\underline{d})$ exists, it embeds into $\mathcal{M}^{\sigma\text{-}ss}_A(\underline{d})$ as an open subscheme, and is a moduli space for stable $A$-modules. The closed points of $\mathcal{M}^{\sigma\text{-}s}_A(\underline{d})$ correspond to the isomorphism classes of stable $A$-modules.
\end{proposition}

It is our aim to show that $Q^{[\sigma\text{-}ss]}$ has a good quotient and that this good quotient restricts to a geometric quotient on $Q^{[\sigma\text{-}s]}$.  We do this with the help of the following well-known lemma, a proof of which can be found in \cite[Sect.~6.1]{ConsulKing}.
\begin{lemma}[Restricting good quotients to saturated open subsets]\label{lem:quotientrestriction}
 Let $\pi\colon~ Z \to Z \hq G$ be a good quotient for the action of a reductive algebraic group $G$ on a scheme $Z$, and let $Y$ be a $G$-invariant open subset of $Z$. Suppose further that for each $G$-orbit $G\acts y$ in $Y$, the (uniquely determined) closed orbit that is contained in the orbit closure $\overline{G \acts y}$ in $Z$ is already contained in $Y$. Then, $\pi$ restricts to a good quotient $Y \to Y \hq G$, where $Y\hq G = \pi(Y)$ is an open subset of $Z\hq G$.
\end{lemma}
The following is the main step in the construction of the desired moduli space.
\begin{proposition}[Sheaf-theoretic semistability and orbit closures]\label{prop:constructionquotient}
 The quotient $\pi_A\colon~R^{\sigma\text{-}ss} \to \mathcal{M}^{\sigma\text{-}ss}_A$ of Proposition~\ref{prop:quivermoduli} and the locally closed embedding $\iota\colon~ Q^{[\sigma\text{-}ss]} \hookrightarrow R^{\sigma\text{-}ss}$ uniquely determine a commutative diagram
\begin{equation}\label{eq:quotientsdiagram}
 \begin{gathered}
  \begin{xymatrix}{Q^{[\sigma\text{-}ss]} \ar@{^(->}[r]^{\iota} \ar[d]_{\pi} & R^{\sigma\text{-}ss} \ar[d]^{\pi_A} \\
 \mathcal{M}^{\sigma\text{-}ss}_{\underline{L}}(\tau) \ar@{^(->}[r]^\varphi & \mathcal{M}^{\sigma\text{-}ss}_A,
}
  \end{xymatrix}
 \end{gathered}
\end{equation}
where $\mathcal{M}^{\sigma\text{-}ss}_{\underline{L}}(\tau)$ is quasi-projective, $\pi$ is a good quotient, and $\varphi$ is the inclusion of a locally closed subscheme. 
\end{proposition}
\begin{proof}[Proof](cf.~Prop.~6.3 of \cite{ConsulKing})  Let $Y = Q^{[\sigma\text{-}ss]}$, and let $\overline{Y}$ be its closure in $R$. Furthermore, set $Z = \overline{Y} \cap R^{\sigma\text{-}ss}$, which is a closed $G$-stable subscheme of $R^{\sigma\text{-}ss}$. We claim that the assumptions of Lemma~\ref{lem:quotientrestriction} are fulfilled. Indeed, if $p \in Q^{[\sigma\text{-}ss]}$ corresponds to a module $M = \Hom(T, E)$, the closed orbit in $\overline{G \acts p}$ is the orbit corresponding to the graded module $gr M$, see \cite[Prop.~3.2]{King} or Theorem~\ref{thm:quivermodulicitation} above. However, part (2) of Theorem~\ref{thm:mainsemistabilitycomparison} states that 
\[gr M \cong \Hom(T, gr E),\]
and we know that $gr E$ is semistable. Hence, this closed orbit is also in $Q^{[\sigma\text{-}ss]}$, as claimed.

Now, Lemma~\ref{lem:quotientrestriction} implies that in order to prove our assertions it suffices to show that the closed subscheme $Z \subset R^{\sigma\text{-}ss}$ has a good quotient that embeds into $\mathcal{M}^{\sigma\text{-}ss}_A$. However, as our ground field has characteristic zero, the existence of a Reynolds operator implies that the scheme-theoretic image $\pi_A(Z)$ of the closed $G$-stable subscheme $Z \subset R^{\sigma\text{-}ss}$ is a closed subscheme of $\mathcal{M}^{\sigma\text{-}ss}_A$, and the restriction $\pi_A|_{Z}\colon Z \to \pi_A(Z)$ is a good quotient for the $G$-action on $Z$, cf.~\cite[statement (3) on p.~29]{MumfordGIT}.
\end{proof}
Finally, we are in the position to complete our construction of a moduli space for multi-Gieseker-semistable sheaves.
\begin{theorem}[Existence of moduli spaces for $(\underline{L}, \sigma)$-semistable sheaves]\label{thm:moduliexist}
The moduli space $\mathcal M_{\sigma}$  of $(\underline{L}, \sigma)$-semistable sheaves of topological type $\tau$ is given by the scheme $\mathcal{M}^{\sigma\text{-}ss}_{\underline{L}}(\tau)$ from Proposition~\ref{prop:constructionquotient}, i.e., it corepresents the moduli functor $\underline{\mathcal M}_{\sigma}$ of flat families of semistable sheaves. The closed points of $\mathcal{M}^{\sigma\text{-}ss}_{\underline{L}}(\tau)$ correspond to the $S$-equivalence classes of semistable sheaves. Furthermore, there is an open subscheme $\mathcal{M}^{\sigma\text{-}s}_{\underline{L}}(\tau) \subset  \mathcal{M}^{\sigma\text{-}ss}_{\underline{L}}(\tau)$ that corepresents the moduli functor of flat families of stable sheaves, and whose closed points correspond to the isomorphism classes of stable sheaves.
\end{theorem}
\begin{proof}
 Since $\pi$ is a good quotient, and hence in particular a categorical quotient, it follows from standard arguments, general principles, and part (1) of Theorem~\ref{thm:mainsemistabilitycomparison} that $\mathcal{M}^{\sigma\text{-}ss}_{\underline{L}}(\tau)$ corepresents the moduli functor of families of semistable sheaves; see for example \cite[Sect.~4.4 and Sect. 4.5]{ConsulKing} and \cite[Chap.~4.3]{Bible} for details. 

The remainder of the proof is mutatis mutandis the same as the one of \cite[Thm.~6.4]{ConsulKing}: The morphism $\varphi$ of diagram \eqref{eq:quotientsdiagram} induces a bijection between the closed points of $\mathcal{M}^{\sigma\text{-}ss}_{\underline{L}}(\tau)$ and the closed points of $\pi_A(Q^{[\sigma\text{-}ss]}) \subset  \mathcal{M}^{\sigma\text{-}ss}_A$. Consequently, Proposition~\ref{prop:quivermoduli} implies that the closed points of $\mathcal{M}^{\sigma\text{-}ss}_{\underline{L}}(\tau)$ correspond to the $S$-equivalence classes of semistable $A$-modules $M$ that are of the form $M = \Hom(T, E)$ for semistable sheaves of $E$ of topological type $\tau$. However, we also know from Part (2) of Theorem~\ref{thm:mainsemistabilitycomparison} that $\Hom(T, E)$ and $\Hom(T, E')$ are $S$-equivalent $A$-modules if and only if $E$ and $E'$ are $S$-equivalent sheaves. This implies the statement about the closed points of $\mathcal{M}^{\sigma\text{-}ss}_{\underline{L}}(\tau)$. For the part of the statement concerning stable sheaves, note that Theorem~\ref{thm:mainsemistabilitycomparison} implies that a semistable sheaf $E$ is stable if and only if the associated $A$-module $\Hom(T,E)$ is stable. It follows that $Q^{[\sigma\text{-}s]}= Q^{[\sigma\text{-}ss]} \cap R^{\sigma\text{-}s}$. In particular, all $G$-orbits in $Q^{[\sigma\text{-}s]}$ are closed in $Q^{[\sigma\text{-}ss]}$, since they are closed in $R^{\sigma\text{-}ss}$. We may hence apply Lemma~\ref{lem:quotientrestriction} to conclude that $Q^{[\sigma\text{-}s]}$ has a good, geometric quotient $\mathcal{M}^{\sigma\text{-}s}_{\underline{L}}(\tau) = Q^{[\sigma\text{-}s]} / G = \pi(Q^{[\sigma\text{-}s]}) $, which is open in $\mathcal{M}^{\sigma\text{-}ss}_{\underline{L}}(\tau) $, and corepresents the moduli functor of families of stable sheaves. Finally, the closed points of $\mathcal{M}^{\sigma\text{-}s}_{\underline{L}}(\tau)$ correspond to isomorphism classes of stable sheaves because $\mathcal{M}^{\sigma\text{-}s}_{\underline{L}}(\tau) = \pi(Q^{[\sigma\text{-}s]})$, and $S$-equivalence classes of stable sheaves are exactly the isomorphism classes.
\end{proof}

\subsection{Properness of the moduli spaces}

It follows from Proposition \ref{prop:constructionquotient} that the moduli space $\mathcal M_{\sigma}$ is quasi-projective.  To show that it is projective, we hence need to show it is proper. The proof of this statement is the same as that of the corresponding statement in \cite{ConsulKing}, which in turn depends on Langton's Theorem \cite{Langton}.

\begin{theorem}[Langton's Theorem]\label{thm:Langton}
Fix a rational stability parameter $\sigma$.    Let $C$ be the spectrum of a discrete valuation ring with generic point $C_0$ and suppose that $F$ is a flat family over $C_0$ of sheaves on $X$ that are semistable with respect to $\sigma$.    Then $F$ extends to a flat family of semistable sheaves over $C$.
\end{theorem}
\begin{proof}
Suppose first that $X$ is smooth.  By the proof of \cite[2.2.4]{Bible},  $F$ extends to a flat family over $C$.  Then, the proof of Langton's Theorem \cite[2.B.1]{Bible} holds verbatim, once the following is noticed: since $\sigma$ is rational, there is an $N$ such that the coefficients of the multi-Hilbert polynomial of any coherent sheaf on $X$ lies in the lattice $(1/r! N)\mathbb Z\subset \mathbb Q$ (one merely has to take $N$ to be sufficiently divisible to deal with the denominators that arise in the $\sigma_j$).  Thus, any descending sequence $\beta_j$ of such coefficients that is strictly positive will eventually become stationary.

Now, the case for a general $X$ follows from the smooth case: Without loss of generality we may assume that the $L_1,\ldots,L_{j_0}$ are very ample, and that their sections give an embedding $\iota\colon X\to \mathbb P^{n_1} \times \cdots \times \mathbb P^{n_{j_0}}=:\mathbb P$.  Then, a sheaf $F$ is (semi)stable on $X$ with respect to $(\sigma,\underline{L})$ if and only it is (semi)stable with respect to $(\sigma,(\mathcal O_{\mathbb P^{n_1}}(1), \ldots,  \mathcal O_{\mathbb P^{n_{j_0}}}(1))$.   Thus the statement for $X$ follows from that for $\mathbb P$.
\end{proof}

Using the previous result, the proof of the following is precisely the same as the corresponding one in \cite[Prop 6.6]{ConsulKing}, which we do not repeat here.

\begin{theorem}[Projectivity of moduli spaces]\label{thm:modulispaceproper}
  Suppose that $\sigma$ is a rational stability parameter.  Then, the moduli space $\mathcal M_{\sigma}$ is proper and thus projective.
\end{theorem}

\begin{remark}
In the statement of Proposition~\ref{thm:modulispaceproper}, we include the hypothesis that $\sigma$ is rational only to emphasise its importance in the proof of Langton's theorem, Theorem~\ref{thm:Langton} above.  As we have seen in Corollary \ref{cor:irrational}, this is not really necessary. 
\end{remark}

\part{Applications}
\section{Variation of multi-Gieseker moduli spaces}\label{subsect:MultiGiesekerVariation}
In this section we consider the variation of the moduli spaces $\mathcal{M}_\sigma$, as $\sigma$ varies. Let $X$ be a projective scheme over an algebraically closed field of characteristic $0$, let $\underline{L} = (L_1,\ldots,L_{j_0})$ be a vector of very ample line bundles on $X$, and fix a topological type $\tau \in B(X)_\mathbb{Q}$.

Suppose that $\Sigma$ is a finite set of bounded positive stability parameters (all taken with respect to the same very ample line bundles $\underline{L} = (L_1,\dots,L_{j_0})$).   We choose $m\gg n\gg p\gg 0$ so the assertions of Theorem~\ref{thm:mainsemistabilitycomparison}, Theorem~\ref{thm:categoryembedding}, Propositions~\ref{prop:embedding_familyversion} and \ref{prop:identifyingtheimage} hold for each $\sigma\in \Sigma$. Consider the union $Y =\bigcup_{\sigma\in \Sigma} Q^{[\sigma\text{-}ss]}$,
where $Q^{[\sigma\text{-}ss]}$ is the locally closed set parametrising modules of the representation space $R$ coming from sheaves that are semistable with respect to $\sigma\in \Sigma$, as well as its scheme-theoretic closure $$Z := \overline{\bigcup_{\sigma\in \Sigma} Q^{[\sigma\text{-}ss]} }$$ inside $R$. We recall from the construction  that each $Q^{[\sigma\text{-}ss]}$ is an open subset of the locally closed subscheme $Q$, see \eqref{eq:iteratedinclusion}, and deduce that each $Q^{[\sigma\text{-}ss]}$ is a Zariski-open subset of $Z$. We will see that the affine scheme $Z$ is a kind of ``master space'' for our variation problem. More precisely, we have the following result.

\begin{theorem}[A master space for the ``variation of polarisation''-problem]\label{thm:masterspace}
Let $Z$ be as above, and let $\pi_A: R^{\sigma\text{-}ss} \to R^{\sigma\text{-}ss}\hq G$ be the quotient morphism. Then, for any $\sigma \in \Sigma$, we have 
\begin{equation}\label{eq:masterspaceintersection}
 Z^{\sigma\text{-}ss}:= R^{\sigma\text{-}ss} \cap Z = Q^{[\sigma\text{-}ss]},
\end{equation}
 and therefore $\pi_A (Z^{\sigma\text{-}ss}) \cong Z^{\sigma\text{-}ss} \hq G \cong \mathcal{M}^{\sigma\text{-}ss}_{\underline{L}}(\tau)$. Here, $\pi_A(Z^{\sigma\text{-}ss})$ is endowed with the natural subscheme structure induced from $R^{\sigma\text{-}ss}\hq G$.
\end{theorem}
In other words, Theorem~\ref{thm:masterspace} says that all the moduli spaces $\mathcal{M}^{\sigma\text{-}ss}_{\underline{L}}(\tau)$, $\sigma \in \Sigma$, occur as GIT-quotients of one and the same affine $G$-scheme $Z$, whose quotient is in turn induced by the quotient of the smooth affine variety $R$ by the action of $G$.

\begin{corollary}[Mumford-Thaddeus principle for moduli of $\sigma$-semistable sheaves]\label{cor:masterspace} Let $\sigma$ and $\sigma'$ be bounded positive stability parameters.  Then, the moduli spaces
 $\mathcal{M}^{\sigma\text{-}ss}_{\underline{L}}$ and $\mathcal{M}^{\underline{\sigma'}\text{-}ss}_{\underline{L}}$ are related by a finite sequence of Thaddeus-flips. 
\end{corollary}

\begin{proof} It is proved for example in \cite[Thm 1.1]{Chindris}, \cite[Thm 3.3]{Halic} or \cite{ArzhantsevHausen} that there exists a rational polyhedral cone $\mathcal{C}_G(R)$ inside the vector space of rational characters $\mathcal{X}(G) \otimes_\mathbb{Z} \mathbb{Q}$ of $G$, consisting of those characters whose associated set of semistable points in $R$ is non-empty, together with a finite fan structure reflecting the equality of the corresponding sets of semistable points. If $\chi$ is any character of $G$, let $\mathscr{L}_\chi$ be the correspondingly linearised trivial line bundle on $R$. As $Z \hookrightarrow R$ is a $G$-equivariant closed embedding into an affine $G$-variety, and since $G$ is reductive, for any non-vanishing $G$-invariant section $s \in H^0(Z, \mathscr{L}_\chi|_Z)^G$ there exists a $G$-invariant section $\bar s \in H^0(R, \mathscr{L}_\chi)^G$ such that $\bar s|_Z = s$. Consequently, if $\mathscr{L}_\chi|_Z$ is $G$-equivariantly effective, $\mathscr{L}_\chi$ is $G$-equivariantly effective. We therefore obtain a subcone $\mathcal{C}_G(Z) \subset \mathcal{C}_G(R)$ together with a potentially coarser, and hence still finite, chamber decomposition reflecting the equality of the corresponding sets of semistable points of $\mathscr{L}_\chi|_Z$. This chamber decomposition induces a chamber decomposition on the intersection $\mathcal{C}_G^\perp(Z) :=  \mathcal{C}_G(Z) \cap \chi^\perp_{\underline{d}}$. 

If $\sigma$ and $\sigma'$ are given, the corresponding rational characters $\chi_{{\sigma}}$ and $\chi_{{\sigma'}}$ belong to two of the chambers $C$ and $C'$ of $\mathcal{C}_G^\perp(Z)$, and by Theorem~\ref{thm:masterspace} above, the corresponding moduli spaces $\mathcal{M}^{\sigma\text{-}ss}_{\underline{L}}(\tau)$ and $\mathcal{M}^{\underline{\sigma'}\text{-}ss}_{\underline{L}}(\tau)$ are isomorphic to the GIT-quotients $Z^{C\text{-}ss}\hq G$ and $Z^{C'\text{-}ss}\hq G $, respectively. As the fan structure on $\mathcal{C}_G^\perp(Z)$ is finite, moving from $\chi_{{\sigma}}$ to $\chi_{{\sigma'}}$ in $\mathcal{C}_G^\perp(Z)$ can be done in finitely many steps (with respect to the fan structure). As explained in \cite[\S\S~1 \& 3]{Thaddeus}, each step induces a finite sequence of Thaddeus-flips on the corresponding GIT-quotients. This shows the claim. 
\end{proof}

\begin{remark}
The above corollary is slightly stronger than stated, in that the same master space may be used for any finite number of bounded stability parameters.  Due to the fact that the whole situation is equivariantly embedded into the (smooth) $G$-module $R$, in specific examples a finer description of the transition from one moduli space to another is possible using the results of Thaddeus \cite[Sect.~4 \& 5]{Thaddeus}. Moreover, as explained in \cite[Sect.~3.1]{Thaddeus} and \cite[Chap.~1.6]{SchmittBuch}, respectively, we may further reduce any explicit analysis to a question of variation of $(\mathbb{C}^*)^k$-, or even $\mathbb{C}^*$-GIT-quotients.
\end{remark}

\begin{proof}[Proof of Theorem~\ref{thm:masterspace}]
 Fix $\sigma\in \Sigma$. To begin, we note that $Q^{[\sigma\text{-}ss]}$ is contained in $Z^{\sigma\text{-}ss}$ by Theorem~\ref{thm:mainsemistabilitycomparison}(1). Our first claim is that $Q^{[\sigma\text{-}ss]}$ is dense in $Z^{\sigma\text{-}ss}$. As a set of semistable points with respect to a $G$-linearisation in a line bundle, the set $Z^{\sigma\text{-}ss}$ is Zariski-open in $Z$. Moreover, as we have already noted above, $Q^{[\sigma\text{-}ss]}$ is Zariski-open in $Z$. If $C$ is any component of $Z^{\sigma\text{-}ss}$, then $C \cap Q^{[\sigma\text{-}ss]}$ is either empty or otherwise a Zariski-open and dense subset of $C$. Therefore, it suffices to show that every component of $Z^{\sigma\text{-}ss}$ contains an element of $Q^{[\sigma\text{-}ss]}$.

If $C$ is any component of $Z^{\sigma\text{-}ss}$, then by the definition of $Z$ there exists some stability parameter $\sigma'\in \Sigma$ such that $Q^{[\sigma'\text{-}ss]}\cap C$ is non-empty. Let $M$ be an element in this intersection. Then, on the one hand, $M$ is GIT-semistable with respect to $\chi_\sigma$, since $M \in Z^{\sigma\text{-}ss}$, and on the other hand, $M$ is of the form $M = \Hom(T, E)$ for some uniquely determined $(p,\underline{L})$-regular sheaf $E$ by the definition of $Q^{[\sigma'\text{-}ss]}$. Owing to the choice of the natural numbers $m\gg n\gg p\gg 0$, Theorem~\ref{thm:mainsemistabilitycomparison}(1) hence implies that the sheaf $E$ is semistable with respect to $\sigma$. In other words, we have $M = \Hom(T,E) \in Q^{[\sigma\text{-}ss]}$, and $Q^{[\sigma\text{-}ss]} \cap C \neq \emptyset$.

Since the quotient map of $Z^{\sigma\text{-}ss} \to Z^{\sigma\text{-}ss}\hq G$ is just the restriction of the quotient map $\pi_A: R^{\sigma\text{-}ss} \to R^{\sigma\text{-}ss}\hq G$, it follows from Theorem~\ref{thm:mainsemistabilitycomparison}(2) that $Q^{[\sigma\text{-}ss]}$ is saturated in $Z^{\sigma\text{-}ss}$. The density property established above and the generalisation of Langton's Theorem to our setup, Theorem~\ref{thm:modulispaceproper}, now together imply that on the level of reduced spaces, we have $\mathcal{M}^{\sigma\text{-}ss}_{\underline{L}} \cong \pi_A (Q^{[\sigma\text{-}ss]}) = Z^{\sigma\text{-}ss}\hq G$. Using saturatedness we conclude that $Q^{[\sigma\text{-}ss]}$ equals $Z^{\sigma\text{-}ss}$, as claimed.
\end{proof}

\section{Gieseker-stability with respect to real ample classes}

It is a natural and old question, raised for example by Tyurin, how to compactify the moduli space of vector bundles on a compact K\"ahler manifold that are slope-stable with respect to a chosen K\"ahler class $\omega \in H^{1,1}(X, \R)$ using Giesker-semistable sheaves. In this section, we solve this problem in case $\omega$ is a class in the real span of the ample cone of a smooth projective threefold, thus providing the first higher-dimensional evidence in favour of a positive answer to this question.

\begin{definition}[Hilbert-polynomial with respect to a real ample class]
Let $X$ be a smooth projective variety, and $\tau \in B(X)_\Q$.  We let $\omega\in N^1(X)_{\mathbb R}$ be the class of a real ample divisor on $X$. Given a coherent sheaf $E$ on $X$, the \emph{Hilbert-polynomial with respect to $\omega$} is defined as
 $$P^{\omega}_E(m) := \int_X ch(E) e^{m\omega} \Todd(X),$$
and (semi-)stability is defined in the usual way using $P^{\omega}$.
\end{definition}

\begin{remark}
 In the projective case, if $\omega \in c_1(L)$ for some ample line bundle $L$ then $P^{\omega}_E(m) = P_{E}^L(m)$ is the usual Hilbert polynomial.     Thus, the above definition generalises the notion of Gieseker-stability to all real ample classes.
 \end{remark}

For the main result of this section, we assume $X$ to be a smooth projective threefold over an algebraically closed field of characteristic zero.  Our goal is to show that Gieseker-stability with respect to $\omega$ is equivalent to multi-Gieseker-stability with respect to some  $(\underline{ L},\sigma)$, where $\sigma$ is possibly irrational. We start with an elementary observation.

\begin{lemma}\label{lem:elementaryconvex}
For any $\tau,\theta\in\R_{>0}$ there exist $\lambda\in\Q_{>0}$ and 
 $\sigma,\sigma'\in\R_{>0}$ such that
\[\sigma+\sigma'\lambda=\tau \;\;\;\;\text{ and }\;\;\;\;\sigma+\sigma'\lambda^2=\theta.
\]
If moreover $\tau\neq\theta$, then $\lambda$ may be taken in $\Q_{>0}\setminus\{1\}$.
\end{lemma}
\begin{proof} 
The fact that for any $\tau,\theta\in\R_{>0}$ there exist  $\lambda\in\Q_{>0}$ and 
 $\sigma,\sigma'\in\R_{\ge0}$  such that
$\sigma+\sigma'\lambda=\tau $ and $\sigma+\sigma'\lambda^2=\theta$
 may be reformulated as the following equality
$$\R_{>0}\times\R_{>0}=\cup_{\lambda\in\Q_{>0}}\Conv(\R_{>0}(1,1)\cup\R_{>0}(1,\lambda)),$$
which is clearly true. 
The supplementary assertion on $\lambda$ and the fact that $\sigma$ and $\sigma'$ may be even taken in $\R_{>0}$ are easily checked.
\end{proof}

\begin{proposition}\label{prop:convex}
Let $X$ be a smooth projective threefold  and $\omega$ a  class in the ample cone $\Amp(X)_\R$. Let $\rho$ be the Picard number of $X$ and set $j_0=4(\rho+1)$. Then, there exist rational ample classes $L_j\in\Amp(X)_\Q$, $0\le j\le j_0$, 
 such that the set
\[ \mathcal{U}_\omega(L_1, \dots, L_{j_0}) :=
\Bigl\{ (\sum\nolimits_{j=1}^{j_0} \sigma_j L_j ,
 \sum\nolimits_{j=1}^{j_0} \sigma_j L_j^2) \ \bigl| \ \sigma_j\in\R_{>0}, \ 0\le j\le j_0 \Bigr\}
\]
is an open neighbourhood of $(\omega,\omega^2)$ in $\Amp(X)_\R\times\Pos(X)_\R$.
\end{proposition}
\begin{proof}
It is clear that we can find rational classes $L_1,\dots,L_{\rho+1}\in\Amp(X)_\Q$ containing $\omega$ in their convex hull. We suppose moreover that $L_1,\dots,L_\rho$ span $N^1(X)_{\R}$. By \cite[Prop.~6.5]{GrebToma} the map $p_2:\Amp(X)_\R\to\Pos(X)_\R$, $p_2(\alpha)=\alpha^2$ is a homeomorphism. One can thus find further rational classes $L_{\rho+2},\dots,L_{2(\rho+1)}\in\Amp(X)_\Q$ such that $\omega^2$ may be expressed as a convex combination of $L_{\rho+2}^2,\dots,L_{2(\rho+1)}^2$. In other words, we may write 
\begin{equation}\label{eq:firstrepresentation}
\omega =\sum\nolimits_{j=1}^{2(\rho+1)} \tau_j L_j, \;\;\;
 \omega^2 =\sum\nolimits_{j=1}^{2(\rho+1)} \theta_j L_j^2
\end{equation}for suitable $\tau_j, \theta_j \in \mathbb{R}_{\geq 0}$. Moreover, by a small perturbation of the line bundles involved, we may assume that $\tau_j, \theta_j \in (0,1)$ and that $\tau_j\neq \theta_j$
 for all $j$.  Lemma \ref{lem:elementaryconvex} then provides us   
 with 
positive rational numbers 
$\lambda_j\neq1$, $j=1, \dots, 2(\rho + 1) $,  and positive reals  $\sigma_j, \sigma'_j$,  $j=1, \dots, 2(\rho + 1) $, satisfying 
 \begin{equation}\label{eq:secondrepresentation}
\sigma_j+\sigma'_j\lambda_j=\tau_j
\;\;\;\text{ and }\;\;\;
\sigma_j+\sigma'_j\lambda_j^2=\theta_j.
 \end{equation}
Thus, if for $2(\rho+1)+1\le j\le 4(\rho+1) =j_0$ we set  
\begin{align*}       
L_j&:=\lambda_{j-2(\rho+1)}L_{j-2(\rho+1)}\\ \sigma_j&:=\sigma'_{j-2(\rho+1)}                                  
\end{align*} we get the expressions
$\omega =\sum_{j=1}^{j_0} \sigma_j L_j,
 \omega^2 =\sum_{j=1}^{j_0} \sigma_j L_j^2$ from \eqref{eq:firstrepresentation} and \eqref{eq:secondrepresentation}.
Moreover, by choosing the directions of the $L_j$-s in $N^1(X)_{\R}$ close enough to that of $\omega$ we make sure that 
\[
\Bigl\{ (\sum\nolimits_{j=1}^{j_0} \sigma_j L_j ,
 \sum\nolimits_{j=1}^{j_0} \sigma_j L_j^2) \ \bigl| \ \sigma_j\in\R_{>0}, \ 0\le j\le j_0 \Bigr\}
\]
is contained in $\Amp(X)_\R\times\Pos(X)_\R$.
It remains to check that this set is open. This will be clear once we have shown that the $\R$-linear map 
$\phi:\R^{4(\rho+1)}\to N^1(X)_{\R}\times N_1(X)_{\R}$, $\sigma\mapsto (\sum\nolimits_{j=1}^{j_0} \sigma_j L_j ,
 \sum\nolimits_{j=1}^{j_0} \sigma_j L_j^2)$ has maximal rank. But if $(e_j)_j$ denotes the canonical basis of $\R^{4(\rho+1)}$, we have $\phi(e_1)=(L_1,L_1^2)$, $\phi(  e_{2(\rho+1)+1})=(\lambda_1 L_1,\lambda_1^2L_1^2)$ and these two vectors span the same subspace of $ N^1(X)\times N_1(X)$ as $(L_1,0)$ and $(0,L_1^2)$, since $\lambda_1\neq 1$. Working in the same way on the pairs $(e_j, e_{2(\rho+1)+j})$, $2\le j\le\rho$, we prove the surjectivity of $\phi$.
\end{proof}
    
 \begin{remark}\label{rmk:j_0}
It is possible to lower the number $j_0$ provided by the Proposition to $1+4(\rho-1)$, if we only ask for proportionality relations 
 $\omega =\alpha_2\sum_{j=1}^{j_0} \sigma_j L_j$ and $ \omega^2 =\alpha_1\sum_{j=1}^{j_0} \sigma_j L_j^2$ 
 for some positive constants $\alpha_i$, which clearly does not affect stability.  Moreover, in the special case $\rho=2$ it is easy to see that we may improve our argument to get $j_0=3$.
  \end{remark}
 \begin{theorem}[Projective moduli spaces for $\omega$-semistable sheaves]\label{thm:kaehlermodulithreefolds}
  Let $\omega\in N^1(X)_{\mathbb R}$ be a real ample class on a smooth projective threefold $X$ and $\tau \in B(X)_\mathbb{Q}$.  Then, there exists a projective moduli space $\mathcal M_{\omega}$ of torsion-free sheaves of topological type $\tau$ that are semistable with respect to $\omega$.  This moduli space contains an open set consisting of points representing isomorphism classes of stable sheaves, while points on the boundary correspond to $S$-equivalence classes of strictly semistable sheaves.  
\end{theorem}
\begin{proof}
 Let $\sigma_j$, $1 \leq j \leq j_0$, be as in Proposition~\ref{prop:convex} and $\underline{L} := (L_1, \dots, L_{j_0})$. Simultaneously scale the rational classes $L_j$ by a positive integer such that all of them become integral. Then, by construction, for any sheaf on $X$ the Hilbert polynomial with respect to $\omega$ is, up to a positive integer factor, the same as the multi-Hilbert polynomial with respect to $(\underline{L}, \sigma)$, and hence semistability with respect to $\omega$ is equivalent to semistability with respect to $(\underline{L}, \sigma)$. 

Moreover, note that the whole first quadrant $\Sigma =  (\mathbb R_{\ge 0})^{j_0}\setminus\{0\}$ is a bounded set of stability parameters by Corollary~\ref{cor:boundednesssurfaceorpicard2}. Therefore, we may take $p$ so that all sheaves $E$ of a given topological type that are semistable with respect to some $\sigma\in \Sigma$ are $(p,\underline{L})$-regular.  Consequently, it follows from Corollary \ref{cor:irrational} (applied with this value of $p$) that there exists  $\sigma'\in (\mathbb Q_{\ge 0})^{j_0}\setminus\{0\}$ such that any $\sigma$-semistable torsion-free sheaf of topological type $\tau$ is $\sigma'$-semistable and vice versa, and similarly for $S$-equivalence classes. Hence, the desired moduli space is provided by $\mathcal{M}^{\sigma\text{-}ss}_{\underline{L}}(\tau)$.
\end{proof}

\section{Variation of Gieseker moduli spaces on threefolds}

Connecting Proposition~\ref{prop:convex} with the general results concerning variation of multi-Gieseker moduli spaces obtained in Section~\ref{subsect:MultiGiesekerVariation}, we obtain one of the main results of our paper:
\begin{theorem}[Variation of moduli spaces on threefolds]\label{thm:MTprincipleGieseker3}
Let $X$ be a smooth  projective threefold over an algebraically closed field of characteristic zero, let $\tau \in B(X)_\mathbb{Q}$, and let  $L_1, L_2$ be two classes in $\Amp(X)_\R$. Then, the moduli spaces $\mathcal{M}_{L_1}$ and $\mathcal{M}_{L_2}$ of sheaves of topological type $\tau$ that are Gieseker-semistable with respect to $L_1$ and $L_2$, respectively, are related by a finite number of Thaddeus-flips.
\end{theorem}

\begin{proof}
Let $\omega \in \Amp(X)_\R$ be any real ample class, let $H_1, \dots, H_{j_0}$ be the set of ample classes guaranteed by Proposition~\ref{prop:convex}, with corresponding open subset $\mathcal{U}_\omega(H_1, \dots, H_{j_0})\subset \Amp(X)_\R \times \Pos(X)_\R$. Let $U_\omega$ be the preimage of $\mathcal{U}_\omega(H_1, \dots, H_{j_0})$ under the map from $\Amp(X)_\R$ to $\Amp(X)_\R \times \Pos(X)_\R$ given by $\alpha \mapsto (\alpha, \alpha^2)$. Clearly, $U_\omega$ is an open neighbourhood of $\omega$ in $\Amp(X)_\R$. Then, as the corresponding multi-Gieseker-stability conditions are all positive by construction, it follows from Corollary~\ref{cor:masterspace} that the Gieseker moduli spaces associated with any two rational points in $U_\omega$ are related by a finite number of Thaddeus-flips. In fact, the argument in the second paragraph of Theorem~\ref{thm:kaehlermodulithreefolds} can now be applied again to show that the same statement holds for any two points in $U_\omega$. Covering a connecting segment between $L_1$ and $L_2$ in $\Amp(X)_\R$ with finitely many open subsets of the form $U_\omega$, we conclude that the moduli spaces $\mathcal{M}_{L_1}$ and $\mathcal{M}_{L_2}$ associated with any two rational classes $L_1$ and $L_2$ are related by a finite number of Thaddeus-flips. 
\end{proof}

\vspace{0.5cm}

\addtocontents{toc}{\protect\setcounter{tocdepth}{0}}

\end{document}